\documentclass[reqno,11pt]{amsart}

\usepackage[top=2.0cm,bottom=2.0cm,left=3cm,right=3cm]{geometry}
\usepackage{amsthm,amsmath,amssymb,dsfont}
\usepackage{mathrsfs,amsfonts,functan,extarrows,mathtools}
\usepackage[colorlinks]{hyperref}

\usepackage{marginnote}
\usepackage{xcolor}
\usepackage{stmaryrd}
\usepackage{esint}
\usepackage{graphicx}
\usepackage{bm}

\newtheorem{theorem}{Theorem}[section]

\newtheorem{remark}{Remark}[section]
\newtheorem{lemma}{Lemma}[section]

\numberwithin{equation}{section}
\allowdisplaybreaks

\def\d{\mathrm{d}}
\def\no{\nonumber}
\def\R{\mathbb{R}}
\def\T{\mathbb{T}}

\def\eps{\epsilon}

\def\l{\langle}
\def\r{\rangle}

\newcounter{wronumber}

\begin{document}
\title[From VPB to incompressible NSFP]
			{From Vlasov-Poisson-Boltzmann system to incompressible Navier-Stokes-Fourier-Poisson system: convergence for classical solutions}

\author[Mengmeng Guo]{Mengmeng Guo}
\address[Mengmeng Guo]
        {\newline School of Mathematics and Statistics, Wuhan University, Wuhan, 430072, P. R. China}
\email{guomengmeng66@whu.edu.cn}
\author[Ning Jiang]{Ning Jiang}
\address[Ning Jiang]{\newline School of Mathematics and Statistics, Wuhan University, Wuhan, 430072, P. R. China}
\email{njiang@whu.edu.cn}

\author[Yi-Long Luo]{Yi-Long Luo}
\address[Yi-Long Luo]
{\newline School of Mathematics and Statistics, Wuhan University, Wuhan, 430072, P. R. China}
\email{yl-luo@whu.edu.cn}

\thanks{ June 1, 2019}

\maketitle

\begin{abstract}
	For the one-species Vlasov-Poisson-Boltzmann (VPB) system in the scaling under which the moments of the fluctuations formally converge to the incompressible Navier-Stokes-Fourier-Poisson (NSFP) system, we prove the uniform estimates with respect to the Knudsen number $\eps$ for the fluctuations. As a consequence, the existence of the global-in-time classical solutions of VPB with all $\eps \in (0,1]$ is established in whole space under small size of initial data, and the convergence to incompressible NSFP as $\eps$ go to 0 is rigorously justified. \\
	
	
\end{abstract}

%\vspace*{10pt}

%\phantomsection
%\addcontentsline{toc}{section}{\contentsname}

\tableofcontents

%%%%%%%%%%%%%%%%%%%%%%%%%%%%%%%%%%%%%%（正文）%%%%%%%%%%%%%%%%%%%%%%%%%%%%%
%%%%%%%%%%%%%%%%%%%%%%%%%%%%%%%%%%%%%%%%%%%%%%%%%%%%%%%%%%%%%%%%%%%%%%%%%%%

\section{Introduction}

\subsection{One-species Vlasov-Poisson-Boltzmann system}
In this paper, we study the following scaled Vlasov-Poisson-Boltzmann (briefly, VPB) system (see Ars\'enio-Saint-Raymond's book \cite{Arsenio-SaintRaymond-2016})
\begin{align}
  \epsilon \partial _t f_\epsilon + v \cdot \nabla_x f_\epsilon + \eps \gamma \nabla_x \phi_\eps \cdot \nabla_v f_\epsilon = \tfrac{1}{\epsilon} B ( f_\epsilon, f_\epsilon) \,, & \label{Equ-Kinetic-f} \\
  \Delta_x \phi_\epsilon = \tfrac{\gamma}{\epsilon} \Big( \int_{\mathbb{R}^3} f_\epsilon dv - 1 \Big) \,, & \label{Equ-Poisson-phi}
\end{align}
where $\epsilon > 0$ denotes the Knudsen number, which is the ratio of the mean free path to the macroscopic length scale, and $\gamma > 0$ is the ratio of the electric repulsion according to Gauss's law and the Knudsen number. Here the unknown $f_\epsilon=f_\epsilon(t,x,v) \geq 0$ is a nonnegative function standing for the number density of gas particles which have velocity $v \in \R^3$ at position
$x \in \mathbb{R}^3$ at time $t>0$. The self-consistent electric potential $\phi_\epsilon = \phi_\epsilon(t,x) \in \R$ is coupled with the distribution function $f_\epsilon(t,x,v)$ through the Poisson equation. The bilinear function $B$ with hard-sphere interaction is defined by
\begin{equation}\label{1.10}
  B(f,f) = \int_{\mathbb{R}^3 \times \mathbb{S}^2} ( f^\prime f_1^\prime - f f_1 ) | ( v - v_1 ) \cdot \omega | \d \omega \d v_1 \,,
\end{equation}
where
\begin{equation*}
  f = f(t,x,v) \,, \quad f^\prime = f ( t , x , v^\prime ) \,, \quad f_1 = f ( t , x , v_1 ) , \quad f^\prime_1 = f ( t , x , v_1^\prime ) \,,
\end{equation*}
\begin{equation*}
  v^\prime = v - [ ( v - v_1 ) \cdot \omega ] \omega \,, \quad v_1^\prime = v_1 + [ ( v - v_1 ) \cdot \omega ] \omega \,, \quad \omega \in \mathbb{S}^2 \,.
\end{equation*}
This hypothesis is satisfied for all physical model and is more convenient to work with but do not impede the generality of our results. The system \eqref{Equ-Kinetic-f}-\eqref{Equ-Poisson-phi} describes the evolution of a gas of one species of charged particles (ions or electrons) subject to an auto-induced electrostatic force. We then impose on the initial data of \eqref{Equ-Kinetic-f}-\eqref{Equ-Poisson-phi}
\begin{equation}
  \begin{aligned}
    f_\eps (0, x , v) = f_{\eps, 0} (x, v) \geq 0 \,.
  \end{aligned}
\end{equation}
A function $\phi_{\eps , 0} (x) \in \R$ can thereby be introduced by
\begin{equation}\label{phi0-1}
  \begin{aligned}
    \Delta_x \phi_{\eps, 0} = \tfrac{\gamma}{\eps} \Big( \int_{\R^3} f_{\eps, 0} \d v - 1 \Big) \,,
  \end{aligned}
\end{equation}
which is consequently the initial value of $\phi_\eps$.

A physically relevant requirement for solutions to the VPB system are their mass, momentum and energy are preserved with time. This is also an {\em a priori} property of the VPB system on the torus, which reads
\begin{equation}\label{Consevtn-law-f}
  \begin{aligned}
    & \tfrac{\d}{\d t} \int_{\R^3 \times \R^3} f_\eps (t,x,v) \d v \d x = 0 \,, \ \tfrac{\d}{\d t} \int_{\R^3 \times \R^3} v f_\eps (t,x,v) \d v \d x = 0 \,, \\
    & \tfrac{\d}{\d t} \Big( \int_{\R^3 \times \R^3} |v|^2 f_\eps (t,x,v) \d v \d x + \int_{\R^3} | \nabla_x \phi_\eps (t,x) |^2 \d x \Big) = 0 \,.
  \end{aligned}
\end{equation}

It is well-known that the global equilibrium for the VPB system is the normalized global {\em Maxwellian} $M(v)$ defined as
\begin{equation*}
  M (v) = \tfrac{1}{( 2 \pi)^{\frac{3}{2}}} \, e^{-\frac{|v|^2}{2}} \,.
\end{equation*}
Our goal in current paper is to justify the incompressible Navier-Stokes-Fourier-Poisson (briefly, NSFP) limit of \eqref{Equ-Kinetic-f}-\eqref{Equ-Poisson-phi}. More precisely, we consider the fluctuation around the global Maxwellian with size $\epsilon$ (this is called Navier-Stokes scaling, see \cite{BGL1}),
$$f_\eps (t,x,v) = M( 1 + \eps g_\eps (t,x,v) )\,.$$
We are interested in justifying the limit of the fluctuation $g_\epsilon$ as the Knudsen number $\eps$ tends to zero. This leads to the perturbed VPB system
\begin{equation}\label{VPB-g}
  \left\{
    \begin{array}{l}
      \epsilon \partial_t g_\epsilon + v \cdot \nabla_x g_\epsilon + \tfrac{1}{\epsilon} L g_\epsilon - \gamma v \cdot \nabla_x \phi_\epsilon - \gamma \epsilon v \cdot \nabla_x \phi_\epsilon g_\epsilon
      +\gamma \epsilon \nabla_x \phi_\epsilon \cdot \nabla_v g_\epsilon = Q ( g_\epsilon , g_\epsilon ) \,, \\
      \Delta_x \phi_\epsilon = \gamma \int_{\R^3} g_\epsilon M \d v  \,,
    \end{array}
  \right.
\end{equation}
where the bilinear symmetric operator $Q$ is
\begin{equation}\label{Colli-Q}
  Q(g,h) = \tfrac{1}{2 M} \big[ B (Mg , Mh) + B (Mh , Mg) \big] \,,
\end{equation}
and the linearized Boltzmann collisional operator $L$ takes the form
\begin{equation}\label{Def-L}
  Lg = - 2 Q (g, 1) = - \tfrac{1}{M} \big\{ B ( M ,M g ) + B ( M g , M) \big\} \,.
\end{equation}
Moreover, the collisional frequency for the hard sphere interaction can be defined as (see \cite{BGL1} for instance)
\begin{equation}\label{Colli-Frequence-nu}
  \begin{aligned}
    \nu (v) = \int_{\R^3} |v-v_1| M (v_1) \d v_1 \,.
  \end{aligned}
\end{equation}
Furthermore, the linearized Boltzmann collisional operator $L$ can be decomposed as
\begin{align*}
  L g = \nu g + K g \,,
\end{align*}
where
\begin{equation*}
  \begin{aligned}
    K g(v) = - \iint_{\mathbb R^3 \times \mathbb {S}^2} | ( v-v_1 ) \cdot \omega | ( g^\prime + g_1^\prime - g_1 ) M_1 \d x \d v = \int_{\mathbb {R}^3} \mathcal{K}(v,v_1) g (v_1) \d v_1
  \end{aligned}
\end{equation*}
is a self-adjoint compact operator on $L^2(Mdv)$ with a real symmetric integral kernel $\mathcal{K}(v,v_1)$ (see \cite{Levermore-Sun-2010-KRM} for instance). The null space of the operator $L$ is the five-dimensional space spanned by the collision invariants \cite{BGL1}
\begin{equation}\label{invariant}
  \mathcal{N} = \textrm{Ker} L = \textrm{Span} \{ 1 , v_1 , v_2 , v_3 , |v|^2 \} \,.
\end{equation}
Notice that the projection operator $L$ is nonnegative. We also denote $P$ by the projection from $L^2(Mdv)$ to $\mathcal{N}$. In what follows, we write $P g_\epsilon$ in the form of coordinates
\begin{equation}\label{Proj-Hydrodynamic}
  P g_\epsilon=\{ a_{\eps} (t,x) + b_{\eps} (t,x) \cdot v + c_{\eps} (t,x) |v|^2 \} \,,
\end{equation}
where $a_\eps,\ b_\eps=(b_{\eps, 1} , b_{\eps, 2}, b_{\eps, 3})$ and $c_\eps$ are the coefficients of the macroscopic components $P g_\eps$.

Furthermore, without loss of generality, the initial data $f_{\eps, 0} (x,v)$ of \eqref{Equ-Kinetic-f}-\eqref{Equ-Poisson-phi} can be given the form of
\begin{equation}\label{IC-near-Equilibrium}
  \begin{aligned}
    f_{\eps, 0} (x,v) = M(v) \big[ 1 + \eps g_{\eps, 0} (x,v) \big] \geq 0 \,,
  \end{aligned}
\end{equation}
where $g_{\eps,0} (x,v) \in \R$ is the initial data of the perturbed system \eqref{VPB-g}, namely,
\begin{equation}\label{IC-g}
\begin{aligned}
g_\eps (0,x,v) = g_{\eps, 0} (x,v) \,,
\end{aligned}
\end{equation}
Moreover, the function $\phi_{\eps, 0}$ given in \eqref{phi0-1} is equivalently defined by
\begin{equation}
   \begin{aligned}
     \Delta_x \phi_{\eps, 0} (x) = \gamma \int_{\R^3} g_{\eps, 0} (x,v) M(v) \d v \,.
   \end{aligned}
\end{equation}

\subsection{Notations and main results}

Through this paper, we use $C$ to represent some generic positive constant (generally large) while $\lambda$ is used to denote some generic positive constant (generally small), where both $C$ and $\lambda$ may take different values at different places. The symbol $C(\cdot)$ denotes that constants depend on some parameters in the argument. In addition, $A \lesssim B$ means that there is a generic constant $C>0$ such that $A\leq C B$. Furthermore, if $C_1 A \leq B \leq C_2 A$ holds for some generic constants $C_1, C_2 > 0$, we denote by $A \thicksim B$.

We index the $L^p$ spaces by the name of the concerned variable. Namely,
\begin{equation*}
  \begin{aligned}
    L^p_x = L^p (\d x) \,,  \ L^p_v (w) = L^p (w M \d v) \,, \ L^\infty_x = L^\infty_x (\d x)
  \end{aligned}
\end{equation*}
with the norms
\begin{equation*}
  \begin{aligned}
    & \| \phi \|_{L^p_x} = \Big( \int_{\R^3} | \phi (x) |^p \d x \Big)^\frac{1}{p} < \infty \,, \ \| g \|_{L^p_v (w)} = \Big( \int_{\R^3} | g (v) |^p w M \d v \Big)^\frac{1}{p} < \infty \,, \\
    & \| \phi \|_{L^\infty_x} = \textrm{ess} \sup_{x \in \R^3} | \phi (x) | < \infty \,,
  \end{aligned}
\end{equation*}
respectively, where $w = 1$ or $\nu$. If $w = 1$, we denote by $L^p_v = L^p_v (1) $. Next we introduce $L^p_x L^q_v (w)$ space for $p,q \in [1, \infty]$ endowed with the norms
\begin{equation*}
  \begin{aligned}
    \| g \|_{L^p_x L^q_v (w)} =
    \left\{
      \begin{array}{l}
        \Big( \int_{\R^3} \| g (x, \cdot) \|^p_{L^q_v (w)} \d x \Big)^\frac{1}{p}  \qquad\qquad\qquad p,q \in [ 1, \infty ) \,, \\
        \textrm{ess}\sup_{x \in \R^3} \| g (x,\cdot) \|_{L^q_v (w)}  \qquad\qquad\quad\ \  p = \infty \,, q \in [ 1, \infty ) \,, \\
        \Big( \int_{\R^3} \big| \textrm{ess}\sup_{v \in \R^3} |g (x,v)| \big|^p \d x \Big)^\frac{1}{p}  \quad p \in [ 1, \infty ) \,, q = \infty \,, \\
        \textrm{ess}\sup_{(x,v) \in \R^3 \times \R^3} |g (x,v)| \,, \qquad\quad \ p = q = \infty \,.
      \end{array}
    \right.
  \end{aligned}
\end{equation*}
If $w=1$, we also denote by $L^p_x L^q_v = L^p_x L^q_v (1)$. Furthermore, if $p = q$, the notation $L^p_{x,v} (w)$ means $L^p_x L^p_v (w)$. For $p=2$, we use $ \langle \cdot , \cdot \rangle_{L^2_{x,v}} $, $\langle \cdot , \cdot \rangle_{L^2_v}$ and $\langle \cdot , \cdot \rangle_{L^2_x}$ to denote the inner product in the Hilbert spaces $L^2_{x,v}$, $L^2_v$ and $L^2_x$, respectively.

The multi-index $\alpha=[\alpha_1,\alpha_2,\alpha]$ and $\beta=[\beta_1,\beta_2,\beta_3]$ in $\mathbb{N}^3$ will be used to record spatial and velocity derivatives, respectively. We denote the $(\alpha, \beta)^{th}$ partial derivative by
 $$ \partial_x^\alpha \partial_v^\beta = \partial_{x_1}^{\alpha_1} \partial_{x_2}^{\alpha_2} \partial_{x_3}^{\alpha_3} \partial_{v_1}^{\beta_1} \partial_{v_2}^{\beta_2} \partial_{v_3}^{\beta_3} \,.$$
If each component of $\alpha \in \mathbb{N}^3$ is not greater than that of $\tilde{\alpha}$'s, we denote by $\alpha \leq \tilde{\alpha}$. The symbol $\alpha < \tilde{\alpha}$ means $\alpha \leq \tilde{\alpha}$ and $|\alpha| < |\tilde{\alpha}|$, where $| \alpha | = \alpha_1 + \alpha_2 + \alpha_3$. We now introduce the Sobolev spaces $H^N_x$, $H^N_x L^2_v$, $H^N_x L^2_v (\nu)$, $H^N_{x,v}$ and $H^N_{x,v} (\nu)$ endowed with the norms
\begin{equation*}
  \begin{aligned}
    & \| \phi \|_{H^N_x} = \Big( \sum_{|\alpha| \leq N} \| \partial^\alpha_x \phi \|_{L^2_x}^2 \Big)^\frac{1}{2} \,, \ \| g \|_{H^N_x L^2_v} = \Big( \sum_{|\alpha| \leq N} \| \partial^\alpha_x g \|^2_{L^2_{x,v}} \Big)^\frac{1}{2} \,, \\
    & \| g \|_{H^N_x L^2_v (\nu)} = \Big( \sum_{|\alpha| \leq N} \| \partial^\alpha_x g \|^2_{L^2_{x,v}(\nu)} \Big)^\frac{1}{2} \,, \ \| g \|_{H^N_{x,v}} = \Big( \sum_{|\alpha| + |\beta| \leq N} \| \partial^\alpha_x \partial^\beta_v g \|^2_{L^2_{x,v}} \Big)^\frac{1}{2} \,, \\
    & \| g \|_{H^N_{x,v}(\nu)} = \Big( \sum_{|\alpha| + |\beta| \leq N} \| \partial^\alpha_x \partial^\beta_v g \|^2_{L^2_{x,v}(\nu)} \Big)^\frac{1}{2} \,,
  \end{aligned}
\end{equation*}
respectively. Furthermore, we give the spaces $\widetilde{H}^N_{x,v}$ and $\widetilde{H}^N_{x,v} (\nu)$ endowed with the norms
\begin{equation*}
  \begin{aligned}
    \| g \|_{\widetilde{H}^N_{x,v}} = \Bigg( \sum_{|\alpha| + |\beta| \leq N \,, \beta \neq 0} \| \partial^\alpha_x \partial^\beta_v g \|^2_{L^2_{x,v}} \Bigg)^\frac{1}{2} \,, \ \| g \|_{\widetilde{H}^N_{x,v} (\nu)} = \Bigg( \sum_{|\alpha| + |\beta| \leq N \,, \beta \neq 0} \| \partial^\alpha_x \partial^\beta_v g \|^2_{L^2_{x,v}(\nu)} \Bigg)^\frac{1}{2} \,.
  \end{aligned}
\end{equation*}
For convenience to state our main results, we will define the energy functional
\begin{equation}\label{Energy-E}
  \begin{aligned}
    \mathcal{E}_N ( g, \phi ) = \| g \|^2_{H^N_x L^2_v} + \| \nabla_x \phi \|^2_{H^N_x} + \| ( I - P ) g \|^2_{H^N_{x,v}}
  \end{aligned}
\end{equation}
and the energy dissipative rate functional
\begin{equation}\label{Energy-D}
  \begin{aligned}
    \mathcal{D}_{N,\eps} ( g ) = \tfrac{1}{\eps^2} \| (I-P) g \|^2_{H^N_{x,v}(\nu)} + \| \nabla_x P g \|^2_{H^{N-1}_x L^2_v} + \| \l g , 1 \r_{L^2_v} \|^2_{H^{N-1}_x} \,.
  \end{aligned}
\end{equation}

There are two main theorems built in current paper. The first theorem is on the global existence of the Vlasov-Poisson-Boltzmann system with respect to the Knudsen number $\epsilon$, and the second is on the incompressible Navier-Stoker-Fourier-Poisson system limit $\epsilon\rightarrow 0$ taken in the solutions $g_{\epsilon}$ of the VPB system \eqref{Equ-Kinetic-f}-\eqref{Equ-Poisson-phi} which is constructed in the first theorem.

\begin{theorem}\label{Thm-global}
	Let integer $N \geq 3$ and $ 0 < \epsilon < 1 $. Then there exists $\ell_0>0$, independent of $\epsilon$, such that if $\mathcal{E}_N (g_{\eps, 0} , \phi_{\eps, 0}) \leq \ell_0$, then the Cauchy problem of \eqref{VPB-g}-\eqref{IC-g} admits a global solution $( g_\eps (t,x,v) , \phi_\eps (t,x) )$ satisfying $g_\eps \in H^N_{x,v}$, $\nabla_x \phi_\eps \in H^N_x$, and the global estimate
	\begin{equation}\label{Uniform-Global-Bnd}
	  \begin{aligned}
	    \sup_{t \geq 0} \mathcal{E}_N (g_\eps (t), \phi_\eps (t)) + c_0 \int_0^\infty \mathcal{D}_{N, \eps} (g_\eps (t)) \d t \lesssim \mathcal{E}_N ( g_{\eps, 0} , \phi_{\eps, 0} ) \,,
	  \end{aligned}
	\end{equation}
	where $c_0 > 0$ is independent of $\eps$. Moreover, $\phi_\eps (t,x)$ and
	\begin{equation*}
	  f(t,x,v) = M (v) ( 1 + \epsilon g_\eps (t,x,v)) \geq 0
	\end{equation*}
	obeys the Cauchy problem \eqref{Equ-Kinetic-f}-\eqref{Equ-Poisson-phi}-\eqref{IC-near-Equilibrium}.
\end{theorem}

The next theorem is about the limit to the incompressible NSFP system
\begin{equation}\label{NSFP}
  \begin{cases}
    & \partial_t u + u \cdot \nabla_x u + \nabla_x p = \mu \Delta_x u + \rho \nabla_x \theta \,,\\
    & \nabla_x \cdot u = 0 \,,\\
    & \partial_t ( \tfrac{3}{5} \theta - \tfrac{2}{5} \rho ) + u \cdot \nabla_x ( \tfrac{3}{5} \theta - \tfrac{2}{5} \rho ) = \kappa \Delta_x \theta \,,\\
    & \Delta_x ( \rho + \theta ) = \gamma^2 \rho \,,
  \end{cases}
\end{equation}
where the viscosity $\mu$ and heat conductivity $\kappa$ are given by
\begin{equation}\label{mu-kappa}
  \mu = \tfrac{1}{15} \langle A_{ij}, \widehat{A}_{ij} \rangle_{L^2_v} \,, \quad \kappa = \tfrac{2}{15}
\langle B_i,\widehat{B}_i \rangle_{L^2_v} \,.
\end{equation}
Here
\begin{equation*}
  A(v) = v \otimes v -\tfrac{|v|^2}{3}I,\quad\quad B(v) = (\tfrac{|v|^2}{2} - \tfrac{5}{2})v.
\end{equation*}
More precisely, it can be shown that \cite{BGL1}, in general, the linearized Boltzmann operator $L$ is self-adjoint and its range is exactly the orthogonal complement of its kernel. It follows that $A\in L^2(Mdv)$ and $B \in L^2(Mdv)$ belong to the range of $L$. Thus, there are inverses $\widehat{A}\in L^2(Mdv)$ and $\widehat{B}\in L^2(Mdv)$ such that
\begin{equation*}
  A  =  L  \widehat{A} \,\quad \text{and} \quad B = L \widehat{B} \,,
\end{equation*}
which can be uniquely determined by the fact that they are orthogonal to the kernel of $L$.

\begin{theorem}\label{Thm-Limit}
	Under the same assumptions as in the Theorem \ref{Thm-global}, let $ 0 < \epsilon \leq 1$, $ N \geq 3$ and $\ell_0 > 0$ be mentioned in the Theorem \ref{Thm-global}. We further assume that there are functions $\rho_0 (x)$, $u_0 (x)$ and $\theta_0 (x)$ in $H^N_x$ such that
	\begin{equation}
	  \begin{aligned}
	    g_{\eps, 0} (x,v) \rightarrow g_0 (x,v) = \rho_0 (x) + u_0 (x) \cdot v + \theta_0 (x) ( \tfrac{|v|^2}{2} - \tfrac{3}{2} )
	  \end{aligned}
	\end{equation}
	strongly in $H^N_{x,v}$ as $\eps \rightarrow 0$. Let $( g_\epsilon (t,x,t), \phi_\eps (t,x) )$ be the family of solutions to the VPB systems \eqref{VPB-g} constructed in Theorem \ref{Thm-global}. Then,
	\begin{equation}
	  g_\epsilon (t,x,v) \rightarrow \rho (t,x) + u (t,x) \cdot v + \theta (t,x) ( \tfrac{|v|^2}{2} - \tfrac{3}{2} )
	\end{equation}
	weakly-$\star$ for $t \geq 0$, strongly in $H^{N-1}_{x,v} $, and weakly in $H^N_{x,v}$ as $\eps \rightarrow 0$. Here $ ( \rho, u, \theta ) \in L^\infty (\R^+; H^N_x) $ with $ (u, \tfrac{3}{5} \theta - \tfrac{2}{5} \rho ) \in C(\R^+; H^{N-1}_x) $ is the solution of the incompressible NSFP system \eqref{NSFP} with initial data:
	\begin{equation}
	  u|_{t=0} = \mathcal{P} u_0(x) \,, \quad ( \tfrac{3}{5} \theta - \tfrac{2}{5} \rho ) |_{t=0} = \tfrac{3}{5} \theta_0 (x) - \tfrac{2}{5} \rho_0 (x) \,,
	\end{equation}
	where $\mathcal{P}$ is the Leray projection. Furthermore, the convergence of the moments holds:
	\begin{equation}
	  \begin{split}
	    & \mathcal{P} \langle g_\epsilon , v \rangle_{L^2_v} \rightarrow u \,,\\
	    & \langle g_\epsilon , ( \tfrac{|v|^2}{5} - 1 ) \rangle_{L^2_v} \rightarrow \tfrac{3}{5} \theta - \tfrac{2}{5} \rho \,,
	  \end{split}
	\end{equation}
	strongly in $C(\R^+; H^{N-1}_x)$, weakly-$\star$ in $t \geq 0$ and weakly in $H^N_x$ as $\eps \rightarrow 0$.
\end{theorem}

\subsection{Difficulties and ideas}

The key point of current paper is to derive a global-in-time energy bound to the perturbed VPB system \eqref{VPB-g} uniformly $\eps \in (0, 1]$ under small size of the initial data. Similar to the incompressible Navier-Stokes-Fourier limit from the perturbed Boltzmann equation (see \cite{GY-06} and \cite{Jiang-Xu-Zhao-2018-Indiana}, for instance), there are different singularities between the microscopic part and fluid part. So, we naturally divide the kinetic function $g_\eps  (t,x,v)$ into two different parts to derive a uniform energy estimate. More precisely, for fixed $(t,x)$, the function $g_\epsilon(t,x,v)$ can be decomposed as
\begin{align*}
  &  g_\epsilon = P g_\eps + (I-P) g_\eps \,,
\end{align*}
where $P g_\eps \in \mathcal{N}$ is called the fluid (macroscopic) part of $g_\eps$ with coefficients $(a_\eps, b_\eps, c_\eps) \in \R \times \R^3 \times \R$ and $(I-P) g_\eps \in \mathcal{N}^\perp$ is the kinetic (microscopic) part of $g_\eps$. Here $\mathcal{N}^\perp$ is the orthogonal space of $\mathcal{N}$ in $L^2_v$. Plugging the above decomposition into the scaled equation \eqref{VPB-g}, one gets the macroscopic evolution
\begin{equation}\label{decomposition}
  \begin{split}
    \epsilon \partial_t ( a_\eps + b_\eps \cdot v + c_\eps |v|^2 ) + v \cdot \nabla_x ( a_\eps + b_\eps \cdot v + c_\eps |v|^2 ) - \gamma v \cdot \nabla_x \phi_\eps = l + h + m \,,
  \end{split}
\end{equation}
where
\begin{equation}\label{l+h+m}
  \begin{aligned}
    l = & - \epsilon \partial_t (I-P) g_\eps - v \cdot \nabla_x (I-P) g_\eps - \tfrac{1}{\epsilon} L (I-P) g_\eps \,, \\
    h = & Q ( g_\eps , g_\eps ) \,, \\
    m = & \gamma \epsilon ( v \cdot \nabla_x \phi_\eps g_\eps - \nabla_x \phi_\eps \cdot \nabla_v g_\eps ) \,.
  \end{aligned}
\end{equation}

On the other hand, $a_\eps$, $b_\eps$ and $c_\eps$ obey the local macroscopic balance laws of mass, moment and energy. In fact, multiplying the first $f_\eps$-equation \eqref{Equ-Kinetic-f} of the VPB system by the collision invariant in \eqref{invariant} and integrating by parts over $v \in \mathbb{R}^3$, we have
\begin{align*}
  & \partial_t \int_{\mathbb{R}^3} f_\epsilon \d v + \nabla_x \cdot \int_{\mathbb{R}^3} v f_\epsilon \d v=0 \,, \\
  & \partial_t \int_{\mathbb{R}^3} v f_\epsilon \d v + \nabla_x \cdot \int_{\mathbb{R}^3} v \otimes v f_\epsilon \d v - \nabla_x \phi_\epsilon \int_{\mathbb{R}^3} f_\epsilon \d v = 0 \,, \\
  & \partial_t \int_{\mathbb{R}^3} |v|^2 f_\epsilon \d v + \nabla_x \cdot \int_{\mathbb{R}^3} |v|^2 v f_\epsilon \d v - 2 \nabla_x \phi_\epsilon \cdot \int_{\mathbb{R}^3} v f_\epsilon \d v = 0 \,.
\end{align*}
By plugging the perturbed form $f_\eps = M ( 1 + \eps g_\eps )$ and using the decomposition \eqref{Proj-Hydrodynamic}, we compute all moments appearing in the above system to obtain the following macroscopic balance laws
\begin{align}
  \partial_t (a_\eps + 3 c_\eps ) + \tfrac{1}{\epsilon} \nabla_x \cdot b_\eps = 0 \,, \label{1.7} \\
  \partial_t b_\eps + \tfrac{1}{\epsilon} \nabla_x (a_\eps + 5 c_\eps ) + \tfrac{1}{\epsilon} \langle v \cdot \nabla_x (I-P) g_\epsilon , v \rangle_{L^2_v} - \tfrac{\gamma}{\epsilon} \nabla_x \phi_\epsilon - \gamma (a_\eps + 3 c_\eps ) \nabla_x \phi_\eps = 0 \,, \label{Balance-b} \\
  3 \partial_t a_\eps + 15 \partial_t c_\eps + \tfrac{5}{\epsilon} \nabla_x \cdot b_\eps + \tfrac{1}{\epsilon} \langle v \cdot \nabla_x (I-P) g_\epsilon , |v|^2 \rangle_{L^2_v} - 2 \gamma b_\eps \cdot \nabla_x \phi_\epsilon = 0 \,, \label{1.8} \\
  \Delta_x \phi_\epsilon = \gamma (a_\eps + 3 c_\eps ) \,,\label{1.9}
\end{align}
where the evolution \eqref{1.8} for $c_\eps$ can be written as
\begin{equation}\label{Balance-c}
  \begin{aligned}
    \partial_t c_\eps + \tfrac{1}{3 \eps} \nabla_x \cdot b_\eps + \tfrac{1}{6 \eps} \langle v \cdot \nabla_x (I-P) g_\epsilon , |v|^2 \rangle_{L^2_v} = \tfrac{\gamma}{3} b_\eps \cdot \nabla_x \phi_\eps \,.
  \end{aligned}
\end{equation}

Our main goal of this paper is to derive the uniform global energy estimates to the perturbed VPB system \eqref{VPB-g}. We will divide into three steps to achieve our goal:

1) We derive the pure spatial derivative energy estimates, which will give us the kinetic dissipation $\tfrac{1}{\eps^2} \| (I-P) g_\eps \|^2_{H^N_x L^2_v (\nu)}$ due to the coercivity of linearized Boltzmann collision operator $L$ as shown in Lemma \ref{Lmm-Coercivity-L}. Moreover, the singular term $\tfrac{1}{\eps} v \cdot \nabla_x g_\eps$ will disappear since $\tfrac{1}{\eps} \l v \cdot \nabla_x \partial^\alpha_x g_\eps , \partial^\alpha_x g_\eps \r_{L^2_{x,v}} = 0$ for all $\alpha \in \mathbb{N}^3$. Thanks to $\tfrac{1}{\eps} Q (g_\eps, g_\eps) \in \mathcal{N}^\perp$, the equality $\tfrac{1}{\eps} \l \partial^\alpha_x Q (g_\eps, g_\eps) , \partial^\alpha_x g_\eps \r_{L^2_{x,v}} = \l \partial^\alpha_x Q (g_\eps, g_\eps) , \tfrac{1}{\eps} \partial^\alpha_x (I-P) g_\eps \r_{L^2_{x,v}}$ is such that the singularity $\tfrac{1}{\eps} \partial^\alpha_x (I-P) g_\eps $ will be controlled by the kinetic dissipation $\tfrac{1}{\eps^2} \| (I-P) g_\eps \|^2_{H^N_x L^2_v (\nu)}$. For the linear singular term $ - \tfrac{\gamma}{\eps} v \cdot \nabla_x \phi_\eps$, we will apply the local conservation law of mass $\partial_t \l g_\eps , 1 \r_{L^2_v} + \tfrac{1}{\eps} \l g_\eps, v \r_{L^2_v} = 0$ and the Poisson equation $\Delta_x \phi_\eps = \gamma \l g_\eps , 1 \r_{L^2_v}$ to deform the singular quantity $\l - \tfrac{\gamma}{\eps} v \cdot \nabla_x \partial^\alpha_x \phi_\eps, \partial^\alpha_x g_\eps \r_{L^2_{x,v}} = \tfrac{1}{2} \tfrac{\d}{\d t} \| \nabla_x \partial^\alpha_x \phi_\eps \|^2_{L^2_x} $, which will be a part of energy.

2) We estimate the macroscopic energy to find a dissipative structure of the fluid part $P g_\eps$ by employing the so-called macro-micro decomposition method, depending on the {\em thirteen moments}, see \cite{GY-06} for instance. Thus, we can obtain a fluid dissipation $\| \nabla_x P g_\eps \|^2_{H^{N-1}_x L^2_v}$, which is the same as in the perturbed Boltzmann. However, what is different with the Boltzmann hydrodynamic limit is that there is a more damping effect $\| \l g_\eps, 1 \r_{L^2_v} \|^2_{H^{N-1}_x} = \| a_\eps + 3 c_\eps \|^2_{H^{N-1}_x}$ in the perturbed VPB system resulted from the Poisson equation $\Delta_x \phi_\eps = \l g_\eps, 1 \r_{L^2_v} = a_\eps + 3 c_\eps$. More precisely, applying the divergence operator $\nabla_x \cdot $ on the balance law \eqref{Balance-b} for $b_\eps$ and replacing $\Delta_x \phi_\eps$ by $\gamma (a_\eps + 3 c_\eps)$ imply that
\begin{equation*}
  \begin{aligned}
    - \Delta_x (a_\eps + 3 c_\eps) + \gamma^2 (a_\eps + 3 c_\eps) = \textrm{ other terms} \,,
  \end{aligned}
\end{equation*}
We emphasize that no singular terms are generated in deriving the macroscopic energy estimates. However, there is a unsigned {\em interactive energy} quantity $E_N^{int} (g_\eps)$ defined in \eqref{Interactive-Energy} arising from the interactions between kinetic effects and fluid behaviors. Fortunately, the unsigned interactive energy $E_N^{int} (g_\eps)$ is of the small size such that it can be dominated by the energy $\mathcal{E}_N (g_\eps, \phi_\eps)$ (see Remark \ref{Rmk-MM}).

3) We derive the $(x,v)$-mixed derivatives estimates to closed the energy inequality. The uncontrolled quantity in the spatial derivative and macroscopic energy estimates are in terms of the $v$-derivatives of the kinetic part $(I-P) g_\eps$. We apply the microscopic projection $I-P$ to the $g$-equation of \eqref{VPB-g} and obtain
\begin{equation}\no
  \begin{aligned}
    \partial_t (I-P) g_\eps + \tfrac{1}{\eps^2} L (I-P) g_\eps = - \tfrac{1}{\eps} (I-P) (v \cdot \nabla_x g_\eps) + \tfrac{1}{\eps} Q (g_\eps, g_\eps) + \textrm{ other nonsingular terms} \,.
  \end{aligned}
\end{equation}
In the mixed derivatives estimate, the coercivity of $L$ will also supply us a kinetic dissipation $ \tfrac{1}{\eps^2} \| \partial^\alpha_x \partial^\beta_v (I-P) g_\eps \|^2_{L^2_{x,v} (\nu)}$ for all $| \alpha | + |\beta| \leq N$ with $\beta \neq 0$. As a result, the singular quantity $ \tfrac{1}{\eps} \l - \partial^\alpha_x \partial^\beta_v (I-P) (v \cdot \nabla_x g_\eps) + Q (g_\eps, g_\eps) , \partial^\alpha_x \partial^\beta_v (I-P) g_\eps \r_{L^2_{x,v}} $ can be absorbed by the kinetic dissipation $ \tfrac{1}{\eps^2} \| \partial^\alpha_x \partial^\beta_v (I-P) g_\eps \|^2_{L^2_{x,v} (\nu)}$. However, for fixed $|\alpha| + |\beta| \leq N$ with $\beta \neq 0$, the coercivity of $L$ under the $v$-derivatives $\partial^\beta_v$ will further generate a uncontrolled quantity $\tfrac{1}{\eps^2} \sum_{\beta' < \beta} \| \partial^\alpha_x \partial^{\beta'}_v (I-P) g_\eps \|^2_{L^2_{x,v} (\nu)}$ with two order singularity $\tfrac{1}{\eps^2}$. Thanks to the order $\beta'$ of $v$-derivatives in that quantity is strictly less than $\beta$ $(\neq 0)$, we can employ the induction to absorb this uncontrolled singular norm by the mixed derivative kinetic dissipation with lower order $v$-derivative. Thus we establish the global uniform energy estimates.

Finally, based on the global-in-time energy estimate uniformly in $\eps \in (0,1]$, we take the limit from the perturbed VPB system \eqref{VPB-g} to the incompressible NSFP equations \eqref{NSFP} as $\eps \rightarrow 0$. We mainly employ the Aubin-Lions-Simon Theorem to obtain enough compactness such that the limits valid.

\subsection{Historical remarks}

There has been lot of works on the well-posedness of VPB. In the context of weak solutions, after DiPerna-Lions' breakthrough on renormalized solutions of the Boltzmann equation \cite{D-L}, Lions generalized this theory to Vlasov type equations, including VPB \cite{Lions1994-1}. In the context of classical solutions, in early 2000, in a series works \cite{guo-1, Guo-2002-VPB}, Guo developed the so-called nonlinear energy method to build up global in time well-posedness of Boltzmann equations, VPB and other kinetic equations. Later on, there are several works on more general collision kernels, among many papers, we only list \cite{D-Y-Z(hard)-2011, D-Y-Z(soft)-2011, XXZhard, XXZsoft}.

One of the most important features of the Boltzmann type equations is its connection to the fluid equations. The so-called fluid regimes of the Boltzmann type equation are those of asymptotic dynamics of the scaled Boltzmann equations when the Knudsen number $\eps$ is very small. Justifying these limiting processes rigorously has been an active research field from late 70's. Among many results obtained, the main contributions are the incompressible Navier-Stokes and Euler limits. There are two types of results in this field:
\begin{enumerate}
  \item First obtaining the solutions of the scaled Boltzmann equation {\em uniform} in the Knudsen number $\eps$, then extracting a convergent (at least weakly) subsequence converging to the solutions of the fluid equations as $\eps\rightarrow 0\,;$.
  \item First obtaining the solutions for the limiting fluid equations, then constructing a sequence of special solutions (around the Maxwellian) of the scaled Boltzmann equations for small Knudsen number $\eps$.
\end{enumerate}
The key difference between the results of type (1) and (2) are: in type (1), the solutions of the fluid equations are {\em not} known a priori, and are completely obtained from taking limits from the Boltzmann equation. In short, it is ``from kinetic to fluid"; In type (2), the solutions of the fluid equations are {\em known} first. In short, it is ``from fluid to kinetic".

The most successful program in type (1) is the so-called BGL program. As mentioned above, the DiPerna-Lions's renormalized solutions for cutoff kernel \cite{D-L} (also the non-cutoff kernels in \cite{al-3}) are the only solutions known to exist globally without any restriction on the size of the initial data so far. From late 80's, Bardos-Golse-Levermore initialized the program (BGL program in brief) to justify Leray's solutions to the incompressible Navier-Stokes equations from DiPerna-Lions' renormalized solutions \cite{BGL1}, \cite{BGL2}. They proved the first convergence result with 5 additional technical assumptions. After 10 years effects by Bardos, Golse, Levermore, Lions and Saint-Raymond, see for example \cite{BGL3, LM3, LM4, GL}, the first complete convergence result without any additional compactness assumption was proved by Golse and Saint-Raymond in \cite{Go-Sai04} for cutoff Maxwell collision kernel, and in \cite{Go-Sai09} for hard cutoff potentials. Later on, it was extended by Levermore-Masmoudi \cite{LM} to include soft potentials. Recently Arsenio got the similar results for non-cutoff case \cite{Arsenio}. Furthermore, by Jiang, Levermore, Masmoudi and Saint-Raymond, these results were extended to bounded domain where the Boltzmann equation was endowed with the Maxwell reflection boundary condition \cite{MSRM-CPAM2003, JLM-CPDE2010, JM-CPAM2017}, based on the solutions obtained by Mischler \cite{mischler2010asens}.

The BGL program says that, given any $L^2\mbox{-}$bounded functions $(\rho_0, \mathrm{u}_0, \theta_0)$, and for any physically bounded initial data (as required in DiPerna-Lions solutions) $F_{\eps,0}= M + \eps \sqrt{M}g_{\eps,0}$, such that suitable moments of the fluctuation $g_{\eps,0}$, say, $(\mathcal{P}(g_{\eps,0}, v\sqrt{M})_{L^2(\mathbb{R}^3_v)}, (g_{\eps,0}, (\tfrac{|v|^2}{5}-1)\sqrt{M})_{L^2(\mathbb{R}^3_v)})$ converges in the sense of distributions to $(\mathrm{u}_0, \theta_0)$, the corresponding DiPerna-Lions solutions are $F_\eps(t,x,v)$. Then the fluctuations $g_{\eps}$ (defined by $F_\eps= M + \eps \sqrt{M}g_{\eps}$) has weak compactness, such that the corresponding moments of $g_\eps$ converge weakly in $L^1$ to $(\mathrm{u},\theta)$ which is a Leray solution of the incompressible Navier-Stokes equation whose viscosity and heat conductivity coefficients are determined by microscopic information, with initial data $(\mathrm{u}_0, \theta_0)$. Under some situations, for example the well-prepared initial data or in bounded domain with suitable boundary condition, the convergence could be strong $L^1$.

We emphasize that the BGL program indeed gave a new proof of Leray's solutions to the incompressible Navier-Stokes equation, in particular the energy inequality which can be derived from the entropy inequality of the Boltzmann equation. Any a priori information of the Navier-Stokes equation is {\em not} needed, and completely derived from the microscopic Boltzmann equation. In this sense, BGL program is spiritually a part of Hilbert's 6th problem: derive and justify the macroscopic fluid equations from the microscopic kinetic equations (see \cite{SRM2010}).

Another direction in type (1) is in the context of classical solutions. The first work in this type is Bardos-Ukai \cite{b-u}. They started from the scaled Boltzmann equation for cut-off hard potentials, and proved the global existence of classical solutions $g_\eps$ uniformly in $0< \eps <1$. The key feature of Bardos-Ukai's work is that they only need the smallness of the initial data, and did not assume the smallness of the Knudsen number $\eps$. After having the uniform in $\eps$ solutions $g_\eps$, taking limits can provide a classical solution of the incompressible Navier-Stokes equations with small initial data. Bardos-Ukai's approach heavily depends on the sharp estimate especially the spectral analysis on the linearized Boltzmann operator $\mathcal{L}$, and the semigroup method (the semigroup generated by the scaled linear operator $\eps^{-2}\mathcal{L}+\eps^{-1}v\cdot\nabla_{\!x}$). It seems that it is hardly extended to soft potential cutoff, and even harder for the non-cutoff cases, since it is well-known that the operator $\mathcal{L}$ has continuous spectrum in those cases. On the torus, semigroup approach has been employed by Briant \cite{Briant} and Briant, Merino-Aceituno and Mouhot \cite{BMMouhot} to prove incompressible Navier-Stokes limit by employing the functional analysis breakthrough of Gualdani-Mischler-Mouhot \cite{GMM}. Again, their results are for cut-off kernels with hard potentials. Recently, there is type (1) convergence result on the incompressible Navier-Stokes limit of the Boltzmann equation. In \cite{Jiang-Xu-Zhao-2018-Indiana}, the uniform in $\eps$ global existence of the Boltzmann equation with or without cutoff assumption was obtained and the global energy estimates were established. Then taking limit as $\eps\rightarrow 0$, it was proved the incompressible Navier-Stokes limit.

Most of the type (2) results are based on the Hilbert expansion  and obtained in the context of classical solutions. It was started from Nishida and Caflisch's work on the compressible Euler limit \cite{Nishida, Caflisch, KMN}. Their approach was revisitied by Guo, Jang and Jiang, combining with nonlinear energy method to apply to the acoustic limit \cite{GJJ-KRM2009, GJJ-CPAM2010, JJ-DCDS2009}. After then this process was used for the incompressible limits, for examples, \cite{DEL-89} and \cite{GY-06}. In \cite{DEL-89}, De Masi-Esposito-Lebowitz considered  Navier-Stokes limit in dimension 2. More recently, using the nonlinear energy method, in \cite{GY-06} Guo justified the Navier-Stokes limit (and beyond, i.e. higher order terms in Hilbert expansion). This result was extended in \cite{JX-SIMA2015} to more general initial data which allow the fast acoustic waves.  These results basically say that, given the initial data which is needed in the classical solutions of the Navier-Stokes equation, it can be constructed the solutions of the Boltzmann equation of the form $F_\eps = M + \eps \sqrt{M}(g_1+ \eps g_2 + \cdots + \eps^n g_\eps)$, where $g_1, g_2, \cdots $ can be determined by the Hilbert expansion, and $g_\eps$ is the error term. In particular, the first order fluctuation $g_1 = \rho_1 + \mathrm{u}_1\!\cdot\! v + \theta_1(\frac{|v|^2}{2}-\frac{3}{2})$, where $(\rho_1, \mathrm{u}_1, \theta_1)$ is the solutions to the incompressible Navier-Stokes equations. 

The main purpose of the current paper is to prove a type (1) convergence result in the context of classical solutions for one-species VPB. This can be viewed as an analogue of incompressible Navier-Stokes limit of the Boltzmann equation recently done by the second-named author with Xu and Zhao \cite{Jiang-Xu-Zhao-2018-Indiana}. We obtain the uniform in $\eps$ estimates and then justify the limit to the incompressible Navier-Stokes-Fourier-Poisson equations. The main difficulties and novelties of the paper has been mentioned in the previous subsection. 

This paper is organized as follows. The next section is devoted to justify the local existence under the small size of the initial data. In Section \ref{Sec: Uniform-Bnd}, we derive the uniform energy estimate and establish the global existence with small initial data. In Section \ref{Sec: Limits}, based on Theorem \ref{Thm-global}, we rigorously prove limit from the perturbed VPB system \eqref{VPB-g} to the incompressible NSFP equations \eqref{NSFP} as $\eps \rightarrow 0$.

\section{Construction of Local Solutions}

In this section, we will construct a unique local-in-time solution to the perturbed VPB system \eqref{VPB-g}-\eqref{IC-g} for all $0 < \eps \leq 1$ by employing an iterative schedule. Before doing that, we will first give the following known preliminaries, which will be frequently used later.

\subsection{Preliminaries}
We first give the following Sobolev type inequalities, which can be seen in \cite{Nirenberg-1959-ASNSP} or \cite{Adams-Fournier-2003-Sobolev} for instance.
\begin{lemma}\label{Lm-Sobolev-Inq}
	Let $ u = u(x) \in H^2_x$. Then
	\begin{enumerate}
		\item $\| u \|_{L^\infty_x} \lesssim \| \nabla_x u \|^{\frac{1}{2}}_{L^2_x} \| \nabla^2_x u \|^{\frac{1}{2}}_{L^2_x} \lesssim  \| \nabla_x u \|_{H^1_x} \,;$
		
		\item $\| u \|_{L^6_x} \lesssim \| \nabla_x u \|_{L^2_x} \,;$
		
		\item $\| u \|_{L^q_x} \lesssim \| u \|_{H^1_x}$ for all $ 2<q<6$.
		
	\end{enumerate}
\end{lemma}

We now give the properties of collision frequency $\nu (v)$, which will frequently used in energy estimates.
\begin{lemma}\label{Lmm-CF-nu}
	The collision frequency $\nu(v)$ defined in \eqref{Colli-Frequence-nu} has the following properties:
	\begin{enumerate}
		\item $\nu(v)$ is smooth and there are positive constants $C_1$ and $C_2$ such that
		\begin{equation}\label{nu-Propty-1}
		C_1 ( 1 + |v| ) \leq  \nu(v) \leq C_2 ( 1 + |v| )
		\end{equation}
		for every $v \in \R^3$.
		
		\item For any $\beta \in \mathbb{N}^3$, $\beta \neq 0$,
		\begin{equation}\label{nu-Propty-2}
		\sup_{v \in \R^3} | \partial^\beta_v \nu(v) | < + \infty \,.
		\end{equation}
		
		\item If the velocities $v$, $v_*$, $v'$, $v_*' \in \R^3$ satisfy $ v + v_* = v' + v_*' $ and $ |v|^2 + |v_*|^2 = |v'|^2 + |v_*'|^2 $, then
		\begin{equation}\label{nu-Propty-3}
		\nu(v) + \nu(v_*) \leq C_3 ( \nu(v') + \nu(v_*') )
		\end{equation}
		holds for some positive constant $C_3$.
	\end{enumerate}
\end{lemma}
We figure out that the proof of Lemma \ref{Lmm-CF-nu} can be referred to \cite{Jiang-Luo-2019-VMB}, \cite{Guo-2002-VPB} or \cite{GY-06}.

Next the coercivity of the linearized Boltzmann collisional operator $L$ defined in \eqref{Def-L}, which will give us the dissipative structure of the kinetic equation. We refer to \cite{Saint-Raymond-2009-Boltzmann} for more details and to \cite{Levermore-Sun-2010-KRM} for a modern and general treatment of the linearized Boltzmann operator.
\begin{lemma}\label{Lmm-Coercivity-L}
	There exists a  $\delta>0$ such that
	\begin{equation}
	  \langle Lf, f \rangle_{L^2_v} \geq \delta \| (I-P) f \|_{L^2_v (\nu)}^2 \,.
	\end{equation}
	Moreover, there are $\delta_1 , \delta_2 > 0$ such that
	\begin{equation}
	  \begin{aligned}
	    \big\langle \partial^\beta_v (I-P) f , \partial^\beta_v (I-P) f \big\rangle_{L^2_v} \geq \delta_1 \| \partial^\beta_v ( I -P ) f \|^2_{L^2_v (\nu)} - \delta_2 \sum_{\beta' < \beta} \| \partial^{\beta'}_v (I-P) f \|^2_{L^2_v}
	  \end{aligned}
	\end{equation}
	for all multi-indexes $\beta \in \mathbb{N}^3$.
\end{lemma}

For the bilinear symmetric operator $Q$ defined in \eqref{Colli-Q}, we refer to \cite{Guo-2002-VPB} or \cite{GY-06} for more details. For convenience to readers, we rewrite the results here.
\begin{lemma}\label{Lmm-Q}
	Let $g_i (x, v) \,, (i=1,2,3) $ be smooth functions, and $(i,j)=(1,2)$ or $(i,j)=(2,1)$, then we have
	\begin{equation}
	\big| \langle \partial^\beta_v Q (g_1, g_2), g_3 \rangle_{L^2_{x,v}} \big| \lesssim \sum_{\substack{\beta_1 + \beta_2 \leq \beta \\ (i,j)}} \int_{\mathbb{R}^3}
	\| \partial_v^{\beta_1} g_i \|_{L^2_v (\nu)} \| \partial^{\beta_2}_v g_j \|_{L^2_v} \| g_3 \|_{L^2_v (\nu)} \d x
	\end{equation}
	for any $\beta \in \mathbb{N}^3$.
\end{lemma}

\subsection{Local Existence}

In this subsection, we will construct a unique local-in-time solution to the perturbed VPB system \eqref{VPB-g} for all $0 < \eps \leq 1$ under small size of the initial data. Fixed $\eps \in (0,1)$, the construction is based on a uniform energy estimate for a sequence of iterating approximate solutions. Then we articulate the following lemma.

\begin{lemma}\label{Lmm-Local}
	There exist $ 0 < \delta \leq 1 $ and $ 0 < T \leq 1 $ such that for any $ 0 < \epsilon \leq 1 $, $ g_{\eps, 0} \in H^N_{x,v}$ $( N \geq 3)$ with $ \mathcal{E}_N ( g_{\eps, 0} , \phi_{\eps, 0} ) \leq \delta$, the system \eqref{VPB-g}-\eqref{IC-g} admits a unique solution $ g_\eps \in L^\infty (0,T; H^s_{x,v}) \cap L^2(0,T; H^s_{x,v} (\nu))$ and $ \phi_\eps \in L^\infty (0, T; H^{N+1}_x)$ with uniform energy bound
	\begin{equation}\label{2.3}
	  \sup_{t\in [0,T]} \mathcal{E}_N (g_\eps (t) , \phi_\eps (t) ) + \tfrac{1}{\epsilon^2} \int_0^T \| ( I - P ) g_\eps (t) \|^2_{H^s_{x,v}(\nu)} \d t \leq C
	\end{equation}
	for some constant $C > 0$ independent of $\eps$.
\end{lemma}
\begin{proof}[Proof of Lemma \ref{Lmm-Local}]
	For any fixed $\eps \in (0,1]$, we consider the following linear iterative approximate sequence $(n \geq 0)$ for solving the perturbed VPB system \eqref{VPB-g} with initial data \eqref{IC-g}:
	\begin{equation}\label{Iter-Approx-Syst}
	\left\{
	\begin{array}{l}
	\partial_t g^{n+1}_\eps + \tfrac{1}{\epsilon} v \cdot \nabla_x g^{n+1}_\eps - \tfrac{\gamma}{\epsilon} v \cdot \nabla_x \phi^{n+1}_\eps + \tfrac{1}{\epsilon^2} L g^{n+1}_\eps \\
	\qquad \qquad + \gamma \nabla_x \phi^n_\eps \cdot \nabla_v g^{n+1}_\eps - \gamma v \cdot \nabla_x \phi^n_\eps g^{n+1}_\eps = \tfrac{1}{\epsilon} Q(g^n_\eps, g^n_\eps) \,,\\
	\Delta_x \phi^{n+1}_\eps = \gamma \langle 1, g^{n+1}_\eps \rangle_{L^2_v}
	\end{array}
	\right.
	\end{equation}
	with initial data
	\begin{equation}\label{IC-Iter-Appro}
	\begin{aligned}
	g^{n+1}_\eps |_{t=0} = g_{\eps, 0} (x,v) \,.
	\end{aligned}
	\end{equation}
	We start with $g_\eps^0 (t,x,v) = g_{\eps, 0} (x,v)$ for all $t \geq 0$. For the linear Cauchy problem \eqref{Iter-Approx-Syst}-\eqref{IC-Iter-Appro}, the existence of $g^{n+1}$ is assured above by employing the standard linear theory, once given $g^n$ satisfying \eqref{2.3}.
	
	Now we derive the uniform energy estimates of the iterative approximate system \eqref{Iter-Approx-Syst}. For notational simplicity, we drop the lower index $\eps$ of $g^{n+1}_\eps$ in what follows. For $ \alpha \in \mathbb{N}^3 $ with $|\alpha| \leq N$, we act the derivative operator $\partial_x^\alpha$ on the first $g^{n+1}$-equation of \eqref{Iter-Approx-Syst} and take $L^2_{x,v}$-inner product with $\partial_x^\alpha g^{n+1},$ then we gain
	\begin{equation}\label{2.4}
	  \begin{split}
	    & \tfrac{1}{2} \tfrac{\d}{\d t} \big( \| \partial_x^\alpha g^{n+1} \|^2_{L^2_{x,v}} + \| \nabla_x \partial_x^\alpha \phi^{n+1} \|^2_{L^2_x} \big) + \tfrac{\delta}{\epsilon^2} \| (I-P) \partial_x^\alpha g^{n+1} \|_{L^2_{x,v}(\nu)}^2 \\
	    & \leq \underbrace{ \gamma \langle \partial_x^\alpha ( \nabla_x \phi^n \cdot \nabla_v g^{n+1} ) - \partial_x^\alpha ( v \cdot \nabla_x \phi^n g^{n+1} ) , \partial_x^\alpha g^{n+1} \rangle_{L^2_{x,v}} }_{I} \\
	    & + \underbrace{ \tfrac{1}{\epsilon} \langle \partial_x^\alpha Q(g^n, g^n) , \partial_x^\alpha (I-P) g^{n+1} \rangle_{L^2_{x,v}} }_{I\!I}\,,
	  \end{split}
	\end{equation}
	where we make use of Lemma \ref{Lmm-Coercivity-L} and the equality
	$$\tfrac{1}{2} \tfrac{\d}{\d t} \| \partial_x^\alpha \nabla_x \phi^{n+1} \|^2_{L^2_x} = \langle - \tfrac{\gamma}{\epsilon} v \cdot \nabla_x \partial_x^\alpha \phi^{n+1} , \partial_x^\alpha g^{n+1} \rangle_{L^2_{x,v}} \,,$$
	which is implied by the Poisson equation of $\phi^{n+1}$ in \eqref{Iter-Approx-Syst} and the charge conservation law
	\begin{equation*}
	  \partial_t \langle g^{n+1} , 1 \rangle_{L^2_v} + \tfrac{1}{\eps} \nabla_x \cdot \langle g^{n+1} , v \rangle_{L^2_v} = 0 \,.
	\end{equation*}
	
	Now we estimate the first term $I$ on the right-hand side of \eqref{2.4}. It is easy to be derived from the Leibniz formula that
	\begin{equation}
	  \begin{split}
	    I = & \underbrace{ \langle \nabla_x \phi^n \cdot \nabla_v \partial_x^\alpha g^{n+1} - v \cdot \nabla_x \phi^n \partial_x^\alpha g^{n+1} , \partial_x^\alpha g^{n+1} \rangle_{L^2_{x,v}} }_{I_1} \\
	    & + \sum_{ 0 \neq \alpha_1 \leq \alpha } \underbrace{ C_\alpha^{\alpha_1} \langle \nabla_x \partial_x^{\alpha_1} \phi^n \cdot \nabla_v \partial_x^{\alpha-\alpha_1} g^{n+1} , \partial_x^\alpha g^{n+1} \rangle_{L^2_{x,v}} }_{I_2^{\alpha_1}} \\
	    & - \sum_{ 0 \neq \alpha_1 \leq \alpha } \underbrace{ C_\alpha^{\alpha_1} \langle v \cdot \nabla_x \partial_x^{\alpha_1} \phi^n \partial_v^{\alpha-\alpha_1} g^{n+1} , \partial_x^\alpha g^{n+1} \rangle_{L^2_{x,v}} }_{I_3^{\alpha_2}} \,.
	  \end{split}
	\end{equation}
	From Lemma \ref{Lmm-CF-nu}, the H\"older inequality and the Sobolev embedding $H^2_x \hookrightarrow L^\infty_x$, we derive that
	\begin{equation}\label{I1}
	  \begin{split}
	    I_1 = & - \tfrac{1}{2} \langle v \cdot \nabla \phi^n, |\partial^\alpha_x g^{n+1}|^2 \rangle_{L^2_{x,v}} \lesssim \| \nabla_x \phi^n \|_{L^\infty_x} \| \partial^\alpha_x g^{n+1} \|^2_{L^2_{x,v}(\nu)} \\
	    \lesssim & \| \nabla_x \phi^n \|_{H^2_x} \| \partial_x^\alpha g^{n+1} \|_{L^2_{x,v}(\nu)}^2 \,.
	  \end{split}
	\end{equation}
	For the quantity $I_2^{\alpha_1}$ with $2 \leq |\alpha_1| \leq |\alpha| -1$, we deduce from the Sobolev embedding $H^1_x \hookrightarrow L^4_x$ that
	\begin{equation}
	  \begin{split}
	    I_2^{\alpha_1} \lesssim & \| \nabla_x \partial^{\alpha_1}_x \phi^n \|_{L^4_x} \| \nabla_v \partial^{\alpha - \alpha_1}_x g^{n+1} \|_{L^4_x L^2_v} \| \partial^\alpha_x g^{n+1} \|_{L^2_{x,v}} \\
	    \lesssim & \| \nabla_x \phi^n \|_{H^N_x} \| \nabla_v g^{n+1} \|_{H^{N-1}_x L^2_v} \| g^{n+1} \|_{H^N_x L^2_v} \,.
	  \end{split}
	\end{equation}
	If $|\alpha_1| = 1$ or $\alpha_1 = \alpha$, we derive from the Sobolev embedding $H^2_x \hookrightarrow L^\infty_x$ that
	\begin{equation}
	  \begin{aligned}
	    & \sum_{|\alpha_1| = 1 \textrm{ or } \alpha_1 = \alpha} I_2^{\alpha_1} \lesssim \| \nabla_x \partial^\alpha_x \phi^n \|_{L^2_x} \| \nabla_v g^{n+1} \|_{L^\infty_x L^2_v} \| \partial^\alpha_x g^{n+1} \|_{L^2_{x,v}} \\
	    + & \sum_{|\alpha_1| = 1} \| \nabla_x \partial^{\alpha_1}_x \phi^n \|_{L^\infty_x} \| \nabla_v \partial^{\alpha - \alpha_1}_x g^{n+1} \|_{L^2_{x,v}} \| \partial^\alpha_x g^{n+1} \|_{L^2_{x,v}} \\
	    \lesssim & \| \nabla_x \phi^n \|_{H^N_x} \| \nabla_v g^{n+1} \|_{H^{N-1}_x L^2_v} \| g^{n+1} \|_{H^N_x L^2_v} \,.
	  \end{aligned}
	\end{equation}
	We thereby have
	\begin{equation}\label{I2-alph1}
	  \begin{aligned}
	    & \sum_{0 \neq \alpha_1 \leq \alpha } I_2^{\alpha_1} \lesssim \| \nabla_x \phi^n \|_{H^N_x} \| \nabla_v g^{n+1} \|_{H^{N-1}_x L^2_v} \| g^{n+1} \|_{H^N_x L^2_v} \\
	    \lesssim & \| \nabla_x \phi^n \|_{H^N_x} \Big( \| \nabla_v P g^{n+1} \|_{H^{N-1}_x L^2_v} + \| \nabla_v (I-P) g^{n+1} \|_{H^{N-1}_x L^2_v} \big) \| g^{n+1} \|_{H^N_x L^2_v} \\
	    \lesssim &  \| \nabla_x \phi^n \|_{H^N_x} \| g^{n+1} \|_{H^N_x L^2_v}^2 + \| \nabla_x \phi^n \|_{H^N_x} \| \nabla_v (I-P) g^{n+1} \|_{H^{N-1}_x L^2_v} \| g^{n+1} \|_{H^N_x L^2_v} \,,
	   \end{aligned}
	\end{equation}
	where the decomposition $g^{n+1} = P g^{n+1} + (I-P) g^{n+1}$ and the bound $\| \nabla_v P g^{n+1} \|_{H^{N-1}_x L^2_v} \lesssim \| g^{n+1} \|_{H^N_x L^2_v}$ are utilized. Analogously, we can also estimate that
	\begin{equation}\label{I3-alpha1}
	  - \sum_{0 \neq \alpha_1 \leq \alpha } I_3^{\alpha_1} \lesssim \| \nabla_x \phi^n \|_{H^N_x} \big( \| g^{n+1} \|^2_{H^N_x L^2_v} + \| (I-P) g^{n+1} \|^2_{H^N_x L^2_v (\nu)} \big) \,.
	\end{equation}
	In summary, from the inequalities \eqref{I1}, \eqref{I2-alph1} and \eqref{I3-alpha1}, we have
	\begin{equation}\label{I}
	  I \lesssim \| \nabla_x \phi^n \|_{H^N_x} \Big( \| g^{n+1} \|^2_{H^N_x L^2_v } + \| \nabla_v (I-P) g^{n+1} \|^2_{H^{N-1}_x L^2_v} + \| (I-P) g^{n+1} \|^2_{H^N_x L^2_v (\nu)} \Big) \,.
	\end{equation}
	
	Next we estimate the term $I\!I$ in the right-hand side of \eqref{2.4}. In Lemma \ref{Lmm-Q}, let $g_1 = \partial^{\alpha'}_x g^n$, $g_2 = \partial^{\alpha - \alpha'}_x g^n$ and $g_3 = \partial^\alpha_x ( I - P ) g^{n+1}$. Then we have
	\begin{equation}
	  \begin{aligned}
	    I\!I = & \tfrac{1}{\eps} \sum_{\alpha' \leq \alpha} C_\alpha^{\alpha'} \langle Q ( \partial^{\alpha'}_x g^n , \partial^{\alpha - \alpha'}_x g^n ) , \partial^\alpha_x ( I - P ) g^{n+1} \rangle_{L^2_{x,v}} \\
	    \lesssim & \tfrac{1}{\eps} \sum_{\alpha' \leq \alpha} \underbrace{ \int_{\T^3} \| \partial^{\alpha'}_x g^n \|_{L^2_v (\nu)} \| \partial^{\alpha - \alpha'}_x g^n \|_{L^2_v} \| \partial^\alpha_x ( I - P ) g^{n+1} \|_{L^2_v (\nu)} \d x }_{I\!I_{\alpha'}} \,.
	  \end{aligned}
	\end{equation}
	If $\alpha' = 0 $ or $\alpha$, it is derived from the Sobolev embedding $H^2_x \hookrightarrow L^\infty_x$ and the decomposition $g^n = P g^n + (I - P) g^n$ that
	\begin{equation}\label{II-alpha'1}
	  \begin{aligned}
	    I\!I_{\alpha'} \lesssim & \big( \| g^n \|_{L^\infty_x L^2_v (\nu)} \| \partial^\alpha_x g^n \|_{L^2_{x,v}} + \| \partial^\alpha_x g^n \|_{L^2_{x,v} (\nu)} \| g^n \|_{L^\infty_x L^2_v} \big) \| \partial^m_x ( I - P ) g^{n+1} \|_{L^2_{x,v} (\nu)} \\
	    \lesssim & \| g^n \|_{H^N_x L^2_v} \| g^n \|_{H^N_x L^2_v (\nu)} \| ( I - P ) g^{n+1} \|_{H^N_x L^2_v (\nu)} \\
	    \lesssim & \| g^n \|_{H^N_x L^2_v} \big( \| g^n \|_{H^N_x L^2_v} + \| ( I - P ) g^n \|_{H^N_x L^2_v (\nu)} \big) \| ( I - P ) g^{n+1} \|_{H^N_x L^2_v (\nu)} \,,
	  \end{aligned}
	\end{equation}
	where the bound $ \| P g^n \|_{H^N_x L^2_v (\nu)} \lesssim \| g^n \|_{H^N_x L^2_v} $ is also utilized. If $1 \leq |\alpha'| \leq |\alpha| - 1$, we deduce from the Sobolev embedding $H^1_x \hookrightarrow L^4_x$ that
	\begin{equation}\label{II-alpha'2}
	  \begin{aligned}
	    I\!I_{\alpha'} \lesssim & \| \partial^{\alpha'}_x g^n \|_{L^4_x L^2_v (\nu)} \| \partial^{\alpha - \alpha'}_x g^n \|_{L^4_x L^2_v} \| \partial^\alpha_x (I-P) g^{n+1} \|_{L^2_{x,v} (\nu)} \\
	    \lesssim & \| g^n \|_{H^N_x L^2_v} \| g^n \|_{H^N_x L^2_v (\nu)} \| ( I - P ) g^{n+1} \|_{H^N_x L^2_v (\nu)} \\
	    \lesssim & \| g^n \|_{H^N_x L^2_v} \big( \| g^n \|_{H^N_x L^2_v} + \| ( I - P ) g^n \|_{H^N_x L^2_v (\nu)} \big) \| ( I - P ) g^{n+1} \|_{H^N_x L^2_v (\nu)} \,.
	  \end{aligned}
	\end{equation}
	Consequently, from \eqref{II-alpha'1} and \eqref{II-alpha'2}, we deduce that
	\begin{equation}\label{II}
	  I\!I \lesssim \tfrac{1}{\eps} \| g^n \|_{H^N_x L^2_v} \big( \| g^n \|_{H^N_x L^2_v} + \| ( I - P ) g^n \|_{H^N_x L^2_v (\nu)} \big) \| ( I - P ) g^{n+1} \|_{H^N_x L^2_v (\nu)} \,.
	\end{equation}
	From plugging the inequalities \eqref{I} and \eqref{II} into (\ref{2.4}) and summing up over all $|\alpha|\leq N$, we deduce that
	\begin{equation}\label{Iter-Spatial-Bnd}
	  \begin{split}
	    & \tfrac{1}{2} \tfrac{\d}{\d t} \big( \| g^{n+1} \|^2_{H^N_x L^2_v} + \| \nabla_x \phi^{n+1} \|^2_{H^N_x} \big) + \tfrac{\delta}{\eps^2} \| ( I - P ) g^{n+1} \|^2_{H^N_x L^2_v(\nu)} \\
	    & \lesssim \| \nabla_x \phi^n \|_{H^N_x} \Big( \| g^{n+1} \|^2_{H^N_x L^2_v } + \| \nabla_v (I-P) g^{n+1} \|^2_{H^{N-1}_x L^2_v} + \| (I-P) g^{n+1} \|^2_{H^N_x L^2_v (\nu)} \Big) \\
	    & + \tfrac{1}{\eps} \| g^n \|_{H^N_x L^2_v} \big( \| g^n \|_{H^N_x L^2_v} + \| ( I - P ) g^n \|_{H^N_x L^2_v (\nu)} \big) \| ( I - P ) g^{n+1} \|_{H^N_x L^2_v (\nu)} \,.
	  \end{split}
	\end{equation}
	
	It remains to control the norm $\| \nabla_v (I-P) g^{n+1} \|_{H^{N-1}_x L^2_v}$ in \eqref{Iter-Spatial-Bnd} isolated with the mixed derivative estimates. By applying the microscopic projection $(I-P)$ to the first $g^{n+1}$-equation of \eqref{Iter-Approx-Syst}, one has the microscopic evolution equation
	\begin{equation}\label{2.8}
	  \begin{split}
	    & \partial_t (I-P) g^{n+1} + \tfrac{1}{\epsilon} ( I - P) \big( v \cdot \nabla_x (I-P) g^{n+1} \big) + \tfrac{1}{\epsilon^2} L (I-P) g^{n+1} \\
	    & + \gamma (I-P) \big[ \nabla_x \phi^n \cdot \nabla_v (I-P) g^{n+1} - v \cdot \nabla_x \phi^n (I-P) g^{n+1} \big] \\
	    = & \tfrac{1}{\epsilon} Q ( g^n , g^n ) + \gamma (I-P) \big( v \cdot \nabla_x \phi^n P g^{n+1} \big) - \tfrac{1}{\eps} (I-P) ( v \cdot \nabla_x P g^{n+1} ) \,,
	  \end{split}
	\end{equation}
	where we also use the relations
	$$ ( I - P ) ( v \cdot \nabla_x \phi^{n+1} ) = (I-P) ( \nabla_x \phi^n \cdot \nabla_v P g^{n+1} ) = 0 \,. $$

	For any fixed $\alpha \,, \beta \in  \mathbb{N}^3$ with $ \beta \neq 0$ and $|\alpha|+|\beta|\leq N (N\geq 4)$, one can deduce from taking mixed derivative operator $ \partial_x^\alpha \partial_v^\beta$ on \eqref{2.8}, taking $L^2_{x,v}$-inner product by dot with $\partial_x^\alpha \partial_v^\beta (I-P) g^{n+1}$ and integrating by parts over $\mathbb{R}^3 \times \mathbb{R}^3 $ that for some constants $\delta_1, \delta_2 > 0$
	\begin{equation}\label{2.9}
	  \begin{split}
	    & \tfrac{1}{2} \tfrac{\d}{\d t} \| \partial_x^\alpha \partial_v^\beta (I-P) g^{n+1} \|^2_{L^2_{x,v}} + \tfrac{\delta_1}{\epsilon^2} \| \partial_x^\alpha \partial_v^\beta (I-P) g^{n+1} \|_{L^2_{x,v}(\nu)}^2 \\
	    & \leq \tfrac{\delta_2}{\eps^2} \sum_{\beta' < \beta} \| \partial_x^\alpha \partial_v^{\beta'} (I-P) g^{n+1} \|_{L^2_{x,v}}^2 \\
	    & + \underbrace{ \tfrac{1}{\epsilon} \sum_{ | \beta_1 | = 1} C_\beta^{\beta_1} \langle \partial_v^{\beta_1} v \cdot \nabla_x \partial_x^\alpha \partial_v^{\beta-\beta_1} (I-P) g^{n+1} , \partial_x^\alpha \partial_v^\beta (I-P) g^{n+1} \rangle_{L^2_{x,v}} }_{S_1} \\
	    & \underbrace{ - \gamma \langle \partial_x^\alpha \partial_v^\beta (I-P) ( \nabla_x \phi^n \cdot \nabla_v (I-P) g^{n+1} ) , \partial_x^\alpha \partial_v^\beta (I-P) g^{n+1} \rangle_{L^2_{x,v}} }_{S_2} \\
	    & + \underbrace{ \gamma \langle \partial_x^\alpha \partial_v^\beta (I-P) ( v \cdot \nabla_x \phi^n (I-P) g^{n+1} ) , \partial_x^\alpha \partial_v^\beta (I-P) g^{n+1} \rangle_{L^2_{x,v}} }_{S_3} \\
	    & + \underbrace{  \tfrac{1}{\epsilon} \langle \partial_x^\alpha \partial_v^\beta Q ( g^n , g^n ) , \partial_x^\alpha \partial_v^\beta (I-P) g^{n+1} \rangle_{L^2_{x,v}} }_{S_4} \\
	    & + \underbrace{ \gamma \langle \partial_x^\alpha \partial_v^\beta (I-P) ( v \cdot \nabla_x \phi^n P g^{n+1} ) , \partial_x^\alpha \partial_v^\beta (I-P) g^{n+1} \rangle_{L^2_{x,v}} }_{S_5} \\
	    & \underbrace{ - \tfrac{1}{\eps} \langle \partial_x^\alpha \partial_v^\beta (I-P) ( v \cdot \nabla_x P g^{n+1} ), \partial_x^\alpha \partial_v^\beta (I-P) g^{n+1} \rangle_{L^2_{x,v}} }_{S_6} \\
	    & + \underbrace{ \tfrac{1}{\eps} \langle \partial_x^\alpha \partial_v^\beta P ( v \cdot \nabla_x (I-P) g^{n+1} ), \partial_x^\alpha \partial_v^\beta (I-P) g^{n+1} \rangle_{L^2_{x,v}} }_{S_7} \,,
	  \end{split}
	\end{equation}
	where Lemma \ref{Lmm-Coercivity-L} and $\partial^{\beta_1}_v v = 0$ for $|\beta_2| \geq 2$ are also used.
	
	Now we estimate terms $S_i (1 \leq i \leq 7)$ in \eqref{2.9}. The H\"older inequality and Lemma \ref{Lmm-CF-nu} reduce to
	\begin{equation}\label{S1}
	  \begin{split}
	    S_1 \lesssim & \tfrac{1}{\eps} \sum_{|\beta_1| = 1} \| \nabla_x \partial^\alpha_x \partial^{\beta - \beta_1}_v ( I - P ) g^{n+1} \|_{L^2_{x,v}} \| \partial^\alpha_x \partial^\beta_v ( I - P ) g^{n+1} \|_{L^2_{x,v}}  \\
	    \lesssim & \tfrac{1}{\eps} \| (I-P) g^{n+1} \|_{H^N_{x,v}} \| ( I-P ) g^{n+1} \|_{\widetilde{H}^N_{x,v}(\nu)} \,.
	  \end{split}
	\end{equation}
	For $S_2$, we divide it into four parts as follows:
	\begin{equation}\no
	  \begin{split}
	    S_2 = & \underbrace{ - \gamma \langle \nabla_x \phi^n \cdot \nabla_v \partial^\alpha_x \partial^\beta_v (I-P) g^{n+1}, \partial^\alpha_x \partial^\beta_v (I-P) g^{n+1} \rangle_{L^2_{x,v}} }_{ S_2^{(1)} } \\
	    & \underbrace{ - \gamma \sum_{\substack{ \alpha_1 \leq \alpha \\ 0 < |\alpha_1| \leq N-2 }} C_\alpha^{\alpha_1} \langle \nabla_x \partial^{\alpha_1}_x \phi^n \cdot \nabla_v \partial^{\alpha - \alpha_1}_x \partial^\beta_v (I-P) g^{n+1}, \partial^\alpha_x \partial^\beta_v (I-P) g^{n+1} \rangle_{L^2_{x,v}} }_{ S_2^{(2)} } \\
	    & \underbrace{ - \gamma \sum_{\substack{ \alpha_1 \leq \alpha \\ N-1 \leq |\alpha_1| \leq N }} C_\alpha^{\alpha_1} \langle \nabla_x \partial^{\alpha_1}_x \phi^n \cdot \nabla_v \partial^{\alpha - \alpha_1}_x \partial^\beta_v (I-P) g^{n+1}, \partial^\alpha_x \partial^\beta_v (I-P) g^{n+1} \rangle_{L^2_{x,v}} }_{ S_2^{(3)} } \\
	    & + \underbrace{ \gamma \langle \partial^\alpha_x \partial^\beta_v P ( \nabla_x \phi^n \cdot \nabla_v ( I-P ) g^{n+1} ) , \partial^\alpha_x \partial^\beta_v (I-P) g^{n+1} \rangle_{L^2_{x,v}} }_{ S_2^{(4)} }  \,.
	  \end{split}
	\end{equation}
	Using Lemma \ref{Lmm-CF-nu} and the Sobolev embedding $H^2_x \hookrightarrow L^\infty_x$, we can estimate that
	\begin{equation}\no
	  \begin{aligned}
	    S_2^{(1)} & = - \tfrac{\gamma}{2} \langle v \cdot \nabla_x \phi^n, | \partial_x^\alpha \partial_v^\beta (I-P) g^{n+1} |^2 \rangle_{L^2_{x,v}} \lesssim \| \nabla_x \phi^n \|_{L_x^\infty} \| \partial_x^\alpha \partial_v^\beta (I-P) g^{n+1} \|_{L^2_{x,v} (\nu)}^2 \\
	    & \lesssim \| \nabla_x \phi^n \|_{H^N_x} \| (I-P) g^{n+1} \|_{\widetilde{H}^N_{x,v}(\nu)}^2 \,,
	  \end{aligned}
	\end{equation}
	and
	\begin{equation}\no
	  \begin{split}
	    S_2^{(2)} & \lesssim \sum_{\substack{ \alpha_1 \leq \alpha \\ 0 < |\alpha_1| \leq N-2 }} \| \nabla_x \partial^{\alpha_1}_x \phi^n \|_{L^\infty_x} \| \nabla_v \partial^{\alpha - \alpha_1}_x \partial^\beta_v (I-P) g^{n+1} \|_{L^2_{x,v}} \| \partial^\alpha_x \partial^\beta_v (I-P) g^{n+1} \|_{L^2_{x,v}} \\
	    & \lesssim \| \nabla_x \phi^n \|_{H^N_x} \| ( I -P ) g^{n+1} \|^2_{\widetilde{H}^N_{x,v}} \,,
	  \end{split}
	\end{equation}
	and
	\begin{equation}\no
	  \begin{split}
	    S_2^{(3)} & \lesssim \sum_{\substack{ \alpha_1 \leq \alpha \\ N - 1 \leq |\alpha_1| \leq N }} \| \nabla_x \partial^{\alpha_1}_x \phi^n \|_{L^2_x} \| \nabla_v \partial^{\alpha - \alpha_1}_x \partial^\beta_v (I-P) g^{n+1} \|_{L^\infty_x L^2_v} \| \partial^\alpha_x \partial^\beta_v (I-P) g^{n+1} \|_{L^2_{x,v}} \\
	    & \lesssim \| \nabla_x \phi^n \|_{H^N_x} \| ( I -P ) g^{n+1} \|^2_{\widetilde{H}^N_{x,v}} \,,
	  \end{split}
	\end{equation}
	For the quantity $ S_2^{(4)} $, by using the fact $\| \partial^\beta_v P g \|_{L^2_{x,v}} \lesssim \| g \|_{L^2_{x,v}}$, we imply that
	\begin{equation}\no
	  \begin{aligned}
	    S_2^{(4)} \lesssim & \| \partial^\alpha_x ( \nabla_x \phi^n \cdot \nabla_v ( I-P ) g^{n+1} ) \|_{L^2_{x,v}} \| \partial^\alpha_x \partial^\beta_v g^{n+1} \|_{L^2_{x,v}} \\
	    \lesssim & \| \nabla_x \phi^n \|_{H^N_x} \| (I-P) g^{n+1} \|^2_{\widetilde{H}^N_{x,v}} \,.
	  \end{aligned}
	\end{equation}
	In summary, we have
	\begin{equation}\label{S2}
	  \begin{aligned}
	    S_2 = S_2^{(1)} + S_2^{(2)} + S_2^{(3)} + S_2^{(4)} \lesssim \| \nabla_x \phi^n \|_{H^N_x} \| (I-P) g^{n+1} \|^2_{\widetilde{H}^N_{x,v}} \,.
	  \end{aligned}
	\end{equation}
	Similarly in estimates of $S_2$, $S_3$ is bounded by
	\begin{equation}\label{S3}
	    S_3 \lesssim \| \nabla_x \phi^n \|_{H^N_x} \big( \| (I-P) g^{n+1} \|^2_{\widetilde{H}^N_{x,v} (\nu)} + \| (I-P) g^{n+1} \|^2_{H^N_x L^2_v (\nu)} \big) \,.
	\end{equation}
	
	Next, we estimate the nonlinear collision term $S_4$. Applying Lemma \ref{Lmm-Q} with $ g_1 = \partial_x^{\alpha_1} g^n $, $ g_2 = \partial_x^{\alpha-\alpha_1} g^n$ and $ g_3 =  \partial_x^\alpha \partial_v^\beta (I-P) g^{n+1}$ implies that
	\begin{equation}\label{S4}
	  \begin{aligned}
	    S_4 = & \tfrac{1}{\eps} \sum_{\alpha_1 \leq \alpha} C_\alpha^{\alpha_1} \langle \partial^\beta_v Q ( \partial^{\alpha_1}_x g^n , \partial^{\alpha - \alpha_1}_x g^n ) , \partial^\alpha_x \partial^\beta_v (I-P) g^{n+1} \rangle_{L^2_{x,v}} \\
	    \lesssim & \tfrac{1}{\eps} \sum_{\substack{ \alpha_1 \leq \alpha \\ \beta_1 + \beta_2 \leq \beta }} \int_{\T^3} \| \partial^{\alpha_1}_x \partial^{\beta_1}_v g^n \|_{L^2_v (\nu)} \| \partial^{\alpha - \alpha_1}_x \partial^{\beta_2}_v g^n \|_{L^2_v} \| \partial_x^\alpha \partial_v^\beta (I-P) g^{n+1} \|_{L^2_v (\nu)} \d x \\
	    \lesssim & \tfrac{1}{\eps} \| g^n \|_{H^N_{x,v}} \| g^n \|_{H^N_{x,v} (\nu)} \| (I-P) g^{n+1} \|_{\widetilde{H}^N_{x,v} (\nu)} \\
	    \lesssim & \tfrac{1}{\eps} \| g^n \|_{H^N_{x,v}} \big( \| g^n \|_{H^N_x L^2_v} + \| (I-P) g^n \|_{H^N_{x,v} (\nu)} \big) \| (I-P) g^{n+1} \|_{\widetilde{H}^N_{x,v} (\nu)} \,.
	  \end{aligned}
	\end{equation}
	
	Direct calculation yields that
	\begin{equation}\no
	  \begin{aligned}
	     & ( I - P ) ( v \cdot \nabla_x \phi^n P g^{n+1} ) \\
	     & = A(v) : \nabla_x \phi^n \otimes \langle g^{n+1} , v \rangle_{L^2_v} + B(v) \cdot \nabla_x \phi^n \langle g^{n+1} , \tfrac{|v|^2}{3} - 1 \rangle_{L^2_v} \,.
	  \end{aligned}
	\end{equation}
	Then the quantity $S_5$ can be estimated that
	\begin{equation}\label{S5}
	  \begin{aligned}
	    S_5 = & \gamma \l \partial^\alpha_x [ \partial^\beta_v A(v) : \nabla_x \phi^n \otimes \l g^{n+1} , v \r_{L^2_v} \\
	    & \qquad \qquad + \partial^\beta_v B(v) \cdot \nabla_x \phi^n \l g^{n+1}, \tfrac{|v|^2}{3} - 1 \r_{L^2_v} ] , \partial^\alpha_x \partial^\beta_v (I-P) g^{n+1} \r_{x,v} \\
	    \lesssim & \| \nabla_x \phi^n \|_{H^N_x} \| g^{n+1} \|_{H^N_x L^2_v} \| \partial^\alpha_x \partial^\beta_v (I-P) g^{n+1} \|_{L^2_{x,v}} \\
	    \lesssim & \| \nabla_x \phi^n \|_{H^N_x} \| g^{n+1} \|_{H^N_x L^2_v} \| (I-P) g^{n+1} \|_{\widetilde{H}^N_{x,v}} \,.
	  \end{aligned}
	\end{equation}
	Notice that \cite{BGL1}
	\begin{equation}\no
	  \begin{aligned}
	    ( I - P ) ( v \cdot \nabla_x P g^{n+1} ) = A(v) : \nabla_x \langle g^{n+1} , v \rangle_{L^2_v} + B (v) \cdot \nabla_x \langle g^{n+1} , \tfrac{|v|^2}{3} - 1 \rangle_{L^2_v} \,.
	  \end{aligned}
	\end{equation}
	Then the term $S_6$ can be estimated by
	\begin{equation}\label{S6}
	  \begin{aligned}
	    S_6 = & \tfrac{1}{\eps} \big\langle \partial^\beta_v A(v) : \nabla_x \partial^\alpha_x \langle g^{n+1} , v \rangle_{L^2_v} + \partial^\beta_v B(v) \cdot \nabla_x \partial^\alpha_x \langle g^{n+1} , \tfrac{|v|^2}{3} - 1 \rangle_{L^2_v} , \partial^\alpha_x \partial^\beta_v (I-P) g^{n+1} \big\rangle_{L^2_{x,v}} \\
	    \lesssim & \tfrac{1}{\eps} \big( \| \partial^\beta_v A(v) \|_{L^2_v} + \| \partial^\beta_v B(v) \|_{L^2_v} \big) \| \nabla_x \partial^\alpha_x g^{n+1} \|_{L^2_{x,v}} \| \partial^\alpha_x \partial^\beta_v (I-P) g^{n+1} \|_{L^2_{x,v}} \\
	    \lesssim & \tfrac{1}{\eps} \| g^{n+1} \|_{H^N_x L^2_v} \| (I-P) g^{n+1} \|_{\widetilde{H}^N_{x,v} (\nu)} \,.
	  \end{aligned}
	\end{equation}
	
	One easily knows that $ \| \partial^\beta_v P ( p (v) g  ) \|_{L^2_v} \lesssim \| g \|_{L^2_v} $ holds for all $\beta \in \mathbb{N}^3$ and any polynomial $p(v)$. We thereby deduce from Lemma \ref{Lmm-CF-nu} that
	\begin{equation}\label{S7}
	  \begin{aligned}
	    S_7 \lesssim & \tfrac{1}{\eps} \| \partial^\alpha_x \partial^\beta_v P ( v \cdot \nabla_x (I-P) g^{n+1} ) \|_{L^2_{x,v}} \| \partial^\alpha_x \partial^\beta_v (I-P) g^{n+1} \|_{L^2_{x,v}} \\
	    \lesssim & \tfrac{1}{\eps} \| \partial^\alpha_x \nabla_x (I-P) g^{n+1} \|_{L^2_{x,v}} \| \partial^\alpha_x \partial^\beta_v (I-P) g^{n+1} \|_{L^2_{x,v} (\nu)} \\
	    \lesssim & \tfrac{1}{\eps} \| (I-P) g^{n+1} \|_{H^N_x L^2_v (\nu)} \| (I-P) g^{n+1} \|_{\widetilde{H}^N_{x,v}(\nu)} \,.
	  \end{aligned}
	\end{equation}
	
	From substituting the estimates \eqref{S1}, \eqref{S2}, \eqref{S3}, \eqref{S4}, \eqref{S5}, \eqref{S6} and \eqref{S7} into \eqref{2.9}, we deduce that
	\begin{equation}\label{2.1.1}
	  \begin{split}
	    & \tfrac{1}{2} \tfrac{\d}{\d t} \| \partial_x^\alpha \partial_v^\beta (I-P) g^{n+1} \|^2_{L^2_{x,v}} + \tfrac{\delta_1}{\epsilon^2} \| \partial_x^\alpha \partial_v^\beta (I-P) g^{n+1} \|_{L^2_{x,v} (\nu)}^2 \\
	    & \lesssim \tfrac{1}{\eps^2} \sum_{\beta' < \beta} \| \partial^\alpha_x \partial^{\beta'}_v (I-P) g^{n+1} \|^2_{L^2_{x,v} (\nu)} + \tfrac{1}{\eps} \| (I-P) g^{n+1} \|_{H^N_x L^2_v (\nu)} \| (I-P) g^{n+1} \|_{\widetilde{H}^N_{x,v}(\nu)} \\
	    & + \| \nabla_x \phi^n \|_{H^N_x} \big( \| (I-P) g^{n+1} \|^2_{H^N_{x,v}} + \| (I-P) g^{n+1} \|^2_{H^N_{x,v} (\nu)} \big) \\
	    & + \tfrac{1}{\eps} \big( \| g^n \|^2_{H^N_{x,v}} + \| g^n \|_{H^N_{x,v}} \| (I-P) g^n \|_{H^N_{x,v}(\nu)} \big) \| (I-P) g^{n+1} \|_{\widetilde{H}^N_{x,v} (\nu)} \,.
	  \end{split}
	\end{equation}
	Combining \eqref{Iter-Spatial-Bnd} and \eqref{2.1.1}, we easily derive from the induction for $0 \leq |\beta| = k \leq N$ that there are positive constants $a_{|\beta|}$, $b_k$ and $c_k$ such that
	\begin{equation}\label{2.20}
	  \begin{split}
	    & \tfrac{1}{2} \tfrac{\d}{\d t} \bigg\{ \| g^{n+1} \|^2_{H^N_x L^2_v} + \| \nabla_x \phi^{n+1} \|^2_{H^N_x} + \sum_{\substack{ |\alpha| + |\beta| \leq N \\ 1 \leq |\beta| \leq k }} c_{|\beta|}  \| \partial_x^\alpha \partial_v^\beta (I-P) g^{n+1} \|^2_{L^2_{x,v}} \bigg\} \\
	    & + \tfrac{b_k}{\epsilon^2} \| (I-P) g^{n+1} \|^2_{H^N_x L^2_v (\nu)} + \tfrac{c_k}{\eps^2} \sum_{\substack{ |\alpha| + |\beta| \leq N \\ 1 \leq |\beta| \leq k }} \| \partial_x^\alpha \partial_v^\beta (I-P) g^{n+1} \|^2_{L^2_{x,v} (\nu)} \\
	    & \lesssim  \| \nabla_x \phi^n \|_{H^N_x} \big( \| g^{n+1} \|^2_{H^N_{x,v}} + \| (I-P) g^{n+1} \|^2_{H^N_{x,v} (\nu)} \big) \\
	    & + \tfrac{1}{\eps} \big( \| g^n \|^2_{H^N_{x,v}} + \| g^n \|_{H^N_{x,v}} \| (I-P) g^n \|_{H^N_{x,v}(\nu)} \big) \| (I-P) g^{n+1} \|_{H^N_{x,v} (\nu)}
	  \end{split}
	\end{equation}
	for all $0 \leq |\beta| = k \leq N$.
	
	We define the so-called instant iterating energy $\mathscr{E}_{N,\eps} ( g, \phi )$
	\begin{equation}
	  \begin{split}
	    \mathscr{E}_{N,\eps} (g, \phi) = & \| g \|^2_{H^N_x L^2_v} + \| \nabla_x \phi \|^2_{H^N_x} + a_N \| (I-P) g \|^2_{\widetilde{H}^N_{x,v}} \\
	    + & \tfrac{2}{\eps^2} \int_0^t \big( b_N \| (I-P) g \|^2_{H^N_x L^2_v (\nu)} + c_N \| (I-P) g \|^2_{\widetilde{H}^N_{x,v} (\nu)} \big) \d \tau \,,
	  \end{split}
	\end{equation}
	and the initial energy $\mathscr{E}_{N} ( g_0, \phi_0 )$ is given as
	\begin{equation}
	  \begin{aligned}
	    \mathscr{E}_{N} ( g_0, \phi_0 ) = \| g_0 \|^2_{H^N_x L^2_v} + \| \nabla_x \phi_0 \|^2_{H^N_x} + a_N \| (I-P) g_0 \|^2_{\widetilde{H}^N_{x,v}} \,.
	  \end{aligned}
	\end{equation}
	It is obvious that
	\begin{equation*}
	  \begin{aligned}
	    \mathscr{E}_{N, \eps} (g, \phi) \thicksim \mathcal{E}_N (g, \phi) + \tfrac{1}{\eps^2} \int_0^t \| (I-P) g \|^2_{H^N_{x,v} (\nu)} \d \tau \,.
	  \end{aligned}
	\end{equation*}
	
	Now we claim that there are small $\ell, T \in (0, 1]$, independent of $\eps$, such that if $0 < T \leq 1 $, $ \mathscr{E}_N (g_{\eps,0}, \phi_{\eps,0}) \leq C_\# \mathcal{E}_N (g_{\eps,0}, \phi_{\eps,0}) \leq C_\# \delta \leq \ell $ for some constant $C_\# > 0$ and
	$$\sup_{0 \leq t \leq T} \mathscr{E}_{N,\eps} (g^n (t), \phi^n (t)) \d t \leq 2 \ell \,, $$
	then
	\begin{equation}\label{Claim-Unif-Bnd}
	  \begin{aligned}
	    \sup_{0 \leq t \leq T} \mathscr{E}_{N,\eps} (g^{n+1} (t), \phi^{n+1} (t)) \leq 2 \ell \,.
	  \end{aligned}
	\end{equation}
	
	Indeed, from taking $k=N$ in the inequality \eqref{2.20} and integrating over $[0,t]$, we derive that for all $t \in [0, T]$
	\begin{equation}\no
	  \begin{aligned}
	    \mathscr{E}_{N, \eps} (g^{n+1} (t), \phi^{n+1} (t) )  \lesssim & \mathscr{E}_N (g_{\eps, 0} , \phi_{\eps, 0}) + ( 1 + t ) \sup_{0 \leq t \leq T} \mathscr{E}_{N,\eps}^\frac{1}{2} ( g^n (t) , \phi^n (t) ) \mathscr{E}_{N,\eps} (g^{n+1} (t) , \phi^{n+1} (t) ) \\
	    + & ( 1 + \sqrt{t} ) \sup_{0 \leq t \leq T} \mathscr{E}_{N,\eps} (g^n (t) , \phi^n (t)) \mathscr{E}_{N,\eps}^\frac{1}{2} (g^{n+1} (t) , \phi^{n+1} (t) ) \,,
	  \end{aligned}
	\end{equation}
	which immediately implies that for some positive constant $C$
	\begin{equation}
	  \begin{aligned}
	    \big[ \tfrac{3}{4} - C ( 1 + T ) \sqrt{2 \ell} \big] \sup_{0 \leq t \leq T} \mathscr{E}_{N, \eps} ( g^{n+1} (t) , \phi^{n+1} (t) ) \leq \ell + 4 C ( 1 + T ) \ell^2 \,.
	  \end{aligned}
	\end{equation}
	We first take $T \in (0, 1)$ and $\ell \in (0,1)$ such that $ \tfrac{3}{4} - 2 C \sqrt{2 \ell} \geq \tfrac{5}{8} $, hence $0 < \ell < \min \{ 1, \tfrac{1}{512 C^2} \}$. Thus, we have
	\begin{equation}
	  \begin{aligned}
	    \sup_{0 \leq t \leq T} \mathscr{E}_{N,\eps} ( g^{n+1} (t) , \phi^{n+1} (t) ) \leq \tfrac{8}{5} \ell + \tfrac{64 C}{5} \ell^2
	  \end{aligned}
	\end{equation}
	We further restrict $0 < \ell < \min \{ 1, \tfrac{1}{512 C^2} , \tfrac{1}{32 C} \}$ such that $ 64 C \ell \leq 2 $. Then we have
	\begin{equation}
	  \begin{aligned}
	    \sup_{0 \leq t \leq T} \mathscr{E}_{N,\eps} ( g^{n+1} (t) , \phi^{n+1} (t) ) \leq \tfrac{8}{5} \ell + \tfrac{2}{5} \ell = 2 \ell \,,
	  \end{aligned}
	\end{equation}
	and the claim \eqref{Claim-Unif-Bnd} holds.
	
	Consequently, by employing the standard compactness arguments and taking $n \rightarrow \infty$, we can obtain a solution $( g_\eps , \phi_\eps )$ to \eqref{VPB-g}-\eqref{IC-g} for any fixed $0 < \eps \leq 1$. Moreover, the uniform bound \eqref{Claim-Unif-Bnd} implies the energy bound \eqref{2.3}. Thus, the proof of Lemma \ref{Lmm-Local} is completed.
\end{proof}

\section{Uniform estimates and global solutions}\label{Sec: Uniform-Bnd}

In this section, our goal is to globally extend the local solutions to \eqref{VPB-g}-\eqref{IC-g} constructed in Lemma \ref{Lmm-Local} by deriving an energy estimates uniformly in $\eps \in (0, 1)$. For notational simplicity, we drop the lower index $\eps$ of $g_\eps $ and $\phi_\eps$ in the perturbed VPB system \eqref{VPB-g}, i.e.,
\begin{equation}\label{VPB-g-drop-eps}
  \left\{
    \begin{array}{l}
      \partial_t g + \tfrac{1}{\epsilon} v \cdot \nabla_x g + \tfrac{1}{\epsilon^2} L g - \tfrac{\gamma}{\epsilon} v \cdot \nabla_x \phi - \gamma v \cdot \nabla_x \phi g + \gamma \nabla_x \phi\cdot \nabla_v g = \tfrac{1}{\epsilon} Q ( g , g ) \,, \\
      \Delta_x \phi = \gamma \langle g , 1 \rangle_{L^2_v} \,.
    \end{array}
  \right.
\end{equation}

\subsection{Pure spatial derivative estimates: kinetic dissipations}

In this subsection, we will consider the energy estimates on the pure spatial derivative of $g$. The following lemma will be established.

\begin{lemma}\label{Lmm-Spatial}
	Assume that $g(t,x,v)$ is the solution to the perturbed VMB system \eqref{VPB-g}-\eqref{IC-g} constructed in Lemma \ref{Lmm-Local}. Then there are constants $\delta$ and $C$ such that
	\begin{equation}\label{Spatial-Bnd}
	  \begin{split}
	    \tfrac{1}{2} \tfrac{\d}{\d t} \Big( \| g \|^2_{H^N_x L^2_v} + \| \nabla_x \phi \|^2_{H^N_x} \Big) + \tfrac{\delta}{\eps^2} \| (I-P) g \|^2_{H^N_x L^2_v (\nu)} \lesssim \mathcal{E}_N^\frac{1}{2} (g, \phi)  \mathcal{D}_{N,\eps} (g)
	  \end{split}
	\end{equation}
	for all $0 < \eps \leq 1$.
\end{lemma}

\begin{proof}[Proof of Lemma \ref{Lmm-Spatial}]
	For any $\alpha \in \mathbb{N}^3$ with $ | \alpha | \leq N $, applying the derivative operator $ \partial_x^\alpha $ to the first $g$-equation \eqref{VPB-g-drop-eps} yields
	\begin{equation}
	  \begin{split}
	    & \partial_t \partial_x^\alpha g  + \tfrac{1}{\eps} v \cdot \nabla_x ( \partial_x^\alpha g ) + \tfrac{1}{\epsilon^2} L \partial_x^\alpha g - \tfrac{\gamma}{\epsilon} v \cdot \nabla_x \partial_x^\alpha \phi \\
	    & = \tfrac{1}{\epsilon} \partial_x^\alpha Q ( g , g ) + \gamma v \cdot \nabla_x \phi \partial_x^\alpha g - \gamma \nabla_x \phi \cdot \nabla_v \partial_x^\alpha g \\
	    & + \sum_{ 0 \neq \alpha_1 \leq \alpha} C_\alpha^{\alpha_1} ( \gamma v \cdot \nabla_x \partial_x^{\alpha_1} \phi \partial_x^{\alpha - \alpha_1 } g - \gamma \nabla_x \partial_x^{\alpha_1} \phi \cdot \nabla_v \partial_x^{\alpha - \alpha_1} g ) \,.
	  \end{split}
	\end{equation}
	Then, we take $L^2_{x,v}$-inner product in the above equation by dot with $ \partial_x^\alpha g $ and integrate by parts over $\R^3 \times \R^3$. We thereby have
	\begin{equation}\label{3.2.1}
	  \begin{aligned}
	     & \tfrac{1}{2} \tfrac{\d}{\d t} \| \partial_x^\alpha g \| + \tfrac{1}{\epsilon^2} \langle L \partial_x^\alpha g , \partial_x^\alpha g \rangle_{L^2_{x,v}} - \tfrac{\gamma}{\epsilon} \langle v \cdot \nabla_x \partial_x^\alpha \phi , \partial_x^\alpha g \rangle_{L^2_{x,v}} \\
	     & = \underbrace{ \tfrac{1}{\epsilon} \langle \partial_x^\alpha Q ( g , g ) , \partial_x^\alpha g \rangle_{L^2_{x,v}} }_{I_1} + \underbrace{ \gamma \langle v \cdot \nabla_x \phi \partial_x^\alpha g - \nabla_x \phi \nabla_v \partial_x^\alpha g , \partial_x^\alpha g \rangle_{L^2_{x,v}} }_{I_2} \\
	     & + \underbrace{ \gamma \sum_{0 \neq \alpha_1 \leq \alpha} C_{\alpha}^{\alpha_1} \langle  v \cdot \nabla_x \partial_x^{\alpha_1} \phi \partial_x^{\alpha-\alpha_1} g -  \nabla_x \partial_x^{\alpha_1} \phi \cdot \nabla_v \partial_x^{\alpha-\alpha_1} g , \partial_x^\alpha g \rangle_{L^2_{x,v}} }_{I_3} \,.
	  \end{aligned}
	\end{equation}
	First, Lemma \ref{Lmm-Coercivity-L} tells us that there is a constant $\delta > 0$ such that
	\begin{equation}
	  \begin{aligned}
	    \langle L \partial_x^\alpha g , \partial_x^\alpha g \rangle_{L^2_{x,v}} \geq \delta \| (I-P) \partial_x^\alpha g \|_{L^2_{x,v} (\nu)}^2 \,.
	  \end{aligned}
	\end{equation}
	Noticing that the local conservation law of mass is
	\begin{equation}
	  \begin{aligned}
	    \partial_t \l g, 1 \r_{L^2_v} + \tfrac{1}{\eps} \nabla_x \cdot \l g , v \r_{L^2_v} = 0 \,,
	  \end{aligned}
	\end{equation}
	we derive from the Poisson equation of $\phi$ in \eqref{VPB-g-drop-eps}, i.e., $ \Delta_x \phi = \gamma \l g , 1 \r_{L^2_v} $, that
	\begin{equation}
	  \begin{aligned}
	    & - \tfrac{\gamma}{\epsilon} \langle v \cdot \nabla_x \partial_x^\alpha \phi , \partial_x^\alpha g \rangle_{L^2_{x,v}} = \tfrac{\gamma}{\eps} \l \partial^\alpha_x \phi , \partial^\alpha_x \nabla_x \cdot \l g , v \r_{L^2_v} \r_{L^2_x} \\
	    & = - \gamma \l \partial^\alpha_x \phi , \partial_t \partial^\alpha_x \l g , 1 \r_{L^2_v} \r_{L^2_x} = - \l \partial^\alpha_x \phi , \partial_t \partial^\alpha_x \Delta_x \phi \r_{L^2_x} \\
	    & = \tfrac{1}{2} \tfrac{\d}{\d t} \| \nabla_x \partial^\alpha_x \phi \|^2_{L^2_x} \,.
	  \end{aligned}
	\end{equation}
	Consequently, we arrive at
	\begin{equation}\label{3.2.2}
	  \begin{aligned}
	    \tfrac{1}{2} \tfrac{\d}{\d t} ( \| \partial_x^\alpha g \|^2_{L^2_{x,v}} + \| \nabla_x \partial_x^\alpha \phi\|^2_{L^2_x} ) + \tfrac{\delta}{\epsilon^2} \| (I-P) \partial_x^\alpha g \|^2_{L^2_{x,v} (\nu)} \leq I_1 + I_2 + I_3
	  \end{aligned}
	\end{equation}
	for all $\alpha \in \mathbb{N}^3$ with $|\alpha| \leq N$.
	
	Next, let us estimate the terms $I_1$, $I_2$ and $I_3$ on the right-hand side of \eqref{3.2.1}. We split $g = P g + (I-P) g$, so that $ I_1 $ is decomposed into
	\begin{equation}
	  \begin{aligned}
	    I_1 \leq & \underbrace{ \tfrac{1}{\epsilon} \langle \partial_x^\alpha Q ( P g , P g ) , \partial_x^\alpha (I-P) g \rangle_{L^2_{x,v}} }_{I_{11}} + \underbrace{ \tfrac{1}{\epsilon} \langle \partial_x^\alpha Q ( P g , (I-P) g ) , \partial_x^\alpha (I-P) g \rangle_{L^2_{x,v}} }_{I_{12}} \\
	    + & \underbrace{ \tfrac{1}{\epsilon} \langle \partial_x^\alpha Q ( (I-P) g , P g ) , \partial_x^\alpha (I-P) g \rangle_{L^2_{x,v}} }_{I_{13}} + \underbrace{ \tfrac{1}{\epsilon} \langle \partial_x^\alpha Q ( (I-P) g , (I-P) g ) , \partial_x^\alpha (I-P) g \rangle_{L^2_{x,v}} }_{I_{14}} \,,
	  \end{aligned}
	\end{equation}
	where we also utilize the fact $Q (g, h) \in \mathcal{N}^\perp$. On the other hand, we plug $P g = a + b \cdot v + c |v|^2$ into the first term $I_{11}$ to get
	\begin{equation}
	  \begin{aligned}
	    & |I_{11}| \lesssim \tfrac{1}{\eps} \| \partial^\alpha_x Q ( P g, P g ) \|_{L^2_{x,v}} \| \partial_x^\alpha (I-P) g \|_{L^2_{x,v}} \\
	    \lesssim & \sum_{\alpha_1 \leq \alpha} \| \, | \partial^{\alpha_1}_x a | \, | \partial^{\alpha - \alpha_1}_x a| + | \partial^{\alpha_1}_x b | \, | \partial^{\alpha - \alpha_1}_x b| + | \partial^{\alpha_1}_x c | \, | \partial^{\alpha - \alpha_1}_x c| \|_{L^2_x} \| \partial_x^\alpha (I-P) g \|_{L^2_{x,v} (\nu)} \,.
	  \end{aligned}
	\end{equation}
	If $\alpha \neq 0$, the first factor is bounded by the calculus inequality
	\begin{equation*}
	  \begin{aligned}
	    \Big( \| \nabla_x a \|_{H^{N-1}_x} + \| \nabla_x b \|_{H^{N-1}_x} + \| \nabla_x c \|_{H^{N-1}_x} \Big) \Big( \| a \|_{H^N_x} + \| b \|_{H^N_x} + \| c \|_{H^N_x}  \Big) \\
	    \lesssim \| \nabla_x P g \|_{H^{N-1}_x L^2_v} \| P g \|_{H^N_x L^2_v} \lesssim \| \nabla_x P g \|_{H^{N-1}_x L^2_v} \| g \|_{H^N_x L^2_v} \,.
	  \end{aligned}
	\end{equation*}
	If $\alpha = 0$, the first factor is bounded by the H\"older inequality and the Sobolev inequality given in Lemma \ref{Lm-Sobolev-Inq}
	\begin{equation*}
	  \begin{aligned}
	    \| a^2 + |b|^2 + c^2 \|_{L^2_x} \lesssim & \big( \| a \|_{L^6_x} + \| b \|_{L^6_x} + \| c \|_{L^6_x} \big) \big( \| a \|_{L^3_x} + \| b \|_{L^3_x} + \| c \|_{L^3_x} \big) \\
	    \lesssim & \big( \| \nabla_x a \|_{L^2_x} + \| \nabla_x b \|_{L^2_x} + \| \nabla_x c \|_{L^2_x} \big) \big( \| a \|_{H^1_x} + \| b \|_{H^1_x} + \| c \|_{H^1_x} \big) \\
	    \lesssim & \| \nabla_x P g \|_{H^{N-1}_x L^2_v} \| g \|_{H^N_x L^2_v}
	  \end{aligned}
	\end{equation*}
	In summary, we get
	\begin{equation}\label{I11}
	  \begin{aligned}
	    | I_{11} | \lesssim \tfrac{1}{\eps} \| \nabla_x P g \|_{H^{N-1}_x L^2_v} \| g \|_{H^N_x L^2_v} \| \partial_x^\alpha (I-P) g \|_{L^2_{x,v} (\nu)} \lesssim \mathcal{E}_N^\frac{1}{2} (g, \phi) \mathcal{D}_{N, \eps} (g) \,,
	  \end{aligned}
	\end{equation}
	where the functionals $ \mathcal{E}_N (g, \phi) $ and $ \mathcal{D}_{N, \eps} (g) $ are defined in \eqref{Energy-E} and \eqref{Energy-D}, respectively. Furthermore, by employing the similar arguments in estimates \eqref{S4}, we easily have
	\begin{equation}\no
	  \begin{aligned}
	    | I_{12} | + |I_{13}| + |I_{14}| \lesssim & \tfrac{1}{\eps} \| g \|_{H^N_x L^2_v} \| (I-P) g \|^2_{H^N_x L^2_v (\nu)} \lesssim \eps \mathcal{E}_N^\frac{1}{2} (g, \phi) \mathcal{D}_{N, \eps} (g) \,.
	  \end{aligned}
	\end{equation}
	Consequently, we estimate that
	\begin{equation}\label{I-1}
	  \begin{aligned}
	    I_1 \leq |I_{11}| + | I_{12} | + |I_{13}| + |I_{14}| \lesssim \mathcal{E}_N^\frac{1}{2} (g, \phi) \mathcal{D}_{N, \eps} (g)
	  \end{aligned}
	\end{equation}
	for all $0 < \eps \leq 1$.
	One notices that integration by parts over $v \in \R^3$ and the decomposition $g = P g + (I-P) g$ imply
	\begin{equation}
	  \begin{aligned}
	    I_2 = & \tfrac{\gamma}{2} \l v \cdot \nabla_x \phi , |\partial^\alpha_x g|^2 \r_{L^2_{x,v}} \\
	    = & \underbrace{ \tfrac{\gamma}{2} \l v \cdot \nabla_x \phi , |\partial^\alpha_x P g|^2 \r_{L^2_{x,v}} }_{I_{21}} + \underbrace{ \tfrac{\gamma}{2} \l v \cdot \nabla_x \phi , |\partial^\alpha_x (I-P) g|^2 \r_{L^2_{x,v}} }_{I_{22}} \\
	    & + \underbrace{ \gamma \l v \cdot \nabla_x \phi , \partial^\alpha_x P g \partial^\alpha_x (I-P) g \r_{L^2_{x,v}} }_{I_{23}} \,.
	  \end{aligned}
	\end{equation}
	From Lemma \ref{Lmm-CF-nu} and the Sobolev embedding $H^2_x \hookrightarrow L^\infty_x$, we derive that
	\begin{equation}\label{I22}
	  \begin{aligned}
	    I_{22} \lesssim \| \nabla_x \phi \|_{L^\infty_x} \| \partial^\alpha_x (I-P) g \|^2_{L^2_{x,v} (\nu)} \lesssim \mathcal{E}_N^\frac{1}{2} (g, \phi) \mathcal{D}_{N, \eps} (g) \,.
	  \end{aligned}
	\end{equation}
	If $\alpha = 0$, the term $I_{23}$ is bounded by
	\begin{equation*}
	  \begin{aligned}
	    \| \nabla_x \phi \|_{L^3_x} \| v P g \|_{L^6_x L^2_v} \| (I-P) g \|_{L^2_{x,v}} \lesssim \| \nabla_x \phi \|_{H^1_x} \| \nabla_x ( a, b, c ) \|_{L^2_x} \| (I-P) g \|_{L^2_{x,v} (\nu)} \\
	    \lesssim \| \nabla_x \phi \|_{H^N_x} \| \nabla_x P g \|_{H^{N-1}_x L^2_v} \| (I-P) g \|_{H^N_x L^2_v (\nu)} \lesssim \mathcal{E}_N^\frac{1}{2} (g, \phi) \mathcal{D}_{N, \eps} (g) \,.
	  \end{aligned}
	\end{equation*}
	If $\alpha \neq 0$, the term $I_{23}$ is bounded by
	\begin{equation}\no
	  \begin{aligned}
	    & \| \nabla_x \phi \|_{L^\infty_x} \| v \partial^\alpha_x P g \|_{L^2_{x,v}} \| \partial^\alpha_x (I-P) g \|_{L^2_{x,v}} \\
	    & \lesssim \| \nabla_x \phi \|_{H^N_x} \| \nabla_x P g \|_{H^{N-1}_x L^2_v} \| (I-P) g \|_{H^N_x L^2_v (\nu)} \lesssim \mathcal{E}_N^\frac{1}{2} (g, \phi) \mathcal{D}_{N, \eps} (g) \,,
	  \end{aligned}
	\end{equation}
	where Lemma \ref{Lmm-CF-nu} and the Sobolev embedding $H^2_x \hookrightarrow L^\infty_x$ are also used. In summary, we have
	\begin{equation}\label{I23}
	  \begin{aligned}
	    I_{23} \lesssim \mathcal{E}_N^\frac{1}{2} (g, \phi) \mathcal{D}_{N, \eps} (g) \,.
	  \end{aligned}
	\end{equation}
	It remains to estimate the term $I_{21}$. If $\alpha \neq 0$, the analogous arguments in \eqref{I22} tells us that $I_{21}$ is bounded by
	\begin{equation}\no
	  \begin{aligned}
	    I_{21} \lesssim \| \nabla_x \phi \|_{H^N_x} \| \nabla_x P g \|_{H^{N-1}_x L^2_v}^2 \lesssim \mathcal{E}_N^\frac{1}{2} (g, \phi) \mathcal{D}_{N, \eps} (g) \,.
	  \end{aligned}
	\end{equation}
	If $\alpha = 0$, we use $P g = a + b \cdot v + c |v|^2$ and take integration with respect to $v \in \R^3$ to get
	\begin{equation*}
	  \begin{aligned}
	    I_{21} = & \tfrac{\gamma}{2} \l v \cdot \nabla_x \phi , ( a + b \cdot v + c |v|^2 )^2 \r_{L^2_{x,v}} = \gamma \l \nabla_x \phi \otimes b , ( a + c |v|^2 ) v \otimes v \r_{L^2_{x,v}} \\
	    = & \gamma \l \nabla_x \phi \otimes b , ( a + c |v|^2 ) ( A(v) + \tfrac{|v|^2}{3} I ) \r_{L^2_{x,v}}  = \gamma \l b \cdot \nabla_x \phi , a \l \tfrac{|v|^2}{3} , 1 \r_{L^2_v} + c \l \tfrac{|v|^4}{3} , 1 \r_{L^2_v} \r_{L^2_x} \\
	    = & \gamma \l \nabla_x \phi , b ( a + 5 c ) \r_{L^2_x} \,,
	  \end{aligned}
	\end{equation*}
	where we use $ \l a + c |v|^2 , A(v) \r_{L^2_v} = 0 $. Hence,
	\begin{equation}\label{I21-0}
	  \begin{aligned}
	    I_{21} = \gamma \l \nabla_x \phi , b ( a + 3 c ) \r_{L^2_x} + 2 \gamma \l \nabla_x \phi , b  c \r_{L^2_x} \,,
	  \end{aligned}
	\end{equation}
	where the first term is estimated by using Lemma \ref{Lm-Sobolev-Inq}
	\begin{equation}\label{I21-1}
	  \begin{aligned}
	    \gamma \l \nabla_x \phi , b ( a + 3 c ) \r_{L^2_x} \lesssim \| \nabla_x \phi \|_{L^3_x} \| b \|_{L^6_x} \| a + 3 c \|_{L^2_x} \lesssim \| \nabla_x \phi \|_{H^1_x} \| \nabla_x b \|_{L^2_x} \| a + 3 c \|_{L^2_x} \\
	    \lesssim \| \nabla_x \phi \|_{H^N_x} \| \nabla_x P g \|_{H^{N-1}_x L^2_v} \| \l g , 1 \r_{L^2_v} \|_{H^{N-1}_x} \lesssim \mathcal{E}_N^\frac{1}{2} (g, \phi) \mathcal{D}_{N, \eps} (g) \,,
	  \end{aligned}
	\end{equation}
	and the second term is bounded by
	\begin{equation}\label{I21-2}
	  \begin{aligned}
	    2 \gamma \l \nabla_x \phi , b  c \r_{L^2_x} \lesssim & \| \nabla_x \phi \|_{L^2_x} \| b \|_{L^6_x} \| c \|_{L^3} \lesssim \| \nabla_x \phi \|_{L^2_x} \| \nabla_x b \|_{L^2_x} \| c \|_{H^1} \\
	    \lesssim & \| c \|_{H^1_x} \| \nabla_x b \|_{L^2_x} \| a + 3c \|_{L^2_x} \\
	    \lesssim & \| P g \|_{H^N_x L^2_v} \| \l g , 1 \r_{L^2_v} \|_{H^{N-1}_x} \| \nabla_x P g \|_{H^{N-1}_x L^2_v} \\
	    \lesssim & \mathcal{E}_N^\frac{1}{2} (g, \phi) \mathcal{D}_{N, \eps} (g) \,.
	  \end{aligned}
	\end{equation}
	Here we make use of Lemma \ref{Lm-Sobolev-Inq} and the bound $\| \nabla_x \phi \|_{L^2_x} \lesssim \| \phi \|_{H^2_x} \lesssim \| a+3c \|_{L^2_x}$ derived from the Poisson equation $\Delta_x \phi = \gamma (a+3c)$. Plugging the estimates \eqref{I21-1} and \eqref{I21-2} into \eqref{I21-0} yields that $ I_{21} \lesssim \mathcal{E}_N^\frac{1}{2} (g, \phi) \mathcal{D}_{N, \eps} (g ) $ for the case $\alpha = 0$. In summary, we obtain that for all $|\alpha| \leq N$
	\begin{equation}\label{I21}
	  \begin{aligned}
	    I_{21} \lesssim \mathcal{E}_N^\frac{1}{2} (g, \phi) \mathcal{D}_{N, \eps} (g ) \,.
	  \end{aligned}
	\end{equation}
	Consequently, the bounds \eqref{I22}, \eqref{I23} and \eqref{I21} tell us
	\begin{equation}\label{I-2}
	  \begin{aligned}
	    I_2 \lesssim \mathcal{E}_N^\frac{1}{2} (g, \phi) \mathcal{D}_{N, \eps} (g ) \,.
	  \end{aligned}
	\end{equation}
	
	Notice that
	\begin{equation}\no
	  \begin{aligned}
	    I_3 = & \underbrace{ \gamma \sum_{0 \neq \alpha_1 < \alpha} C_\alpha^{\alpha_1} \l v \cdot \nabla_x \partial^{\alpha_1}_x \phi \partial^{\alpha - \alpha_1}_x g - \nabla_x \partial^{\alpha_1}_x \phi \cdot \nabla_v \partial^{\alpha - \alpha_1}_x g , \partial^\alpha_x g \r_{L^2_{x,v}} }_{I_{31}} \\
	    + & \underbrace{ \gamma \l v \cdot \nabla_x \partial^\alpha_x \phi g - \nabla_x \partial^\alpha_x \phi \cdot \nabla_v g , \partial^\alpha_x g \r_{L^2_{x,v}} }_{I_{32}} \,.
	  \end{aligned}
	\end{equation}
	It follows from Lemma \ref{Lm-Sobolev-Inq}, Lemma \ref{Lmm-CF-nu} and the decomposition $g = P g + (I-P) g$ that
	\begin{equation}\no
	  \begin{aligned}
	    I_{31} \lesssim & \sum_{0 \neq \alpha_1 < \alpha} \| \nabla_x \partial^{\alpha_1}_x \phi \|_{L^4_x} \| \partial^{\alpha-\alpha_1}_x g \|_{L^4_x L^2_v (\nu)} \| \partial^\alpha_x g \|_{L^2_{x,v} (\nu)} \\
	    + & \sum_{\alpha_1 < \alpha , |\alpha_1| = 1} \| \nabla_x \partial^{\alpha_1} \phi \|_{L^\infty_x} \| \nabla_v \partial^{\alpha-\alpha_1}_x g \|_{L^2_{x,v}} \| \partial^\alpha_x g \|_{L^2_{x,v}} \\
	    + & \sum_{\alpha_1 < \alpha , |\alpha_1| \geq 2} \| \nabla_x \partial^{\alpha_1} \phi \|_{L^4_x} \| \nabla_v \partial^{\alpha-\alpha_1}_x g \|_{L^4_x L^2_v} \| \partial^\alpha_x g \|_{L^2_{x,v}} \\
	    \lesssim & \| \nabla_x \phi \|_{H^N_x} \Big( \| \nabla_x g \|^2_{H^{N-1}_x L^2_v (\nu)} + \| \nabla_x \nabla_v g \|_{H^{N-2}_x L^2_v} \| \nabla_x g \|_{H^{N-1}_x L^2_v} \Big) \\
	    \lesssim & \| \nabla_x \phi \|_{H^N_x} \Big( \| (I-P) g \|^2_{H^N_{x,v} (\nu)} + \| \nabla_x P g \|^2_{H^{N-1}_x L^2_v} \Big) \\
	    \lesssim & \mathcal{E}_N^\frac{1}{2} (g, \phi) \mathcal{D}_{N,\eps} (g) \,,
	  \end{aligned}
	\end{equation}
	and
	  \begin{align*}
	    I_{32} \lesssim & \| \nabla_x \partial^\alpha_x \phi \|_{L^2_x} \| \partial^\alpha_x g \|_{L^2_{x,v} (\nu)} \Big( \| g \|_{L^\infty_x L^2_v (\nu)} + \| \nabla_v g \|_{L^\infty_x L^2_v} \Big) \\
	    \lesssim & \| \nabla_x \phi \|_{H^N_x} \| \nabla_x g \|_{H^{N-1}_x L^2_v (\nu)} \Big( \| P g \|_{L^\infty_x L^2_v} + \| (I-P) g \|_{H^2_{x,v} (\nu)} \Big) \\
	    \lesssim & \| \nabla_x \phi \|_{H^N_x} \| \nabla_x g \|_{H^{N-1}_x L^2_v (\nu)} \Big( \| \nabla_x P g \|_{H^1_x L^2_v} + \| (I-P) g \|_{H^2_{x,v} (\nu)} \Big) \\
	    \lesssim & \| \nabla_x \phi \|_{H^N_x} \Big( \| (I-P) g \|^2_{H^N_{x,v} (\nu)} + \| \nabla_x P g \|^2_{H^{N-1}_x L^2_v} \Big)\\
	    \lesssim & \mathcal{E}_N^\frac{1}{2} (g, \phi) \mathcal{D}_{N,\eps} (g) \,.
	  \end{align*}
	Consequently, we have
	\begin{equation}\label{I-3}
	  \begin{aligned}
	    I_3 = I_{31} + I_{32} \lesssim \mathcal{E}_N^\frac{1}{2} (g, \phi) \mathcal{D}_{N,\eps} (g) \,.
	  \end{aligned}
	\end{equation}
	We thereby know that plugging all estimates \eqref{I-1}, \eqref{I-2}, \eqref{I-3} into \eqref{3.2.2} and summing up for all $|\alpha| \leq N$ reduce to the bound \eqref{Spatial-Bnd}, and then complete the proof of Lemma \ref{Lmm-Spatial}.
\end{proof}

\subsection{Macroscopic energy estimates: fluid dissipations}

In this subsection, we will find a dissipative structure of the fluid part $P g$ by using the so-called micro-macro decomposition method for the VPB version, depending on the so-called {\em thirteen moments}, see \cite{GY-06} for instance. More precisely, one can obtain the evolution equation of each coefficient $(a,b,c)$ of $P g$, and the basis $\{e_k\}_{k=1}^{13} \subseteq L^2_v$ of the coefficients consisting of
\begin{equation}\label{e 13}
  \big\{ 1, v_i, v_i^2, v_i |v|^2 \ (i=1,2,3) ; v_i v_j \ (1 \leq i < j \leq 3) \big\} \subseteq L^2_v \,.
\end{equation}
Let  $\{e_k^*\}_{k=1}^{13}$  be a corresponding bi-orthogonal basis, i.e., a basis such that
\begin{equation*}
  \langle e_k^* , e_j^* \rangle_{L^2_v} = \delta_{jk} \,, \ j,k=1,2\cdots, 13 \,,
\end{equation*}
where $\delta_{jk} = 0$ if $j \neq k$ and $\delta_{jk} = 1$ if $j=k$. In fact, $e_k^*$ is a given linear combination of \eqref{e 13}. From \eqref{decomposition}, one thereby has the following macroscopic equations on coefficients $(a, b , c)$ of $P g$:
\begin{align}
  1 : & \qquad \epsilon \partial_t a =l_a + h_a + m_a \,, \label{Coef-Evl-a}\\
  v_i : & \qquad \epsilon \partial_t b_i + \partial_i a - \gamma \partial_i \phi = l_{bi} + h_{bi} + m_{bi} \label{Coef-Evl-b}\,, \\
  v_i^2 : & \qquad \epsilon \partial_t c + \partial_i b_i = l_i + h_i + m_i \,,\label{3.3.8} \\
  v_i v_j : & \qquad \partial_i b_j + \partial_j b_i = l_{ij} + h_{ij} + m_{ij} \quad (i \neq j) \,, \label{3.3.7} \\
  v_i |v|^2 : & \qquad \partial_i c = l_{ci} + h_{ci} + m_{ci} \,, \label{3.3.9}
\end{align}
where all terms on the right-hand side are the coefficients of $l$, $h$, $m$, defined in \eqref{l+h+m}, in terms of the basis \eqref{e 13}.

Our goal of this subsection is to devote ourselves finding a macroscopic dissipation. Similar to the case of the Boltzmann equation, the high order derivatives of the fluid coefficients $(a,b,c)$ are dissipative from the balance laws \eqref{1.7}, \eqref{Balance-b}, \eqref{1.9}, \eqref{Balance-c} and the coefficients equations of $a$, $b$, $c$ in \eqref{Coef-Evl-a}-\eqref{3.3.9}. However, $\l g , 1 \r_{L^2_v} = (a+3c)$ will be a further damping structure in VPB system, because of the Poisson equation \eqref{1.9}, namely, $\Delta_x \phi = \l g , 1 \r_{L^2_v} = (a+3c)$. More precisely, we will give the following lemma.

\begin{lemma}\label{Lmm-MM-Decomp}
	Assume $g (t,x,v)$ is the solution to the perturbed VPB system \eqref{VPB-g}-\eqref{IC-g} constructed in Lemma \ref{Lmm-Local}. Then
	\begin{equation}\label{MM-Inq}
	  \begin{split}
	    \eps \tfrac{\d}{\d t} E^{int}_N (g) + \| \nabla_x b \|^2_{H^{N-1}_x} + \| \nabla_x c \|^2_{H^{N-1}_x} + \| \nabla_x (a+3c) \|^2_{H^{N-1}_x} + \| a+3c \|^2_{H^{N-1}_x}  \\
	    \lesssim \tfrac{1}{\eps^2} \| (I-P) g \|^2_{H^N_x L^2_v (\nu)} + \mathcal{E}_N^\frac{1}{2} (g, \phi) \mathcal{D}_{N, \eps} (g)
	  \end{split}
	\end{equation}
	for all $0 < \eps \leq 1$ and $N \geq 3$, where the so-called interactive energy functional $E_N^{int} (g)$ is defined as
	\begin{equation}\label{Interactive-Energy}
	  \begin{aligned}
	    E_N^{int} (g) = \sum_{|\alpha| \leq N-1} \Bigg\{ & 32 \sum_{i=1}^3 \l \partial_i \partial^\alpha_x (I-P) g , \partial^\alpha_x c \zeta_i (v) \r_{L^2_{x,v}} - 2 \l \partial^\alpha_x \nabla_x \cdot b , \partial^\alpha_x (a+3c) \r_{L^2_x} \\
	    & + 32 \sum_{i,j=1}^3 \l \partial_j \partial^\alpha_x (I-P) g , \partial^\alpha_x c \zeta_{ij} (v) \r_{L^2_{x,v}} \Bigg\} \,.
	  \end{aligned}
	\end{equation}
	Here $\zeta_i (v)$ and $\zeta_{ij} (v)$ are some linear combinations of the basis \eqref{e 13}.
\end{lemma}

\begin{remark}\label{Rmk-MM}
	One notices that
	\begin{equation}\label{Interactive-Energy-Bnd}
	  \begin{aligned}
	    E_N^{int} (g) \lesssim \| g \|^2_{H^N_x L^2_v} \lesssim \mathcal{E}_N (g, \phi)
	  \end{aligned}
	\end{equation}
	and
	\begin{equation}\no
	  \begin{aligned}
	    \| \nabla_x b \|^2_{H^{N-1}_x} + \| \nabla_x c \|^2_{H^{N-1}_x} + \| \nabla_x (a+3c) \|^2_{H^{N-1}_x} + \| a+3c \|^2_{H^{N-1}_x} \\
	    \thicksim \| \nabla_x P g \|^2_{H^{N-1}_x L^2_v} + \| \l g , 1 \r_{L^2_v} \|^2_{H^{N-1}_x} \,.
	  \end{aligned}
	\end{equation}
	Then the inequality \eqref{MM-Inq} can be simplified as
	\begin{equation}\label{MM-Inq-Simple}
	  \begin{aligned}
	    c_0 \eps \tfrac{\d}{\d t} E^{int}_N (g) + \| \nabla_x P g \|^2_{H^{N-1}_x L^2_v} + \| \l g , 1 \r_{L^2_v} \|^2_{H^{N-1}_x} \\
	    \lesssim \tfrac{1}{\eps^2} \| (I-P) g \|^2_{H^N_x L^2_v (\nu)} + \mathcal{E}_N^\frac{1}{2} (g, \phi) \mathcal{D}_{N, \eps} (g)
	  \end{aligned}
	\end{equation}
	for some constant $c_0 > 0$, independent of $\eps$.
\end{remark}

\begin{proof}[Proof of Lemma \ref{Lmm-MM-Decomp}]
	From taking divergence operator $\nabla_x \cdot $ on the balance law \eqref{Balance-b} for $b$ and replacing $\Delta_x \phi$ by $\gamma (a+3c)$ via using the Poisson equation \eqref{1.9}, we deduce that
    \begin{equation}\label{3.3.1}
       \begin{aligned}
         - \Delta_x ( a+3c ) + \gamma^2 (a+3c) = & \eps \partial_t \nabla_x \cdot b - \gamma \eps \nabla_x \cdot [ \nabla_x \phi (a+3c) ] \\
         + & 2 \Delta_x c + \nabla_x \cdot \l v \cdot \nabla_x (I-P) g , v \r_{L^2_v} \,.
       \end{aligned}
    \end{equation}
    For any multi-index $\alpha \in \mathbb{N}^3$ with $  |\alpha| \leq N-1$, from applying the derivative operator $\partial_x^\alpha$ to the equation \eqref{3.3.1}, multiplying by $\partial^\alpha_x (a+3c)$ and integrating by parts over $x \in \R^3$, one deduces  that
    \begin{equation}\label{R1+R2+R3+R4}
      \begin{aligned}
        & \| \nabla_x \partial_x^\alpha (a+3c) \|^2_{L_x^2} + \gamma^2 \| \partial_x^\alpha (a+3c) \|^2_{L^2_x} \\
        = & \underbrace{ \l \epsilon \partial_t \nabla_x \cdot \partial_x^\alpha b , \partial_x^\alpha (a+3c) \r_{L^2_x} }_{R_1} + \underbrace{ 2 \l \Delta_x \partial_x^\alpha c , \partial_x^\alpha (a+3c) \r_{L^2_x} }_{R_2} \\
        & + \underbrace{ \l \nabla_x \partial_x^\alpha \cdot \langle v \cdot \nabla_x (I-P) g , v \rangle_{L^2_v} , \partial_x^\alpha (a+3c) \r_{L^2_x} }_{R_3} \\
        & + \underbrace{ \gamma \epsilon \l \partial_x^\alpha \nabla_x \cdot [ \nabla_x \phi (a+3c) ] , \partial_x^\alpha (a+3c) \r_{L^2_x} }_{R_4} \,.
      \end{aligned}
    \end{equation}
    For $R_1$, it follows from the balance law \eqref{1.7} for $a+3c$ and integration by parts over $x \in \R^3$ that
    \begin{equation}\label{R1}
      \begin{split}
         R_1 & = \tfrac{\d}{\d t} \l \epsilon \nabla_x \cdot \partial_x^\alpha b , \partial_x^\alpha (a+3c) \r_{L^2_x} - \l \partial_x^\alpha \nabla_x \cdot b , \epsilon \partial_x^\alpha \partial_t (a+3c) \r_{L^2_x} \\
         & = \epsilon \tfrac{\d}{\d t} \l \nabla_x \cdot \partial_x^\alpha b , \partial_x^\alpha (a+3c) \r_{L^2_x} + \l \partial_x^\alpha \nabla_x \cdot b , \partial_x^\alpha \nabla_x \cdot b \r_{L^2_x} \\
         & \leq \epsilon \tfrac{\d}{\d t} \l \nabla_x \cdot \partial_x^\alpha b , \partial_x^\alpha (a+3c) \r_{L^2_x} + \| \partial_x^\alpha \nabla_x b \|^2_{L^2_x} \,.
      \end{split}
    \end{equation}
    For $R_2$, $R_3$, by applying integration by parts and the H\"older inequality, one has
    \begin{equation}\label{R2}
      \begin{split}
        R_2 \leq 2 \| \nabla_x \partial_x ^\alpha c \|^2_{L^2_x} \| \nabla_x \partial_x^\alpha (a+3c) \|^2_{L^2_x} \leq 4 \| \nabla_x \partial_x^\alpha c \|^2_{L^2_x} + \tfrac{1}{8} \| \nabla_x \partial^\alpha_x (a+3c)\|^2_{L^2_x} \,,
      \end{split}
    \end{equation}
    and
    \begin{equation}\label{R3}
      \begin{split}
        R_3 & = \l \partial_x^\alpha \langle v \cdot \nabla_x (I-P) g , v \rangle_{L^2_v} , \nabla_x \partial _x^\alpha (a+3c) \r_{L^2_x} \\
        & \leq \| \nabla_x \partial^\alpha_x (I-P) g \|_{L^2_{x,v}} \| \nabla_x \partial^\alpha_x (a + 3c) \|_{L^2_x} \l |v|^4 , 1 \r_{L^2_v}^\frac{1}{2} \\
        & = \sqrt{15} \| \nabla_x \partial_x^\alpha (a+3c) \|_{L^2_{x,v}} \| \nabla_x \partial_x^\alpha (I-P) g \|_{L^2_x} \\
        & \leq 30 \| \nabla_x \partial_x^\alpha (I-P) g \|^2_{L^2_{x,v}} + \tfrac{1}{8} \| \nabla_x \partial_x^\alpha (a+3c) \|^2_{L^2_x} \,.
      \end{split}
    \end{equation}
    Since $|\alpha| \leq N-1$, the term $R_4$ can be estimated by using the calculus inequality
    \begin{equation}\label{R4}
      \begin{split}
        R_4 = & - \gamma \epsilon \l \partial_x^\alpha [ \nabla_x \phi (a+3c) ] , \nabla_x \partial_x^\alpha (a+3c) \r_{L^2_x} \\
        \lesssim & \eps \| \partial_x^\alpha [ \nabla_x \phi (a+3c) ] \|_{L^2_x} \| \nabla_x \partial_x^\alpha (a+3c) \|_{L^2_x} \\
        \lesssim & \| \nabla_x \phi \|_{H^N_x} \| a+3c \|_{H^{N-1}_x} \| \nabla_x (a,c) \|_{H^{N-1}_x} \\
        \lesssim & \eps \| \nabla_x \phi \|_{H^N_x} \| \l g , 1 \r_{L^2_v} \|_{H^{N-1}_x} \| \nabla_x P g \|_{H^{N-1}_x L^2_v} \\
        \lesssim & \eps \mathcal{E}_N^\frac{1}{2} (g, \phi) \mathcal{D}_{N,\eps} (g) \,.
      \end{split}
    \end{equation}
    Plugging the bounds \eqref{R1}, \eqref{R2}, \eqref{R3} and \eqref{R4} into \eqref{R1+R2+R3+R4} tells us that
    \begin{equation}\label{3.3.16}
      \begin{aligned}
        & - \eps \tfrac{\d}{\d t} \l \nabla_x \cdot \partial^\alpha_x b , \partial^\alpha_x (a+3c) \r_{L^2_x} + \tfrac{3}{4} \| \nabla_x \partial_x^\alpha (a+3c) \|^2_{L^2_x} + \gamma^2 \| \partial_x^\alpha (a+3c) \|^2_{L^2_x} \\
        & - \| \nabla_x \partial^\alpha_x b \|^2_{L^2_x} - 4 \| \nabla_x \partial^\alpha_x c \|^2_{L^2_x} - 30 \| \nabla_x \partial^\alpha_x (I-P) g \|^2_{L^2_{x,v}} \\
        & \lesssim \eps \mathcal{E}_N^\frac{1}{2} (g, \phi) \mathcal{D}_{N,\eps} (g)
      \end{aligned}
    \end{equation}
    holds for all $|\alpha| \leq N-1$.

    For all multi-indexes $\alpha \in \mathbb{N}^3$ satisfying $|\alpha| \leq N-1$, we derive from the $b$-equation \eqref{3.3.7} and $c$-evolution \eqref{3.3.8} that for any fixed index $i \in \{ 1,2,3 \}$
    \begin{equation}
      \begin{split}
        - \Delta_x \partial_x^\alpha b_i = & - \sum_{j \neq i} \partial_j \partial_j \partial_x^\alpha b_i - \partial_i \partial_i \partial_x^\alpha b_i \\
        = & \sum_{j \neq i} \partial_j \partial_x^\alpha ( \partial_i b_j - l_{ij} - h_{ij} - m_{ij} ) + \partial_i \partial_x^\alpha b_i ( \eps \partial_t c - l_i - h_i - m_i ) \\
        = & \sum_{j \neq i} \partial_i \partial_x^\alpha ( - \eps \partial_t c + l_j + h_j + m_j ) - \sum_{j \neq i} \partial_j \partial_x^\alpha ( l_{ij} + h_{ij} + m_{ij} ) \\
        & + \partial_i \partial^\alpha_x ( \eps \partial_t c - l_i - h_i - m_i ) \\
        = & - \eps \partial_t \partial_i \partial^\alpha_x c - \sum_{j \neq i} \partial_j \partial_x^\alpha ( l_{ij} + h_{ij} + m_{ij} ) \\
        & + \partial_i \partial^\alpha_x \Big[ \sum_{j \neq i} ( l_j + h_j + m_j ) - (l_i + h_i + m_i) \Big] \,.
      \end{split}
    \end{equation}
    Since $l_i$, $l_{ij}$, $h_i$, $h_{ij}$ and $m_i$, $m_{ij}$ are the coefficients of $l$, $h$ and $m$, respectively, there is a certain linear combination $\zeta_{ij} (v)$ of the basis \eqref{e 13} such that
    \begin{equation}
      \begin{aligned}
        \partial_i \partial^\alpha_x \Big[ \sum_{j \neq i} ( l_j + h_j + m_j ) - (l_i + h_i + m_i) \Big] - \sum_{j \neq i} \partial_j \partial_x^\alpha ( l_{ij} + h_{ij} + m_{ij} ) \\
        = \sum_{j=1}^3 \partial_j \partial^\alpha_x \l l + h + m , \zeta_{ij} (v) \r_{L^2_v} \,.
      \end{aligned}
    \end{equation}
    Combining the balance law \eqref{Balance-c} for $c$, we know that
    \begin{equation}\label{3.3.5}
      \begin{aligned}
        - \Delta_x \partial^\alpha_x b_i - \tfrac{1}{3} \partial_i \partial^\alpha_x \nabla_x \cdot b = &  \partial_i \partial^\alpha_x \Big( - \tfrac{\gamma}{3} \eps b \cdot \nabla_x \phi + \tfrac{1}{6} \l v \cdot \nabla_x (I-P) g , |v|^2 \r_{L^2_v} \Big) \\
        & + \sum_{j=1}^3 \partial_j \partial^\alpha_x \l l + h + m , \zeta_{ij} (v) \r_{L^2_v}
      \end{aligned}
    \end{equation}
    for any fixed index $i \in \{ 1,2,3 \}$ and all $\alpha \in \mathbb{N}^3$ with $|\alpha| \leq N-1$. We multiply \eqref{3.3.5} by $\partial_x^\alpha b_i$, sum up for the index $i$ and integrate over $x \in \mathbb{R}^3$. We thereby have
    \begin{equation}\label{T1+T2+T3+T4+T5}
      \begin{aligned}
        & \| \nabla_x \partial^\alpha_x b \|^2_{L^2_x} + \tfrac{1}{3} \| \partial^\alpha_x \nabla_x \cdot b \|^2_{L^2_x} \\
        = & \underbrace{ - \sum_{i,j=1}^3 \l \partial^\alpha_x l , \partial_j \partial^\alpha_x b_i \zeta_{ij} (v) \r_{L^2_{x,v}} }_{T_1} \ \underbrace{ - \sum_{i,j=1}^3 \l \partial^\alpha_x h , \partial_j \partial^\alpha_x b_i \zeta_{ij} (v) \r_{L^2_{x,v}} }_{T_2} \\
        & \underbrace{ - \sum_{i,j=1}^3 \l \partial^\alpha_x m , \partial_j \partial^\alpha_x b_i \zeta_{ij} (v) \r_{L^2_{x,v}} }_{T_3} + \underbrace{ \tfrac{\gamma}{3} \eps \l \partial^\alpha_x ( b \cdot \nabla_x \phi ) , \partial^\alpha_x \nabla_x \cdot b \r_{L^2_x} }_{T_4} \\
        & \underbrace{ - \tfrac{1}{6} \l v \cdot \nabla_x \partial^\alpha_x (I-P) g , |v|^2 \partial^\alpha_x \nabla_x \cdot b \r_{L^2_{x,v}} }_{T_5}
      \end{aligned}
    \end{equation}
    for $|\alpha| \leq N-1$.

    We next estimate each term $T_i\ (1\leq i\leq 5)$ in \eqref{T1+T2+T3+T4+T5}. Recalling the definition of $l$ in \eqref{l+h+m}, we have
    \begin{equation}\label{T1=T11+T12+T13}
      \begin{aligned}
        T_1 = & - \eps \tfrac{\d}{\d t} \sum_{i,j=1}^3 \l \partial_j \partial^\alpha_x (I-P) g , \partial^\alpha_x b_i \zeta_{ij} (v) \r_{L^2_{x,v}} \\
        & + \underbrace{ \sum_{i,j=1}^3 \l \partial_j \partial^\alpha_x (I-P) g , \eps \partial_t \partial^\alpha_x b_i \zeta_{ij} (v) \r_{L^2_{x,v}} }_{T_{11}} \\
        & + \underbrace{ \sum_{i,j=1}^3 \l v \cdot \nabla_x \partial^\alpha_x (I-P) g , \partial_j \partial^\alpha_x b_i \zeta_{ij} (v) \r_{L^2_{x,v}} }_{T_{12}} \\
        & + \underbrace{ \sum_{i,j=1}^3 \l \tfrac{1}{\eps} L \partial^\alpha_x (I-P) g , \partial_j \partial^\alpha_x b_i \zeta_{ij} (v) \r_{L^2_{x,v}} }_{T_{13}} \,.
      \end{aligned}
    \end{equation}
    It follows from plugging the balance law \eqref{Balance-b} into the term $T_{11}$ that
    \begin{equation}\no
      \begin{aligned}
        T_{11} = & \underbrace{ \sum_{i,j=1}^3 \l \partial_j \partial^\alpha_x (I-P) g , \gamma \partial^\alpha_x \partial_i \phi \zeta_{ij} (v) \r_{L^2_{x,v}} }_{T_{111}} \\
        & \underbrace{ - \sum_{i,j=1}^3 \l \partial_j \partial^\alpha_x (I-P) g , \partial_i \partial^\alpha_x (a+5c) \zeta_{ij} (v) \r_{L^2_{x,v}} }_{T_{112}} \\
        & \underbrace{ - \sum_{i,j=1}^3 \l \partial_j \partial^\alpha_x (I-P) g , \partial^\alpha_x \l v \cdot \nabla_x (I-P) g , v_i \r_{L^2_v} \zeta_{ij} (v) \r_{L^2_{x,v}} }_{T_{113}} \\
        & + \underbrace{ \sum_{i,j=1}^3 \l \partial_j \partial^\alpha_x (I-P) g , \partial^\alpha_x [ \gamma (a+3c) \partial_i \phi ] \zeta_{ij} (v) \r_{L^2_{x,v}} }_{T_{114}} \,.
      \end{aligned}
    \end{equation}
    Then the term $I_{111}$ is bounded by
    \begin{equation}\no
      \begin{aligned}
        T_{111} \lesssim \| \nabla_x \partial^\alpha_x (I-P) g \|_{L^2_{x,v}} \| \nabla_x \partial^\alpha_x \phi \|_{L^2_x} \lesssim \| (I-P) g \|_{H^N_x L^2_v (\nu)} \| \partial^\alpha_x (a+3c) \|_{L^2_x} \,,
      \end{aligned}
    \end{equation}
    where we make use of Lemma \ref{Lmm-CF-nu} and the estimate $\| \nabla_x \partial^\alpha_x \phi \|_{L^2_x} \lesssim \| \partial^\alpha_x \phi \|_{H^2_x} \lesssim \| \partial^\alpha_x (a+3c) \|_{L^2_x}$ derived from the Poisson equation $\Delta_x \phi = \gamma ( a + 3 c )$. For the term $I_{112}$, we estimate that
    \begin{equation}\no
      \begin{aligned}
        T_{112} \lesssim & \| \nabla_x \partial^\alpha_x  (I-P) g \|_{L^2_{x,v}} \big( \| \nabla_x \partial^\alpha_x (a+3c) \|_{L^2_x} + \| \nabla_x \partial^\alpha_x c \|_{L^2_x} \big) \\
        \lesssim & \| (I-P) g \|_{H^N_x L^2_v (\nu)} \big( \| \nabla_x \partial^\alpha_x (a+3c) \|_{L^2_x} + \| \nabla_x \partial^\alpha_x c \|_{L^2_x} \big) \,.
      \end{aligned}
    \end{equation}
    The terms $I_{113}$ and $I_{114}$ are bounded by
    \begin{equation}\no
      \begin{aligned}
        T_{113} \lesssim \| \nabla_x \partial^\alpha_x (I-P) g \|^2_{L^2_{x,v} } \lesssim \| (I-P) g \|^2_{H^N_x L^2_v (\nu)} \,,
      \end{aligned}
    \end{equation}
    and
    \begin{equation}\no
      \begin{aligned}
        T_{114} \lesssim & \| \nabla_x \partial^\alpha_x (I-P) g \|_{L^2_{x,v}} \| \partial^\alpha_x [ (a+3c) \nabla_x \phi ] \|_{L^2_x} \\
        \lesssim & \| (I-P) g \|_{H^N_x L^2_v (\nu)} \| (a+3c) \|_{H^{N-1}_x} \| \nabla_x \phi \|_{H^N_x}  \lesssim \eps \mathcal{E}_N^\frac{1}{2} (g, \phi) \mathcal{D}_{N,\eps} (g) \,.
      \end{aligned}
    \end{equation}
    In summary, we have
    \begin{equation}\label{T11}
      \begin{aligned}
        T_{11} \lesssim & \| (I-P) g \|_{H^N_x L^2_v (\nu)} \Big( \| \nabla_x \partial^\alpha_x (a+3c) \|_{L^2_x} + \| \partial^\alpha_x (a+3c) \|_{L^2_x} + \| \nabla_x \partial^\alpha_x c \|_{L^2_x} \Big) \\
        \lesssim & \| (I-P) g \|_{H^N_x L^2_v (\nu)}^2 + \eps \mathcal{E}_N^\frac{1}{2} (g, \phi) \mathcal{D}_{N,\eps} (g) \,.
      \end{aligned}
    \end{equation}
    Furthermore, one can easily estimate that
    \begin{equation}\label{T12}
      \begin{aligned}
        T_{12} \lesssim \| \nabla_x \partial^\alpha_x (I-P) g \|_{L^2_{x,v}} \| \nabla_x \partial^\alpha_x b \|_{L^2_x} \lesssim \| (I-P) g \|_{H^N_x L^2_v (\nu)} \| \nabla_x \partial^\alpha_x b \|_{L^2_x} \,,
      \end{aligned}
    \end{equation}
    and
    \begin{equation}\label{T13}
      \begin{aligned}
        T_{13} = & \tfrac{1}{\eps} \sum_{i,j=1}^3 \l \partial^\alpha_x (I-P) g , \partial_j \partial^\alpha_x b_i L \zeta_{ij} (v) \r_{L^2_{x,v}} \\
        \lesssim & \tfrac{1}{\eps} \| \partial^\alpha_x (I-P) g \|_{L^2_{x,v}} \| \nabla_x \partial^\alpha_x b \|_{L^2_x} \lesssim \tfrac{1}{\eps} \| (I-P) g \|_{H^N_x L^2_v (\nu)} \| \nabla_x \partial^\alpha_x b \|_{L^2_x} \,,
      \end{aligned}
    \end{equation}
    where we use the self-adjoint property of the linearized Boltzmann collision operator $L$. From plugging the bounds \eqref{T11}, \eqref{T12} and \eqref{T13} into the equality \eqref{T1=T11+T12+T13}, one deduces that
    \begin{equation}\label{T1}
      \begin{aligned}
        T_1 & + \eps \tfrac{\d}{\d t} \sum_{i,j=1}^3 \l \partial_j \partial^\alpha_x (I-P) g , \partial^\alpha_x b_i \zeta_{ij} (v) \r_{L^2_{x,v}} \\
        \lesssim & \| (I-P) g \|_{H^N_x L^2_v (\nu)} \Big( \| \nabla_x \partial^\alpha_x (a+3c) \|_{L^2_x} + \| \partial^\alpha_x (a+3c) \|_{L^2_x} + \| \nabla_x \partial^\alpha_x c \|_{L^2_x} \Big) \\
        + & \| (I-P) g \|_{H^N_x L^2_v (\nu)}^2 + \tfrac{1}{\eps} \| (I-P) g \|_{H^N_x L^2_v (\nu)} \| \nabla_x \partial^\alpha_x b \|_{L^2_x} + \eps \mathcal{E}_N^\frac{1}{2} (g, \phi) \mathcal{D}_{N,\eps} (g)
      \end{aligned}
    \end{equation}
    for all $0 < \eps \leq 1$ and $|\alpha| \leq N-1$.

    Recalling the definition of $h$ in \eqref{l+h+m}, we derive from the similar estimates in \eqref{I-1} that
    \begin{equation}\label{T2}
      \begin{aligned}
        T_2 \lesssim & \| g \|_{H^N_x L^2_v} \| \nabla_x \partial^\alpha_x b \|_{L^2_x} \Big( \| (I-P) g \|_{H^N_x L^2_v (\nu)} + \| \nabla_x P g \|_{H^{N-1}_x L^2_v} \Big) \\
        \lesssim & \| g \|_{H^N_x L^2_v} \Big( \| (I-P) g \|_{H^N_x L^2_v (\nu)}^2 + \| \nabla_x P g \|_{H^{N-1}_x L^2_v}^2 \Big) \lesssim \mathcal{E}_N^\frac{1}{2} (g, \phi) \mathcal{D}_{N,\eps} (g) \,.
      \end{aligned}
    \end{equation}

    We next estimate the term $T_3$. By definition of $m$ in \eqref{l+h+m}, we have
    \begin{equation}\label{T3}
      \begin{aligned}
        T_3 = & - \gamma \eps \sum_{i,j=1}^3 \l \partial^\alpha_x ( v \cdot \nabla_x \phi - \nabla_x \phi \cdot \nabla_v g ) , \partial_j \partial^\alpha_x b_i \zeta_{ij} (v) \r_{L^2_{x,v}} \\
        \lesssim & \eps \| \nabla_x \partial^\alpha_x  b \|_{L^2_x} \| \partial^\alpha_x ( v \cdot \nabla_x \phi g - \nabla_x \phi \cdot \nabla_v g ) \|_{L^2_{x,v}} \\
        \lesssim & \eps \| \nabla_x \partial^\alpha_x  b \|_{L^2_x} \| \nabla_x \phi \|_{H^{N-1}_x} \| g \|_{H^N_{x,v} (\nu)} \\
        \lesssim & \eps \| \nabla_x \phi \|_{H^{N-1}_x} \| a + 3 c \|_{H^{N-1}_x} \Big( \| P g \|_{H^N_x L^2_v} + \| (I-P) g \|_{H^N_{x,v} (\nu)} \Big) \\
        \lesssim & \eps \mathcal{E}_N^\frac{1}{2} (g, \phi) \mathcal{D}_{N,\eps} (g) \,,
      \end{aligned}
    \end{equation}
    where we use the bound $\| \nabla_x \phi \|_{H^{N-1}_x} \lesssim \|a+3c\|_{H^{N-1}_x}$ derived from the Poisson equation $\Delta_x \phi = \gamma (a+3c)$. Similarly in \eqref{T3}, the term $T_4$ can be bounded by
    \begin{equation}\label{T4}
      \begin{aligned}
        T_4 \lesssim & \eps \| b \|_{H^{N-1}_x} \| \nabla_x \phi \|_{H^{N-1}_x} \| \nabla_x b \|_{H^{N-1}_x} \lesssim \eps \| b \|_{H^N_x} \| a+3c \|_{H^{N-1}_x} \| \nabla_x b \|_{H^{N-1}_x} \\
        \lesssim & \eps \| g \|_{H^N_x L^2_v} \| \l g,1 \r_{L^2_v} \|_{H^{N-1}_x} \| \nabla_x P g \|_{H^{N-1}_x L^2_v} \lesssim \eps \mathcal{E}_N^\frac{1}{2} (g, \phi) \mathcal{D}_{N, \eps} (g) \,.
      \end{aligned}
    \end{equation}
    Furthermore, the term $T_5$ is easily to be estimated that
    \begin{equation}\label{T5}
      \begin{aligned}
        T_5 \lesssim & \| \partial^\alpha_x \nabla_x \cdot b \|_{L^2_x} \| \nabla_x \partial^\alpha_x (I-P) g \|_{L^2_{x,v}} \lesssim \| (I-P) g \|_{H^N_x L^2_v (\nu)} \| \nabla_x \partial^\alpha_x b \|_{L^2_x} \,.
      \end{aligned}
    \end{equation}
    Finally, we substitute the inequalities \eqref{T1}, \eqref{T2}, \eqref{T3}, \eqref{T4} and \eqref{T5} into \eqref{T1+T2+T3+T4+T5}. The Young's inequality tells us that
    \begin{equation}\label{3.3.15}
      \begin{aligned}
        & \tfrac{1}{2} \| \nabla_x \partial^\alpha_x b \|_{L^2_x}^2 + \tfrac{1}{3} \| \partial^\alpha_x \nabla_x \cdot b \|^2_{L^2_x} + \eps \tfrac{\d}{\d t} \sum_{i,j=1}^3 \l \partial_j \partial^\alpha_x (I-P) g , \partial^\alpha_x b_i \zeta_{ij} (v) \r_{L^2_{x,v}} \\
        \lesssim & \| (I-P) g \|_{H^N_x L^2_v (\nu)} \Big( \| \nabla_x \partial^\alpha_x (a+3c) \|_{L^2_x} + \| \partial^\alpha_x (a+3c) \|_{L^2_x} + \| \nabla_x \partial^\alpha_x c \|_{L^2_x} \Big) \\
        + & \tfrac{1}{\eps^2} \| (I-P) g \|^2_{H^N_x L^2_v (\nu)} + \mathcal{E}_N^\frac{1}{2} (g, \phi) \mathcal{D}_{N, \eps} (g)
      \end{aligned}
    \end{equation}
    for all $0 < \eps \leq 1$ and $|\alpha| \leq N-1$.

    It remains to estimate the norm $\| \nabla_x \partial^\alpha_x c \|_{L^2_x}$ for all $|\alpha| \leq N-1$. From \eqref{3.3.9}, we have
    \begin{equation}\label{3.3.12}
      \begin{aligned}
        - \Delta_x \partial^\alpha_x c = - \sum_{i=1}^3 \partial_i \partial^\alpha_x \big(  l_{ci} + h_{ci} + m_{ci} \big) = \sum_{i=1}^3 \partial_i \partial^\alpha_x \l l + h + m , \zeta_i (v) \r_{L^2_v}
      \end{aligned}
    \end{equation}
    for some certain linear combinations $\zeta_i (v)$ of the basis in \eqref{e 13}. Multiplying \eqref{3.3.12} by $\nabla_x\partial_x^\alpha c$ yields
    \begin{equation}\label{M1+M2+M3}
      \begin{split}
        \| \nabla_x \partial_x^\alpha c \|^2_{L^2_x} = & \underbrace{ - \sum_{i=1}^3 \l \partial^\alpha_x l , \partial_i \partial^\alpha_x c \zeta_i (v) \r_{L^2_{x,v}} }_{M_1} \ \underbrace{ - \sum_{i=1}^3 \l \partial^\alpha_x h , \partial_i \partial^\alpha_x c \zeta_i (v) \r_{L^2_{x,v}} }_{M_2} \\
        & \underbrace{ - \sum_{i=1}^3 \l \partial^\alpha_x m , \partial_i \partial^\alpha_x c \zeta_i (v) \r_{L^2_{x,v}} }_{M_3}
      \end{split}
    \end{equation}
    for all $|\alpha| \leq N-1$.

    We next estimate each term in \eqref{M1+M2+M3} as follows. For $M_1$, by plugging $l$ defined in \eqref{l+h+m}, we have
    \begin{equation}
      \begin{aligned}
        M_1 = & - \eps \tfrac{\d}{\d t} \sum_{i=1}^3  \l \partial_i \partial^\alpha_x (I-P) g , \ \partial^\alpha_x c \zeta_i (v) \r_{L^2_{x,v}} \\
        + & \sum_{i=1}^3 \l \partial_i \partial^\alpha_x (I-P) g , \tfrac{1}{3} \partial^\alpha_x \nabla_x \cdot b \zeta_i (v) - \tfrac{\gamma}{3} \eps \partial^\alpha_x ( b \cdot \nabla_x \phi ) \zeta_i (v) \r_{L^2_{x,v}} \\
        + & \sum_{i=1}^3 \l \partial_i \partial^\alpha_x (I-P) g , \tfrac{1}{6} \partial^\alpha_x \l v \cdot \nabla_x (I-P) g , |v|^2 \r_{L^2_v} \zeta_i (v) \r_{L^2_{x,v}} \\
        + & \sum_{i=1}^3 \l v \cdot \nabla_x \partial^\alpha_x (I-P) g + \tfrac{1}{\eps} L \partial^\alpha_x (I-P) g , \partial_i \partial^\alpha_x c \zeta_i (v) \r_{L^2_{x,v}} \,.
      \end{aligned}
    \end{equation}
    Then, by employing the analogous arguments in \eqref{T1}, \eqref{T2} and \eqref{T3}, one can estimate that
    \begin{equation}
      \begin{aligned}
        & M_1 + \eps \tfrac{\d}{\d t} \sum_{i=1}^3  \l \partial_i \partial^\alpha_x (I-P) g , \ \partial^\alpha_x c \zeta_i (v) \r_{L^2_{x,v}} \\
        \lesssim & \tfrac{1}{\eps} \| (I-P) g \|_{H^N_x L^2_v (\nu)} \| \nabla_x \partial^\alpha_x c \|_{L^2_x} + \| (I-P) g \|_{H^N_x L^2_v (\nu)}^2 \\
        + & \| (I-P) g \|_{H^N_x L^2_v (\nu)} \| \nabla_x \partial^\alpha_x b \|_{L^2_x} + \mathcal{E}^\frac{1}{2} (g, \phi) \mathcal{D}_{N,\eps} (g) \,,
      \end{aligned}
    \end{equation}
    and
    \begin{equation}
      \begin{aligned}
        M_2 + M_3 \lesssim \mathcal{E}^\frac{1}{2} (g, \phi) \mathcal{D}_{N,\eps} (g) \,.
      \end{aligned}
    \end{equation}
    Consequently, we have
    \begin{equation}\label{3.3.14}
      \begin{split}
        & \tfrac{1}{2} \| \nabla_x \partial_x^\alpha c \|^2_{L^2_x} + \eps \tfrac{\d}{\d t} \sum_{i=1}^3  \l \partial_i \partial^\alpha_x (I-P) g , \ \partial^\alpha_x c \zeta_i (v) \r_{L^2_{x,v}} \\
        & \lesssim \| (I-P) g \|_{H^N_x L^2_v (\nu)} \| \nabla_x \partial^\alpha_x b \|_{L^2_x} + \tfrac{1}{\eps^2} \| (I-P) g \|_{H^N_x L^2_v (\nu)}^2 + \mathcal{E}^\frac{1}{2} (g, \phi) \mathcal{D}_{N,\eps} (g)
      \end{split}
    \end{equation}
    for all $0 < \eps \leq 1$ and $|\alpha| \leq N-1$.

    From adding the bounds \eqref{3.3.15}, \eqref{3.3.14} to $\tfrac{1}{16}$ times of \eqref{3.3.16}, summing up for $|\alpha| \leq N-1$ and the Young's inequality, we deduce that
    \begin{equation}
      \begin{aligned}
        \tfrac{\eps}{32} \tfrac{\d}{\d t} E^{int}_N (g) + \tfrac{1}{32} \Big( \| \nabla_x b \|^2_{H^{N-1}_x} + \| \nabla_x c \|^2_{H^{N-1}_x} + \| \nabla_x (a+3c) \|^2_{H^{N-1}_x} + \| a+3c \|^2_{H^{N-1}_x} \Big) \\
        \lesssim \tfrac{1}{\eps^2} \| (I-P) g \|^2_{H^N_x L^2_v (\nu)} + \mathcal{E}_N^\frac{1}{2} (g, \phi) \mathcal{D}_{N, \eps} (g)
      \end{aligned}
    \end{equation}
    for all $0 < \eps \leq 1$ and $N \geq 4$, where the so-called interactive energy functional $E^{int}_N (g)$ is given in \eqref{Interactive-Energy}. Thus the proof of Lemma \ref{Lmm-MM-Decomp} is finished.
\end{proof}

\subsection{Energy estimates for $(x,v)$-mixed derivatives}

In this subsection, our goal is to derive a closed energy estimate. Recalling the estimates obtained in the previous two subsections, we only require to estimate the energy of $(x,v)$-mixed derivatives of the kinetic part $(I-P) g$. We so first give the following lemma.

\begin{lemma}\label{Lmm-Mixed-Est}
	Assume that $g (t,x,v)$ is the solution to the perturbed VPB system \eqref{VPB-g}-\eqref{IC-g} constructed in Lemma \ref{Lmm-Local}. Let $N \geq 3$ be any fixed integer. For any given $0 \leq k \leq N$ and $|\beta| \leq k$, there are positive constants $\varrho_k$, $\varrho_k^*$, $C_{|\beta|}$, $r_k$, $r_k^*$, $ \xi_k $ and $\eta_0$, independent of $\eps$, such that
	\begin{equation}\label{Mixed-Inq-Closed}
	  \begin{aligned}
	    \tfrac{1}{2} \tfrac{\d}{\d t} & \Bigg( \varrho_k \| g \|^2_{H^N_x L^2_v} + \varrho_k \| \nabla_x \phi \|^2_{H^N_x} + \varrho_k^* \eta E_N^{int} (g) + \sum_{\substack{|\alpha| + |\beta| \leq N \\ |\beta| \leq k}} C_{|\beta|} \eta^2 \| \partial^\alpha_x \partial^\beta_v (I-P) g \|^2_{L^2_{x,v}} \Bigg) \\
	    & + \tfrac{r_k}{\eps^2} \| (I-P) g \|^2_{H^N_x L^2_v (\nu)} + r_k^* \eta \| \nabla_x P g \|^2_{H^{N-1}_x L^2_v} + r_k^* \eta \| \l g ,1 \r_{L^2_v} \|^2_{H^{N-1}_x} \\
	    & + \tfrac{\xi_k}{\eps^2} \eta^2 \sum_{\substack{|\alpha| + |\beta| \leq N \\ |\beta| \leq k}} \| \partial^\alpha_x \partial^\beta_v (I-P) g \|^2_{L^2_{x,v} (\nu)} \lesssim \mathcal{E}_N^\frac{1}{2} (g, \phi) \mathcal{D}_{N, \eps} (g)
	  \end{aligned}
	\end{equation}
\end{lemma}
for all $0 < \eps \leq 1$ and $0 < \eta \leq \eta_0$.

\begin{remark}\label{Rmk-Mixed}
	The small $\eta \in (0, \eta_0 ]$ mentioned in Lemma \ref{Lmm-Mixed-Est} will be chosen smaller later, such that the unsigned interactive energy functional $\varrho_k^* \eta E_N^{int} (g)$ can be dominated by $\varrho_k \| g \|^2_{H^N_x L^2_v}$ due to the bound \eqref{Interactive-Energy-Bnd} in Remark \ref{Rmk-MM}. More precisely, we introduce the following instant energy functional
	\begin{equation}
	  \begin{aligned}
	    \mathbb{E}_{N, \eta} (g, \phi) = & \varrho_N \| g \|^2_{H^N_x L^2_v} + \varrho_N \| \nabla_x \phi \|^2_{H^N_x} + \varrho_N^* \eta E_N^{int} (g) \\
	    & + \sum_{\substack{|\alpha| + |\beta| \leq N \\ |\beta| \leq N}} C_{|\beta|} \eta^2 \| \partial^\alpha_x \partial^\beta_v (I-P) g \|^2_{L^2_{x,v}}
	  \end{aligned}
	\end{equation}
	and the instant energy dissipative rate functional
	\begin{equation}
	  \begin{aligned}
	    \mathbb{D}_{N, \eps, \eta} (g) = & \tfrac{r_N}{\eps^2} \| (I-P) g \|^2_{H^N_x L^2_v (\nu)} + r_N^* \eta \| \nabla_x P g \|^2_{H^{N-1}_x L^2_v} + r_N^* \eta \| \l g ,1 \r_{L^2_v} \|^2_{H^{N-1}_x} \\
	    & + \tfrac{\xi_k}{\eps^2} \eta^2 \sum_{\substack{|\alpha| + |\beta| \leq N \\ |\beta| \leq N}} \| \partial^\alpha_x \partial^\beta_v (I-P) g \|^2_{L^2_{x,v} (\nu)} \,.
	  \end{aligned}
	\end{equation}
	It is easy to verify that there is a small $\eta_1 \in (0, \eta_0 ]$, independent of $\eps$, such that
	\begin{equation}\label{Equivalences-Energies}
	  \begin{aligned}
	    \mathbb{E}_{N, \eta_1} (g, \phi) \thicksim \mathcal{E}_N (g, \phi) \,, \quad \mathbb{D}_{N, \eps, \eta_1} (g) \thicksim \mathcal{D}_{N, \eps} (g) \,.
	  \end{aligned}
	\end{equation}
	Consequently, by letting $k = N$ in the inequality \eqref{Mixed-Inq-Closed}, we have
	\begin{equation}\label{Closed-Energy-Inq}
	  \begin{aligned}
	    \tfrac{1}{2} \tfrac{\d}{\d t} \mathbb{E}_{N, \eta_1} (g, \phi) + \mathbb{D}_{N, \eps, \eta_1} (g) \leq C \mathbb{E}^\frac{1}{2}_{N, \eta_1} (g, \phi) \mathbb{D}_{N, \eps, \eta_1} (g)
	  \end{aligned}
	\end{equation}
	for some positive constant $C > 0$, independent of $\eps$, and for all $0 < \eps \leq 1$.
\end{remark}

\begin{proof}[Proof of Lemma \ref{Lmm-Mixed-Est}]

Employing the microscopic projection $\{I-P\}$ to \eqref{VPB-g-drop-eps}, one obtains
\begin{equation}\label{3.4.2}
  \begin{split}
    & \partial_t (I-P) g + \tfrac{1}{\eps} (I-P) (v \cdot \nabla_x (I-P) g ) + \tfrac{1}{\eps^2} L (I-P) g \\
    & + \gamma (I-P) \big[ \nabla_x \phi \cdot \nabla_v (I-P) g - v \cdot \nabla_x \phi (I-P) g \big] \\
    & = \tfrac{1}{\eps} Q(g, g) + \gamma (I-P) ( v \cdot \nabla_x \phi P g ) - \tfrac{1}{\eps} (I-P) ( v \cdot \nabla_x P g ) \,,
  \end{split}
\end{equation}
where we make use of the decomposition $g = P g + (I-P) g$ and the relations $ (I-P) (v \cdot \nabla_x \phi) = (I-P) (\nabla_x \cdot \nabla_v P g) = 0 $. For all multi-indexes $\alpha$, $\beta \in \mathbb{N}^3$ with $|\alpha| + |\beta| \leq N$ and $\beta \neq 0$, we apply the mixed derivatives operator $\partial_x^\alpha\partial_v^\beta$ to \eqref{3.4.2}, take $L^2_{x,v}$-inner product via multiplying by $\partial^\alpha_x \partial^\beta_v (I-P) g$ and integrate by parts over $(x,v) \in \R^3 \times \R^3$. Then, from Lemma \ref{Lmm-Coercivity-L} and the fact $\partial^{\beta_1}_v v = 0$ for $|\beta_1| \geq 2$, we deduce that there are two positive constants $\delta_1$ and $\delta_2$, independent of $\eps$, such that
  \begin{align}\label{W1+W2+W3+W4+W5+W6+W7}
    \no & \tfrac{1}{2} \tfrac{\d}{\d t} \| \partial_x^\alpha \partial_v^\beta (I-P) g \|^2_{L^2_{x,v}} + \tfrac{\delta_1}{\epsilon^2} \| \partial_x^\alpha \partial_v^\beta (I-P) g \|_{L^2_{x,v}(\nu)}^2 - \tfrac{\delta_2}{\eps^2} \sum_{\beta' < \beta} \| \partial_x^\alpha \partial_v^{\beta'} (I-P) g  \|_{L^2_{x,v}}^2 \\
    \no & \leq \underbrace{ \tfrac{1}{\epsilon} \sum_{ | \beta_1 | = 1} C_\beta^{\beta_1} \langle \partial_v^{\beta_1} v \cdot \nabla_x \partial_x^\alpha \partial_v^{\beta-\beta_1} (I-P) g , \partial_x^\alpha \partial_v^\beta (I-P) g \rangle_{L^2_{x,v}} }_{W_1} \\
    \no & \underbrace{ - \gamma \langle \partial_x^\alpha \partial_v^\beta (I-P) ( \nabla_x \phi^n \cdot \nabla_v (I-P) g ) , \partial_x^\alpha \partial_v^\beta (I-P) g \rangle_{L^2_{x,v}} }_{W_2} \\
    \no & + \underbrace{ \gamma \langle \partial_x^\alpha \partial_v^\beta (I-P) ( v \cdot \nabla_x \phi^n (I-P) g ) , \partial_x^\alpha \partial_v^\beta (I-P) g  \rangle_{L^2_{x,v}} }_{W_3} \\
    \no & + \underbrace{  \tfrac{1}{\epsilon} \langle \partial_x^\alpha \partial_v^\beta Q ( g , g ) , \partial_x^\alpha \partial_v^\beta (I-P) g  \rangle_{L^2_{x,v}} }_{W_4} \\
    \no & + \underbrace{ \gamma \langle \partial_x^\alpha \partial_v^\beta (I-P) ( v \cdot \nabla_x \phi P g ) , \partial_x^\alpha \partial_v^\beta (I-P) g \rangle_{L^2_{x,v}} }_{W_5} \\
    \no & \underbrace{ - \tfrac{1}{\eps} \langle \partial_x^\alpha \partial_v^\beta (I-P) ( v \cdot \nabla_x P g ), \partial_x^\alpha \partial_v^\beta (I-P) g \rangle_{L^2_{x,v}} }_{W_6} \\
    & + \underbrace{ \tfrac{1}{\eps} \langle \partial_x^\alpha \partial_v^\beta P ( v \cdot \nabla_x (I-P) g ), \partial_x^\alpha \partial_v^\beta (I-P) g \rangle_{L^2_{x,v}} }_{W_7} \,,
  \end{align}

By the similar arguments in \eqref{S1}, we have
\begin{equation}\label{W1}
  \begin{aligned}
    W_1 \lesssim \tfrac{1}{\eps} \sum_{\beta' < \beta} \| \nabla_x \partial^\alpha_x \partial^{\beta'}_v (I-P) g \|_{L^2_{x,v} (\nu)} \| \partial^\alpha_x \partial^\beta_v (I-P) g \|_{L^2_{x,v} (\nu) } \,.
  \end{aligned}
\end{equation}
The derivations of the bounds \eqref{S2}, \eqref{S3} and Lemma \ref{Lmm-CF-nu} tell us that
\begin{equation}\label{W2+W3}
  \begin{aligned}
    W_2 + W_3 \lesssim \| \nabla_x \phi \|_{H^N_x} \big( \| (I-P) g \|^2_{\widetilde{H}^N_{x,v} (\nu)} + \| (I-P) g \|^2_{H^N_x L^2_v (\nu)} \big) \lesssim \eps^2 \mathcal{E}_N^\frac{1}{2} (g, \phi) \mathcal{D}_{N,\eps} (g) \,.
  \end{aligned}
\end{equation}
We plug the splitting $g = P g + (I-P) g$ into the term $W_4$, namely,
\begin{equation}\no
  \begin{aligned}
    W_4 = & \tfrac{1}{\eps} \l \partial^\alpha_x \partial^\beta_v Q (P g , P g) , \partial^\alpha_x \partial^\beta_v (I-P) g \r_{L^2_{x,v}} + \tfrac{1}{\eps} \l \partial^\alpha_x \partial^\beta_v Q (P g , (I-P) g)  , \partial^\alpha_x \partial^\beta_v (I-P) g \r_{L^2_{x,v}} \\
    + & \tfrac{1}{\eps} \l \partial^\alpha_x \partial^\beta_v Q ((I-P) g , P g) , \partial^\alpha_x \partial^\beta_v (I-P) g \r_{L^2_{x,v}} \\
    + & \tfrac{1}{\eps} \l \partial^\alpha_x \partial^\beta_v Q ((I-P) g , (I-P) g) , \partial^\alpha_x \partial^\beta_v (I-P) g \r_{L^2_{x,v}}  \,,
  \end{aligned}
\end{equation}
where the last three terms can be bounded by Lemma \ref{Lmm-CF-nu} and \ref{Lmm-Q}
\begin{equation}\no
  \begin{aligned}
    \tfrac{1}{\eps} \| g \|_{H^N_{x,v}} \| (I-P) g \|^2_{H^N_{x,v}(\nu)} \lesssim \eps \mathcal{E}_N^\frac{1}{2} (g, \phi) \mathcal{D}_{N,\eps} (g) \,,
  \end{aligned}
\end{equation}
and the first term can be bounded by using the similar estimates of \eqref{I11}
\begin{equation}\no
  \begin{aligned}
    \tfrac{1}{\eps} \| g \|_{H^N_x L^2_v} \| \nabla_x P g \|_{H^{N-1}_x L^2_v} \| (I-P) g \|_{H^N_{x,v} (\nu)} \lesssim \mathcal{E}_N^\frac{1}{2} (g, \phi) \mathcal{D}_{N,\eps} (g) \,.
  \end{aligned}
\end{equation}
Consequently, we obtain
\begin{equation}\label{W4}
  \begin{aligned}
    W_4 \lesssim \mathcal{E}_N^\frac{1}{2} (g, \phi) \mathcal{D}_{N,\eps} (g)
  \end{aligned}
\end{equation}
for all $0 < \eps \leq 1$. Furthermore, by the similar arguments in the estimate \eqref{S5}, we have
\begin{equation}\label{W5}
  \begin{aligned}
    W_5 \lesssim & \| g \|_{H^N_x L^2_v} \| \nabla_x \phi \|_{H^N_x} \| (I-P) g \|_{\widetilde{H}^N_{x,v}} \\
    \lesssim & \| g \|_{H^N_x L^2_v} \| \l g , 1 \r_{L^2_v} \|_{H^{N-1}_x} \| (I-P) g \|_{H^N_{x,v} (\nu)} \\
    \lesssim & \eps \mathcal{E}_N^\frac{1}{\eps} (g, \phi) \mathcal{D}_{N,\eps} (g) \,,
  \end{aligned}
\end{equation}
where we also use Lemma \ref{Lmm-CF-nu} and the bound $\| \nabla_x \phi \|_{H^N_x} \lesssim \| \phi \|_{H^{N+1}_x} \lesssim \| \l g, 1 \r_{L^2_v} \|_{H^{N-1}_x}$ derived from the Poisson equation $\Delta_x \phi = \gamma \l g , 1 \r_{L^2_v}$. Substituting $P  g = a + b \cdot v + c |v|^2$ into the term $W_6$, one easily has
\begin{equation}\label{W6}
  \begin{aligned}
    W_6 = & - \tfrac{1}{\eps} \l \partial^\beta_v A (v) : \nabla_x \partial^\alpha_x b + \partial^\beta_v B (v) \cdot \nabla_x \partial^\alpha_x c , \partial^\alpha_x \partial^\beta_v (I-P) g \r_{L^2_{x,v}} \\
    \lesssim & \tfrac{1}{\eps} \| \nabla_x (b,c) \|_{H^{N-1}_x} \| \partial^\alpha_x \partial^\beta_v (I-P) g \|_{L^2_{x,v}} \\
    \lesssim & \tfrac{1}{\eps} \| \nabla_x P g \|_{H^{N-1}_x L^2_v} \| \partial^\alpha_x \partial^\beta_v (I-P) g \|_{L^2_{x,v} (\nu)} \,,
  \end{aligned}
\end{equation}
where the inequalities are derived from the H\"older inequality and Lemma \ref{Lmm-CF-nu}. It is easily derived from the same estimates of $S_7$ as in \eqref{S7} that the term $W_7$ can be bounded by
\begin{equation}\label{W7}
  \begin{aligned}
    W_7 \lesssim \tfrac{1}{\eps} \| (I-P) g \|_{H^N_x L^2_v (\nu)} \| \partial^\alpha_x \partial^\beta_v (I-P) g \|_{L^2_{x,v} (\nu)} \,.
  \end{aligned}
\end{equation}
From plugging the bounds \eqref{W1}, \eqref{W2+W3}, \eqref{W4}, \eqref{W5}, \eqref{W6} and \eqref{W7} into the inequality \eqref{W1+W2+W3+W4+W5+W6+W7} and the Young's inequality, we deduce that
\begin{equation}\label{Mixed-Inq-1}
  \begin{aligned}
    & \tfrac{1}{2} \tfrac{\d}{\d t} \| \partial_x^\alpha \partial_v^\beta (I-P) g \|^2_{L^2_{x,v}} + \tfrac{\delta_1}{2 \epsilon^2} \| \partial_x^\alpha \partial_v^\beta (I-P) g \|_{L^2_{x,v}(\nu)}^2 \\
    & \lesssim \tfrac{1}{\eps^2} \sum_{\beta' < \beta} \| \partial_x^\alpha \partial_v^{\beta'} (I-P) g  \|_{L^2_{x,v} (\nu)}^2 + \| \nabla_x P g \|^2_{H^{N-1}_x L^2_v} \\
    & + \tfrac{1}{\eps^2} \| (I-P) g \|^2_{H^N_x L^2_v (\nu)} + \mathcal{E}_N^\frac{1}{2} (g, \phi) \mathcal{D}_{N,\eps} (g)
  \end{aligned}
\end{equation}
for all $|\alpha| + |\beta| \leq N$ with $\beta \neq 0$ and for all $0 < \eps \leq 1$.

Let $\eta > 0$ be sufficiently small number to be determined. We add \eqref{Spatial-Bnd} in Lemma \ref{Lmm-Spatial} and $\eta$ times of \eqref{MM-Inq-Simple} in Remark \ref{Rmk-MM} to the $\eta^2$ times of the above inequality \eqref{Mixed-Inq-1}. Then there is a small $\eta_0 > 0$, independent of $\eps$, such that for all $0 < \eta \leq \eta_0$
\begin{equation}\label{Mixed-Inq-2}
  \begin{aligned}
    & \tfrac{1}{2} \tfrac{\d}{\d t} \Big( \| g \|^2_{H^N_x L^2_v} + \| \nabla_x \phi \|^2_{H^N_x} + c_0 \eta E_N^{int} (t) + \eta^2 \| \partial^\alpha_x \partial^\beta_v (I-P) g \|^2_{L^2_{x,v}} \Big)\\
    & + \tfrac{\delta}{2 \eps^2} \| (I-P) g \|^2_{H^N_x L^2_v (\nu)} + \tfrac{\eta}{2} \| \nabla_x P g \|^2_{H^{N-1}_x L^2_v} \\
    & + \tfrac{\eta}{2} \| \l g , 1 \r_{L^2_v} \|^2_{H^{N-1}_x} + \tfrac{\eta^2 \delta_1}{2 \eps^2} \| \partial^\alpha_x \partial^\beta_x (I-P) g \|^2_{L^2_{x,v} (\nu)} \\
    & \lesssim \tfrac{1}{\eps^2} \sum_{\beta' < \beta} \| \partial^\alpha_x \partial^{\beta'}_v (I-P) g \|^2_{L^2_{x,v} (\nu)} + \mathcal{E}_N^\frac{1}{2} (g,  \phi) D_{N, \eps} (g)
  \end{aligned}
\end{equation}
for all $|\alpha| + |\beta| \leq N$ with $\alpha \neq 0$ and for all $0 < \eps \leq 1$.

One notices that the quantity $ \tfrac{1}{\eps^2} \sum_{\beta' < \beta} \| \partial^\alpha_x \partial^{\beta'}_v (I-P) g \|^2_{L^2_{x,v} (\nu)} $ in the right-hand side of \eqref{Mixed-Inq-2} is still not controlled. However, we observe that the orders of $v$-derivatives in this quantity is strictly less than $|\beta|$, so that we can employ an induction over $|\beta| = k$, which ranges between $0$ and $N$, to obtain the inequality \eqref{Mixed-Inq-Closed}. For simplicity, we omit the details of the induction, and the proof of Lemma \ref{Lmm-Mixed-Est} is finished.
\end{proof}

\subsection{Global classical solutions: proof of Theorem \ref{Thm-global}}

From the differential inequality \eqref{Closed-Energy-Inq} in Remark \ref{Rmk-Mixed} and the energy bound \eqref{2.3} in Lemma \ref{Lmm-Local}, we deduce that for any $[t_1, t_2] \subseteq [0,T]$ and $0 < \eps \leq 1$
\begin{equation}\no
  \begin{aligned}
    & \big| \mathbb{E}_{N, \eta_1} (g_\eps, \phi_\eps) (t_2) - \mathbb{E}_{N, \eta_1} (g_\eps, \phi_\eps) (t_1) \big| \lesssim \int_{t_1}^{t_2} \mathbb{E}_{N,\eta_1}^\frac{1}{2} (g_\eps, \phi_\eps) \mathbb{D}_{N, \eps, \eta_1} (g_\eps) \d t \\
    & \lesssim \sup_{0 \leq t \leq T} \mathcal{E}_N^\frac{1}{2} (g_\eps, \phi_\eps) \int_{t_1}^{t_2} \mathcal{D}_{N,\eps} (g_\eps) \d t \lesssim \int_{t_1}^{t_2} \tfrac{1}{\eps^2} \| (I-P) g_\eps \|^2_{H^N_{x,v} (\nu)} \d t + |t_2 - t_1 | \\
    & \rightarrow 0 \quad \textrm{as} \quad t_2 \rightarrow t_1 \,.
  \end{aligned}
\end{equation}
Thus the local solution $g_\eps (t,x,v)$ to \eqref{VPB-g}-\eqref{IC-g} constructed in Lemma \ref{Lmm-Local} is such that the energy functional $ \mathbb{E}_{N, \eta_1} (g_\eps, \phi_\eps) $ is continuous in $t \in [0,T]$.

We now define
\begin{equation}\no
  \begin{aligned}
    T^* = \sup \Big\{ \tau \geq 0 ; \ C \sup_{t \in [0, \tau]} \mathbb{E}_{N, \eta_1}^\frac{1}{2} (g_\eps, \phi_\eps) (t) \leq \tfrac{1}{2} \Big\} \geq 0 \,.
  \end{aligned}
\end{equation}
From $ \mathbb{E}_{N, \eta_1} (g_{\eps, 0} , \phi_{\eps, 0}) \thicksim \mathcal{E}_N (g_{\eps, 0}, \phi_{\eps, 0}) $ and the initial condition in Theorem \ref{Thm-global}, we derive that
\begin{equation}\no
  \begin{aligned}
    \mathbb{E}_{N, \eta_1} (g_{\eps, 0} , \phi_{\eps, 0}) \leq C_0 \mathcal{E}_N (g_{\eps, 0}, \phi_{\eps, 0}) \leq C_0 \ell_0
  \end{aligned}
\end{equation}
for some constant $C_0 > 0$, where $\ell_0 \in (0,1]$ is small to be determined. If we take $0 < \ell_0 \leq \min \{ 1 , \delta , \tfrac{1}{16 C^2 C_0} \}$, where $\delta > 0$ is mentioned in Lemma \ref{Lmm-Local}, we have
\begin{equation}\label{Initial-Bnd}
  \begin{aligned}
    C \mathbb{E}_{N, \eta_1}^\frac{1}{2} (g_\eps, \phi_\eps) (0) = C \mathbb{E}_{N, \eta_1}^\frac{1}{2} (g_{\eps, 0}, \phi_{\eps, 0}) \leq C \sqrt{C_0 \ell_0} \leq \tfrac{1}{4} < \tfrac{1}{2} \,.
  \end{aligned}
\end{equation}
Then the continuity of $\mathbb{E}_{N, \eta_1} (g_\eps, \phi_\eps) (t)$ implies that $T^* > 0$. Consequently, we derive from the definition of $T^*$ and the inequality \eqref{Closed-Energy-Inq} in Remark \ref{Rmk-Mixed} that for all $t \in [0, T^*]$ and $0 < \eps \leq 1$
\begin{equation}\no
  \begin{aligned}
    \tfrac{\d}{\d t} \mathbb{E}_{N, \eta_1} (g_\eps , \phi_\eps) + \mathbb{D}_{N, \eps, \eta_1} (g_\eps) \leq 0 \,.
  \end{aligned}
\end{equation}
From integrating the above inequality on $[0,t]$ for any $t \in [0, T^*]$, we deduce that
\begin{equation}\label{Uniform-Global-Bnd-Inst}
  \begin{aligned}
    \mathbb{E}_{N,\eta_1} (g_\eps, \phi_\eps) (t) + \int_0^t \mathbb{D}_{N, \eps, \eta_1} (g_\eps) (\tau) \d \tau \leq \mathbb{E}_{N, \eta_1} (g_{\eps, 0} , \phi_{\eps, 0}) \leq C_0 \ell_0
  \end{aligned}
\end{equation}
uniformly for all $0 < \eps \leq 1$, which immediately implies by the initial bound \eqref{Initial-Bnd} that
\begin{equation}\no
  \begin{aligned}
    C \sup_{t \in [0, \tau]} \mathbb{E}_{N, \eta_1}^\frac{1}{2} (g_\eps, \phi_\eps) (t) \leq \tfrac{1}{4} < \tfrac{1}{2} \,.
  \end{aligned}
\end{equation}
Thus, the continuity of $\mathbb{E}_{N, \eta_1} (g_\eps, \phi_\eps) (t)$ and the definition of $T^*$ imply that $T^* = + \infty$. In other words, the local solutions $g_\eps (t,x,v)$ constructed in Lemma \ref{Lmm-Local} can be extended globally. Moreover, the uniform energy bound \eqref{Uniform-Global-Bnd} can be derived from \eqref{Equivalences-Energies} and \eqref{Uniform-Global-Bnd-Inst}. Then the proof of Theorem \ref{Thm-global} is finished.  \qquad\qquad\qquad\qquad\qquad\qquad\qquad\qquad\qquad\qquad\qquad\qquad\ \ \  $\square$

\section{Limit to Incompressible NSFP Equations}\label{Sec: Limits}

In this section, based on the uniform global energy bound \eqref{Uniform-Global-Bnd} in Theorem \ref{Thm-global}, we aim at deriving the incompressible NSFP equations \eqref{NSFP} from the perturbed VPB system \eqref{VPB-g}-\eqref{IC-g} as $\eps \rightarrow 0$.

\subsection{Limits from the global energy estimate}

From Theorem \ref{Thm-global}, we deduce that the Cauchy problem \eqref{VPB-g}-\eqref{IC-g} admits a global solution $g_\eps (t,x) \in L^\infty (\R^+; H^N_{x,v}) $, which subjects to the uniform global energy estimate \eqref{Uniform-Global-Bnd}, namely, there is a positive constant $C$, independent of $\eps$, such that
\begin{equation}\label{Unf-Bnd-1}
  \begin{aligned}
    \sup_{t \geq 0} \big( \| g_\eps (t) \|^2_{H^N_{x,v}} + \| \nabla_x \phi_\eps (t) \|^2_{H^N_x}  \big) \leq C \,,
  \end{aligned}
\end{equation}
and
\begin{equation}\label{Unf-Bnd-2}
  \begin{aligned}
    \int_0^\infty \| (I-P) g_\eps (t) \|^2_{H^N_{x,v} (\nu)} \d t \leq C \eps^2 \,.
  \end{aligned}
\end{equation}
From the uniform energy bound \eqref{Unf-Bnd-1}, there are $g (t,x,v) \in L^\infty (\R^+ ; H^N_{x,v})$ and $\Phi (t,x) \in L^\infty (\R^+ ; H^N_x)$ such that
\begin{equation}\label{Convgc-g-phi}
  \begin{aligned}
    & g_\eps \rightarrow g \quad \quad \ \ \textrm{weakly-}\star \ \textrm{for } t \geq 0, \ \quad \textrm{weakly in } H^N_{x,v} \,, \\
    & \nabla_x \phi_\eps \rightarrow \Phi \quad \textrm{weakly-}\star \ \textrm{for } t \geq 0, \ \quad \textrm{weakly in } H^N_x \textrm{ and strongly in } H^{N-\sigma}_x \textrm{ locally}
  \end{aligned}
\end{equation}
for any $\sigma > 0$ as $\eps \rightarrow 0$. Since $\nabla_x \phi_\eps$ is the form of a gradient, the function $\Phi$ should also be the same form, i.e., there is a function $\phi (t,x) \in L^\infty (\R^+; H^{N+1}_x)$ such that $\Phi = \nabla_x \phi$. The limits may hold for some subsequences. But, for convenience, we still employ the original notations of the sequences to denote by the subsequences throughout this paper. From the energy dissipation bound \eqref{Unf-Bnd-2} and the inequality $\| (I-P) g \|^2_{H^N_{x,v}} \lesssim \| (I-P) g \|^2_{H^N_{x,v} (\nu)} $ implied by Lemma \ref{Lmm-CF-nu}, we have
\begin{equation}\label{Convgc-g-perp}
  \begin{aligned}
    (I-P) g_\eps \rightarrow 0 \quad \textrm{strongly in } \ L^2(\R^+; H^N_{x,v})
  \end{aligned}
\end{equation}
as $\eps \rightarrow 0$. We thereby deduce from combining the first convergence in \eqref{Convgc-g-phi} and \eqref{Convgc-g-perp} that
\begin{equation}\no
  \begin{aligned}
    (I-P) g = 0 \,,
  \end{aligned}
\end{equation}
which immediately means that there are functions $\rho (t,x)$, $u (t,x)$, $\theta (t,x) \in L^\infty (\R^+; H^N_x)$ such that
\begin{equation}\label{Limit-g}
  \begin{aligned}
    g(t,x,v) = \rho (t,x) + u (t,x) \cdot v + \theta (t,x) ( \tfrac{|v|^2}{2} - \tfrac{3}{2} ) \,.
  \end{aligned}
\end{equation}

We now introduce the following fluid variables
\begin{equation}
  \begin{aligned}
    \rho_\eps = \l g_\eps , 1 \r_{L^2_v} \,, \ u_\eps = \l g_\eps , v \r_{L^2_v} , \theta_\eps = \l g_\eps , \tfrac{|v|^2}{3} - 1 \r_{L^2_v} \,.
  \end{aligned}
\end{equation}
Taking inner products with the perturbed VPB equations \eqref{VPB-g} in $L^2_v$ by $1$, $v$ and $ \tfrac{|v|^2}{3}-1 $, respectively, gives the local conservation laws:
\begin{equation}\label{Consevtn-Law-rho-u-theta-phi}
  \begin{aligned}
    & \partial_t \rho_\epsilon + \tfrac{1}{\epsilon} \nabla_x \cdot u_\epsilon = 0 \,, \\
    & \partial_t u_\epsilon + \tfrac{1}{\epsilon} \nabla_x ( \rho_\epsilon + \theta_\epsilon ) + \nabla_x \cdot
    \l \widehat{A}, \tfrac{1}{\epsilon} L (I-P) g_\epsilon \r_{L^2_v} - \tfrac{\gamma}{\epsilon} \nabla_x \phi_\epsilon - \gamma \rho_\epsilon \nabla_x \phi_\epsilon = 0 \,, \\
    & \partial_t \theta_\eps + \tfrac{2}{3 \eps} \nabla_x \cdot u_\epsilon + \tfrac{2}{3} \nabla_x \cdot \l \widehat{B}, \tfrac{1}{\epsilon} L (I-P) g_\epsilon \r_{L^2_v} -\tfrac{2}{3} \gamma u_\epsilon \cdot \nabla_x \phi_\epsilon = 0 \,, \\
    & \Delta_x \phi_\eps = \gamma \rho_\eps \,.
  \end{aligned}
\end{equation}
Moreover, the uniform bound \eqref{Unf-Bnd-1} tells us that
\begin{equation}\label{Unf-Bnd-3}
  \begin{aligned}
    \sup_{t \geq 0} \big( \| \rho_\eps \|_{H^N_x} + \| u_\eps \|_{H^N_x} + \| \theta_\eps \|_{H^N_x} \big) \leq C \,.
  \end{aligned}
\end{equation}
We thereby deduce the following convergences from the convergences \eqref{Convgc-g-phi} and the form of limit function $g (t,x,v)$ given in \eqref{Limit-g} that
\begin{equation}\label{Convgn-rho-u-theta}
  \begin{aligned}
    & \rho_\eps = \l g_\eps , 1 \r_{L^2_v} \rightarrow \l g , 1 \r_{L^2_v} = \rho \,, \\
    & u_\eps = \l g_\eps , v \r_{L^2_v} \rightarrow \l g , v \r_{L^2_v} = u \,, \\
    & \theta_\eps = \l g_\eps , \tfrac{|v|^2}{3} - 1 \r_{L^2_v} \rightarrow \l g , \tfrac{|v|^2}{3} - 1 \r_{L^2_v} = \theta \,,
  \end{aligned}
\end{equation}
weakly-$\star$ for $t \geq 0$, weakly in $H^N_x$ and strongly in $H^{N-\sigma}_x$ locally for any $\sigma > 0$ as $\eps \rightarrow 0$.

\subsection{Convergences to limit equations}

In this subsection, we will derive the incompressible NSFP equations \eqref{NSFP} from the local conservation laws \eqref{Consevtn-Law-rho-u-theta-phi} and the convergences obtained in the previous subsection.

\subsubsection{Incompressibility and Boussinesq relation}

From the first equation of \eqref{Consevtn-Law-rho-u-theta-phi} and the uniform bound \eqref{Unf-Bnd-3}, it is easy to deduce that
\begin{equation}\no
  \begin{aligned}
    \nabla_x \cdot u_\eps = - \eps \partial_t \rho_\eps \rightarrow 0
  \end{aligned}
\end{equation}
in the sense of distribution as $\eps \rightarrow 0$, which implies that by combining with the convergence \eqref{Convgn-rho-u-theta}
\begin{equation}\label{Incompressibility}
  \begin{aligned}
    \nabla_x \cdot u = 0 \,.
  \end{aligned}
\end{equation}

Via the second $u_\eps$-equation of \eqref{Consevtn-Law-rho-u-theta-phi},  we have
\begin{equation}\no
  \begin{aligned}
    \nabla_x ( \rho_\eps + \theta_\eps - \gamma \phi_\eps ) = - \eps \partial_t u_\eps + \gamma \eps \rho_\eps \nabla_x \phi_\eps - \nabla_x \cdot \l \widehat{A} , L (I-P) g_\eps \r_{L^2_v} \,.
  \end{aligned}
\end{equation}
The bound \eqref{Unf-Bnd-3} tells us that $- \eps \partial_t u_\eps \rightarrow 0$ in the sense of distribution as $\eps \rightarrow 0$. We derive from the bound \eqref{Unf-Bnd-1} and \eqref{Unf-Bnd-3} that
\begin{equation*}
  \begin{aligned}
    \| \gamma \eps \rho_\eps \nabla_x \phi_\eps \|_{H^N_x} \lesssim \eps \| \rho_\eps \|_{H^N_x} \| \nabla_x \phi_\eps \|_{H^N_x} \lesssim \eps \,,
  \end{aligned}
\end{equation*}
which means that
\begin{equation}\no
  \begin{aligned}
    \gamma \eps \rho_\eps \nabla_x \phi_\eps \rightarrow 0
  \end{aligned}
\end{equation}
strongly in $L^\infty (\R^+ ; H^N_x)$ as $\eps \rightarrow 0$. It is further derived from the uniform energy dissipation bound \eqref{Unf-Bnd-2} and Lemma \ref{Lmm-CF-nu} that
\begin{equation}\no
  \begin{aligned}
    \int_0^\infty \| \nabla_x \cdot \l \widehat{A} , L (I-P) g_\eps \r_{L^2_v} \|^2_{H^{N-1}_x} \d t \lesssim \int_0^\infty \| (I-P) g \|^2_{H^N_{x,v} (\nu)} \d t \lesssim \eps^2 \,.
  \end{aligned}
\end{equation}
We thereby obtain that
\begin{equation}\no
  \begin{aligned}
    \nabla_x \cdot \l \widehat{A} , L (I-P) g_\eps \r_{L^2_v} \rightarrow 0
  \end{aligned}
\end{equation}
strongly in $L^2 (\R^+; H^{N-1}_x)$ as $\eps \rightarrow 0$. Consequently, we have proved that
\begin{equation*}
  \begin{aligned}
    \nabla_x (\rho_\eps + \theta_\eps - \gamma \phi_\eps) \rightarrow 0
  \end{aligned}
\end{equation*}
in the sense of distribution as $\eps \rightarrow 0$, which, combining with the convergences \eqref{Unf-Bnd-1} and \eqref{Unf-Bnd-3}, gives the Boussinesq relation
\begin{equation}\label{Boussinesq-Reltn}
  \begin{aligned}
    \nabla_x ( \rho + \theta - \gamma \phi ) = 0 \,.
  \end{aligned}
\end{equation}

\subsubsection{Convergences of $\tfrac{3}{5} \theta_\eps - \tfrac{2}{5} \rho_\eps$ and $\mathcal{P} u_\eps$}

Before doing this, we introduce the following Aubin-Lions-Simon Theorem, a fundamental result of compactness in the study of nonlinear evolution problems, which can be referred to Theorem II.5.16 of \cite{Boyer-Fabrie-2013-BOOK} or \cite{Simon-1987-AMPA}, for instance.

\begin{lemma}[Aubin-Lions-Simon Theorem]\label{Lmm-Aubin-Lions-Simon}
	Let $B_0 \subset B_1 \subset B_2$ be three Banach spaces. We assume that the embedding of $B_1$ in $B_2$ is continuous and that the embedding of $B_0$ in $B_1$ is compact. Let $p$, $r$ be such that $1 \leq p, r \leq + \infty$. For $T > 0$, we define
	\begin{equation*}
	  \begin{aligned}
	    E_{p,r} = \big\{ u \in L^p (0,T; B_0), \ \partial_t u \in L^r (0,T; B_2) \big\} \,.
	  \end{aligned}
	\end{equation*}
	\begin{enumerate}
		\item If $p < + \infty$, the embedding of $E_{p,r}$ in $L^p (0,T; B_1)$ is compact.
		
		\item If $p = + \infty$ and $r > 1$, the embedding of $E_{p,r}$ in $C(0,T; B_1)$ is compact.
	\end{enumerate}
\end{lemma}

We now consider the convergence of $\tfrac{3}{5} \theta_\eps - \tfrac{2}{5} \rho_\eps$. The third $\theta_\eps$-equation of \eqref{Consevtn-Law-rho-u-theta-phi} multiplied by $ \tfrac{3}{5} $ minus $ \tfrac{2}{5} $ times of the first $\rho_\eps$-equation gives
\begin{equation}\label{Equ-theta-rho}
  \begin{aligned}
    \partial_t ( \tfrac{3}{5} \theta_\eps - \tfrac{2}{5} \rho_\eps ) + \tfrac{2}{5} \nabla_x \cdot \l \widehat{B} , \tfrac{1}{\eps} L (I-P) g_\eps \r_{L^2_v} - \tfrac{2}{5} \gamma u_\eps \cdot \nabla_x \phi_\eps = 0 \,,
  \end{aligned}
\end{equation}
which yields that
\begin{equation*}
  \begin{aligned}
    \| \partial_t ( \tfrac{3}{5} \theta_\eps - \tfrac{2}{5} \rho_\eps ) \|_{H^{N-1}_x} \lesssim \tfrac{1}{\eps} \| B \|_{L^2_v} \| \nabla_x (I-P) g_\eps \|_{H^{N-1}_x L^2_v } \lesssim \tfrac{1}{\eps} \| (I-P) g_\eps \|_{H^N_x L^2_v (\nu)} \,.
  \end{aligned}
\end{equation*}
One then derives from the uniform energy bounds \eqref{Unf-Bnd-2} that
\begin{equation}\label{Bnd-rho-theta-t}
  \begin{aligned}
    \| \partial_t ( \tfrac{3}{5} \theta_\eps - \tfrac{2}{5} \rho_\eps ) \|_{L^2 (\R^+; H^{N-1}_x)} \lesssim \tfrac{1}{\eps} \| (I-P) g_\eps \|_{L^2(\R^+ ; H^N_{x,v} (\nu))} \lesssim 1
  \end{aligned}
\end{equation}
for all $0 < \eps \leq 1$. It is easily derived from the uniform bound \eqref{Unf-Bnd-3} that
\begin{equation}\label{Bnd-rho-theta}
  \begin{aligned}
    \| \tfrac{3}{5} \theta_\eps - \tfrac{2}{5} \rho_\eps \|_{L^\infty (\R^+; H^N_x)} \lesssim 1
  \end{aligned}
\end{equation}
for all $0 < \eps \leq 1$. One notices that
\begin{equation}\label{Embeddings}
  \begin{aligned}
    H^N_x \hookrightarrow H^{N-1}_x \hookrightarrow H^{N-1}_x \,,
  \end{aligned}
\end{equation}
where the embedding of $H^N_x$ in $H^{N-1}_x$ is compact and the embedding of $H^{N-1}_x $ in $H^{N-1}_x$ is naturally continuous. Then, from Aubin-Lions-Simon Theorem in Lemma \ref{Lmm-Aubin-Lions-Simon}, the bounds \eqref{Bnd-rho-theta-t}, \eqref{Bnd-rho-theta} and the embeddings \eqref{Embeddings}, we deduce that there is a $\widetilde{\theta} (t,x) \in L^\infty (\R^+; H^N_x) \cap C (\R^+ ; H^{N-1}_x)$ such that
\begin{equation}\no
  \begin{aligned}
    \tfrac{3}{5} \theta_\eps - \tfrac{2}{5} \rho_\eps \rightarrow \widetilde{\theta}
  \end{aligned}
\end{equation}
strongly in $C(\R^+; H^{N-1}_x)$ as $\eps \rightarrow 0$. Combining with the convergences \eqref{Convgn-rho-u-theta}, we know that $ \widetilde{\theta} = \tfrac{3}{5} \theta - \tfrac{2}{5} \rho $. Consequently, we have
\begin{equation}\label{Convgnc-rho-theta}
  \begin{aligned}
    \tfrac{3}{5} \theta_\eps - \tfrac{2}{5} \rho_\eps \rightarrow \tfrac{3}{5} \theta - \tfrac{2}{5} \rho
  \end{aligned}
\end{equation}
strongly in $C(\R^+; H^{N-1}_x)$ as $\eps \rightarrow 0$, where $ \tfrac{3}{5} \theta - \tfrac{2}{5} \rho \in L^\infty (\R^+; H^N_x) \cap C (\R^+ ; H^{N-1}_x)  $.

Next we consider the convergence of $\mathcal{P} u_\eps$, where $\mathcal{P}$ is the Leray projection on $\R^3$. Taking $\mathcal{P}$ on the second $u_\eps$-equation of \eqref{Consevtn-Law-rho-u-theta-phi} gives
\begin{equation}\label{Equ-Pu}
  \begin{aligned}
    \partial_t \mathcal{P} u_\eps + \mathcal{P} \nabla_x \cdot \l \widehat{A} , \tfrac{1}{\eps} L (I-P) g_\eps \r_{L^2_v} - \mathcal{P} ( \gamma \rho_\eps \nabla_x \phi_\eps ) = 0 \,.
  \end{aligned}
\end{equation}
It is easy to deduce from the H\"older inequality, the bound $\| A \|_{L^2_v} \lesssim 1$, the calculus inequality and Lemma \ref{Lmm-CF-nu} that
\begin{equation*}
  \begin{aligned}
    \| \partial_t \mathcal{P} u_\eps \|_{H^{N-1}_x} \lesssim & \tfrac{1}{\eps} \| A \|_{L^2_v} \| \nabla_x (I-P) g_\eps \|_{H^{N-1}_x L^2_v} + \| \rho_\eps \nabla_x \phi_\eps \|_{H^{N-1}_x} \\
    \lesssim & \tfrac{1}{\eps} \| (I-P) g \|_{H^N_x L^2_v (\nu)} + \| \rho_\eps \|_{H^N_x} \| \nabla_x \phi_\eps \|_{H^N_x} \,,
  \end{aligned}
\end{equation*}
which, by the uniform bounds \eqref{Unf-Bnd-1}, \eqref{Unf-Bnd-2} and \eqref{Unf-Bnd-3}, implies that
\begin{equation}\label{Bnd-Pu-t}
  \begin{aligned}
    \| \partial_t \mathcal{P} u_\eps \|_{L^2 (0, T ; H^{N-1}_x)} \lesssim & \tfrac{1}{\eps} \| (I-P) g_\eps \|_{H^N_{x,v} (\nu)} + \| \rho_\eps \|_{L^\infty (\R^+; H^N_x)} \| \nabla_x \phi_\eps \|_{L^\infty (\R^+ ; H^N_x)} T
  \end{aligned}
\end{equation}
for any $T > 0$ and $0 < \eps \leq 1$. Furthermore, the bound \eqref{Unf-Bnd-3} tells us that for all $T > 0$ and $0 < \eps \leq 1$
\begin{equation}\label{Bnd-Pu}
  \begin{aligned}
    \| \mathcal{P} u_\eps \|_{L^\infty (0,T; H^N_x)} \lesssim 1 \,.
  \end{aligned}
\end{equation}
Then, from Aubin-Lions-Simon Theorem in Lemma \ref{Lmm-Aubin-Lions-Simon}, the bounds \eqref{Bnd-Pu-t}, \eqref{Bnd-Pu} and the embeddings \eqref{Embeddings}, we derive that there is a $\widetilde{u} \in L^\infty (\R^+ ; H^N_x) \cap C (\R^+ ; H^{N-1}_x)$ such that
\begin{equation*}
  \begin{aligned}
    \mathcal{P} u_\eps \rightarrow \widetilde{u}
  \end{aligned}
\end{equation*}
strongly in $ C (0, T; H^{N-1}_x) $ for all $T > 0$ as $\eps \rightarrow 0$. Furthermore, from the convergences \eqref{Convgn-rho-u-theta} and incompressibility \eqref{Incompressibility}, we deduce that $ \widetilde{u} = \mathcal{P} u = u $. Consequently,
\begin{equation}\label{Convgnc-Pu}
  \begin{aligned}
    \mathcal{P} u_\eps \rightarrow u
  \end{aligned}
\end{equation}
strongly in $ C (\R^+; H^{N-1}_x) $ as $\eps \rightarrow 0$, where $u \in L^\infty (\R^+; H^N_x) \cap C (\R^+; H^{N-1}_x)$. We thereby have
\begin{equation}\label{Convgnc-Pu-perp}
  \begin{aligned}
    \mathcal{P}^\perp u_\eps \rightarrow 0
  \end{aligned}
\end{equation}
weakly-$\star$ in $t \geq 0$, weakly in $H^N_x$ and strongly in $H^{N-\sigma}_x$ locally for any $\sigma > 0$ as $\eps \rightarrow 0$, where $\mathcal{P}^\perp$ is the orthogonal projection of $\mathcal{P}$ in $L^2_x$ with the form $\mathcal{P}^\perp = \nabla_x \Delta_x^{-1} \nabla_x \cdot$.

\subsubsection{Equations of $\rho$, $u$ and $\theta$}

Based on the all convergences obtained in the previous subsection, we will deduce the incompressible NSFP system from \eqref{Equ-Pu}, \eqref{Equ-theta-rho} and the last Poisson equation of \eqref{Consevtn-Law-rho-u-theta-phi}. We first calculate the term
\begin{equation*}
  \begin{aligned}
    \l \widehat{\Xi} , \tfrac{1}{\eps} L (I-P) g_\eps \r_{L^2_v} \,,
  \end{aligned}
\end{equation*}
where $\Xi = A$ or $B$. Following the standard formal derivation of fluid dynamic limits of Boltzmann equation (see \cite{BGL1}, for instance), we obtain
\begin{equation}\label{A-Dissipation}
  \begin{aligned}
    \l \widehat{A} , \tfrac{1}{\eps} L (I-P) g_\eps \r_{L^2_v} = u_\eps \otimes u_\eps - \tfrac{|u_\eps|^2}{3} I - \mu \Sigma (u_\eps) - R_{\eps, A} \,,
  \end{aligned}
\end{equation}
and
\begin{equation}\label{B-Dissipation}
  \begin{aligned}
    \l \widehat{B} , \tfrac{1}{\eps} L (I-P) g_\eps \r_{L^2_v} = \tfrac{5}{2} u_\eps \theta_\eps - \tfrac{5}{2} \kappa \nabla_x \theta_\eps - R_{\eps, B} \,,
  \end{aligned}
\end{equation}
where $\Sigma (u_\eps) = \nabla_x u_\eps + \nabla_x u_\eps^\top - \tfrac{2}{3} \nabla_x \cdot u_\eps I$, $\mu$, $\kappa$ are given in \eqref{mu-kappa}, and for $\Xi = A$ or $B$, $R_{\eps, \Xi}$ are of the form
\begin{equation}\label{R-eps-Xi}
  \begin{aligned}
    R_{\eps, \Xi} = & \l \Xi , -\eps \partial_t g_\eps - v \cdot \nabla_x (I-P) g_\eps + Q ( (I-P) g_\eps , P g_\eps ) + Q ( (I-P) g_\eps , (I-P) g_\eps ) \\
    & \qquad + Q ( P g_\eps , (I-P) g_\eps ) - \gamma \eps \nabla_x \phi_\eps \cdot \nabla_v g_\eps + \gamma \eps v \cdot \nabla_x \phi_\eps g_\eps \r_{L^2_v} \,.
  \end{aligned}
\end{equation}

From plugging the relation \eqref{A-Dissipation} and decomposition $u_\eps = \mathcal{P} u_\eps + \mathcal{P}^\perp u_\eps$ into \eqref{Equ-Pu}, we deduce that
\begin{equation}\label{Evolution-Pu}
  \begin{aligned}
    \partial_t \mathcal{P} u_\eps + \mathcal{P} \nabla_x \cdot ( \mathcal{P} u_\eps \otimes \mathcal{P} u_\eps ) - \mu \Delta_x \mathcal{P} u_\eps = \tfrac{5}{3} \mathcal{P} [ \, \rho_\eps \nabla_x ( \tfrac{3}{5} \theta_\eps - \tfrac{2}{5} \rho_\eps ) ] + R_{\eps, u} \,,
  \end{aligned}
\end{equation}
where
\begin{equation}\label{R-eps-u}
  \begin{aligned}
    R_{\eps, u} = & \mathcal{P} \nabla_x \cdot R_{\eps, A}  + \mathcal{P} [ \rho_\eps \nabla_x (\gamma \phi_\eps - \rho_\eps - \theta_\eps) ] \\
    - & \mathcal{P} \nabla_x \cdot ( \mathcal{P} u_\eps \otimes \mathcal{P}^\perp u_\eps + \mathcal{P}^\perp \otimes \mathcal{P} u_\eps + \mathcal{P}^\perp u_\eps \otimes \mathcal{P}^\perp u_\eps ) \,.
  \end{aligned}
\end{equation}
Here we make use of $ \mathcal{P} (\rho_\eps \nabla_x \rho_\eps) = \mathcal{P} \nabla_x \cdot ( \tfrac{|u_\eps|^2}{3} I + \nabla_x u_\eps^\top - \tfrac{2}{3} \nabla_x \cdot u_\eps I ) = 0 $. If we substitute the relation \eqref{B-Dissipation} and splitting $u_\eps = \mathcal{P} u_\eps + \mathcal{P}^\perp u_\eps$ into the evolution \eqref{Equ-theta-rho}, we yield that
\begin{equation}\label{Evolution-theta-rho}
  \begin{aligned}
    \partial_t ( \tfrac{3}{5} \theta_\eps - \tfrac{2}{5} \rho_\eps ) + \mathcal{P} u_\eps \cdot \nabla_x ( \tfrac{3}{5} \theta_\eps - \tfrac{2}{5} \rho_\eps ) - \kappa \Delta_x \theta_\eps = R_{\eps, \theta} \,,
  \end{aligned}
\end{equation}
where
\begin{equation}\label{R-eps-theta}
  \begin{aligned}
    R_{\eps, \theta} = \tfrac{2}{5} \nabla_x \cdot R_{\eps, B} + \tfrac{2}{5} \gamma \mathcal{P}^\perp u_\eps \cdot \nabla_x \phi_\eps - \nabla_x \cdot ( \mathcal{P}^\perp u_\eps \theta_\eps ) - \tfrac{2}{5} \mathcal{P} u_\eps \cdot \nabla_x ( \rho_\eps + \theta_\eps - \gamma \phi_\eps ) \,.
  \end{aligned}
\end{equation}

Now we take the limit from \eqref{Evolution-Pu} to obtain the $u$-equation of \eqref{NSFP}. For any $T > 0$, let a vector-valued text function $\psi (t,x) \in C^1 (0,T; C_c^\infty (\R^3)) $ with $\nabla_x \cdot \psi = 0$, $\psi (0,x) = \psi_0 (x) \in C_c^\infty (\R^3)$ and $\psi (t,x) = 0$ for $t \geq T'$, where $T' < T$. Then we obtain
\begin{equation}\no
  \begin{aligned}
    & \int_0^T \int_{\R^3} \partial_t \mathcal{P} u_\eps \cdot \psi (t,x) \d x \d t \\
    = & - \int_{\R^3} \mathcal{P} u_\eps (0,x) \cdot \psi (0,x) \d x - \int_0^T \int_{\R^3} \mathcal{P} u_\eps \cdot \partial_t \psi (t,x) \d x \d t \\
    = & - \int_{\R^3} \mathcal{P}\l g_{\eps, 0} , v \r_{L^2_v} \cdot \psi_0 (x) \d x - \int_0^T \int_{\R^3} \mathcal{P} u_\eps \cdot \partial_t \psi (t,x) \d x \d t \,.
  \end{aligned}
\end{equation}
From the initial conditions in Theorem \ref{Thm-Limit} and the convergence \eqref{Convgnc-Pu}, we deduce that
\begin{equation}\no
  \begin{aligned}
    \int_{\R^3} \mathcal{P}\l g_{\eps, 0} , v \r_{L^2_v} \cdot \psi_0 (x) \d x \rightarrow \int_{\R^3} \mathcal{P} \l g_0 , v \r_{L^2_v} \cdot \psi_0 (x) \d x = \int_{\R^3} \mathcal{P} u_0 (x) \cdot \psi_0 (x) \d x
  \end{aligned}
\end{equation}
and
\begin{equation}\no
  \begin{aligned}
    \int_0^T \int_{\R^3} \mathcal{P} u_\eps \cdot \partial_t \psi (t,x) \d x \d t \rightarrow \int_0^T \int_{\R^3} u \cdot \partial_t \psi (t,x) \d x \d t
  \end{aligned}
\end{equation}
as $\eps \rightarrow 0$. Namely, we have
\begin{equation}\label{Converge-Pu-1}
  \begin{aligned}
    \int_0^T \int_{\R^3} \partial_t \mathcal{P} u_\eps \cdot \psi (t,x) \d x \d t \rightarrow - \int_{\R^3} \mathcal{P} u_0 (x) \cdot \psi_0 (x) \d x - \int_0^T \int_{\R^3} u \cdot \partial_t \psi (t,x) \d x \d t
  \end{aligned}
\end{equation}
as $\eps \rightarrow 0$. It is also easily implied by the strong convergences \eqref{Convgnc-Pu} that
\begin{equation}\label{Converge-Pu-2}
  \begin{aligned}
    & \mathcal{P} \nabla_x \cdot ( \mathcal{P} u_\eps \otimes \mathcal{P} u_\eps ) \rightarrow \mathcal{P} \nabla_x \cdot (u \otimes u) \quad \textrm{strongly in } \quad C(\R^+; H^{N-2}_x) \,, \\
    & \mu \Delta_x \mathcal{P} u_\eps \rightarrow \mu \Delta_x u \quad \textrm{strongly in } \quad C(\R^+; H^{N-3}_x)
  \end{aligned}
\end{equation}
as $\eps \rightarrow 0$. From the bound \eqref{Unf-Bnd-3} and strong convergence \eqref{Convgnc-rho-theta}, we know that

\begin{equation}\label{Converge-Pu-3}
  \begin{aligned}
    \tfrac{5}{3} \mathcal{P} \big[ \rho_\eps \nabla_x ( \tfrac{3}{5} \theta_\eps - \tfrac{2}{5} \rho_\eps ) \big] \rightarrow \tfrac{5}{3} \mathcal{P} \big[ \rho \nabla_x ( \tfrac{3}{5} \theta - \tfrac{2}{5} \rho ) \big] = \mathcal{P} (\rho \nabla_x \theta)
  \end{aligned}
\end{equation}
weakly-$\star$ for $t \geq 0$, weakly in $H^N_x$ and strongly in $H^{N-1}_x$ locally as $\eps \rightarrow 0$.

It remains to prove
\begin{equation}\label{Converge-Pu-4}
   \begin{aligned}
     R_{\eps, u} \rightarrow 0
   \end{aligned}
\end{equation}
in the sense of distribution as $\eps \rightarrow 0$, where $R_{\eps, u}$ is defined in \eqref{R-eps-u}. Indeed, by employing the convergences \eqref{Convgnc-Pu} and \eqref{Convgnc-Pu-perp}, one can obtain
\begin{equation}
  \begin{aligned}
    \mathcal{P} \nabla_x \cdot ( \mathcal{P} u_\eps \otimes \mathcal{P}^\perp u_\eps + \mathcal{P}^\perp u_\eps \otimes \mathcal{P} u_\eps + \mathcal{P}^\perp u_\eps \otimes \mathcal{P}^\perp u_\eps ) \rightarrow 0
  \end{aligned}
\end{equation}
weakly-$\star$ in $t \geq 0$ and strongly in $H^{N-2}_x$ locally as $\eps \rightarrow 0$. Moreover, the convergences \eqref{Convgc-g-phi}, \eqref{Convgn-rho-u-theta} and the Boussinesq relation \eqref{Boussinesq-Reltn} tell us that for almost all $t \geq 0$,
\begin{equation}
  \begin{aligned}
    \mathcal{P} \big[ \rho_\eps \nabla_x ( \gamma \phi_\eps - \rho_\eps - \theta_\eps ) \big] (t) \rightarrow \mathcal{P} \big[ \rho \nabla_x ( \gamma \phi - \rho - \theta ) \big] (t) = 0
  \end{aligned}
\end{equation}
strongly in $H^{N-1}_x$ locally as $\eps \rightarrow 0$. Finally, from the definition of $R_{\eps, \Xi}$ for $\Xi = A$ or $B$ in \eqref{R-eps-Xi} and the convergences \eqref{Unf-Bnd-1}, \eqref{Unf-Bnd-2}, \eqref{Unf-Bnd-3}, \eqref{Convgnc-Pu}, \eqref{Convgnc-Pu-perp}, we can easily deduce that
\begin{equation}\label{Div-R-eps-Xi-to-0}
  \begin{aligned}
    \nabla_x \cdot R_{\eps, \Xi} \rightarrow 0
  \end{aligned}
\end{equation}
in the sense of distribution as $\eps \rightarrow 0$. Consequently, we justify the convergence \eqref{Converge-Pu-4}. It thereby follows from the convergences \eqref{Converge-Pu-1}, \eqref{Converge-Pu-2}, \eqref{Converge-Pu-3}, \eqref{Converge-Pu-4} and the incompressibility \eqref{Incompressibility} that $u \in L^\infty (\R^+; H^N_x) \cap C (\R^+; H^{N-1}_x)$ subjects to the evolution equation
\begin{equation}
  \begin{aligned}
    \partial_t u + \mathcal{P} \nabla_x \cdot (u \otimes u) - \mu \Delta_x u = \mathcal{P} ( \rho \nabla_x \theta ) \,, \\
    \nabla_x \cdot u = 0 \,,
  \end{aligned}
\end{equation}
with initial data
\begin{equation}
  \begin{aligned}
    u (0, x) = \mathcal{P} u_0 (x) \,.
  \end{aligned}
\end{equation}

We next take the limit from \eqref{Evolution-theta-rho} to the third $(\tfrac{3}{5} \theta - \tfrac{2}{5} \rho)$-equation in \eqref{NSFP} as $\eps \rightarrow 0$. For any $T > 0$, let $\xi (t,x)$ be a test function satisfying $\xi (t,x) \in C^1 (0,T; C_c^\infty (\R^3))$ with $\xi (0,x) = \xi_0 (x) \in C_c^\infty (\R^3)$ and $\xi (t,x) = 0$ for $t \geq T'$, where $T' < T$. From the initial conditions in Theorem \ref{Thm-Limit} and the strong convergence \eqref{Convgnc-rho-theta}, we deduce that
\begin{equation}\label{Converge-theta-1}
  \begin{aligned}
    & \int_0^T \int_{\R^3} \partial_t ( \tfrac{3}{5} \theta_\eps - \tfrac{2}{5} \rho_\eps ) (t,x) \xi (t,x) \d x \d t \\
    = & - \int_{\R^3} \l g_{\eps , 0} (x,v) , \tfrac{|v|^2}{5} - 1 \r_{L^2_v} \xi_0 (x) \d x - \int_0^T \int_{\R^3} ( \tfrac{3}{5} \theta_\eps - \tfrac{2}{5} \rho_\eps ) (t,x) \partial_t \xi (t,x) \d x \d t \\
    \rightarrow & - \int_{\R^3} \l g_0 (x,v) , \tfrac{|v|^2}{5} - 1 \r_{L^2_v} \xi_0 (x) \d x - \int_0^T \int_{\R^3} ( \tfrac{3}{5} \theta - \tfrac{2}{5} \rho ) (t,x) \partial_t \xi (t,x) \d x \d t \\
    = & - \int_{\R^3} ( \tfrac{3}{5} \theta_0 - \tfrac{2}{5} \rho_0 ) (x) \xi_0 (x) \d x - \int_0^T \int_{\R^3} ( \tfrac{3}{5} \theta - \tfrac{2}{5} \rho ) (t,x) \partial_t \xi (t,x) \d x \d t
  \end{aligned}
\end{equation}
as $\eps \rightarrow 0$. It follows from the strong convergences \eqref{Convgnc-Pu} and \eqref{Convgnc-rho-theta} that
\begin{equation}\label{Converge-theta-2}
  \begin{aligned}
    & \nabla_x \cdot [ \mathcal{P} u_\eps ( \tfrac{3}{5} \theta_\eps - \tfrac{2}{5} \rho_\eps ) ] \rightarrow \nabla_x \cdot [ u ( \tfrac{3}{5} \theta - \tfrac{2}{5} \rho ) ] \quad \textrm{strongly in } C(\R^+; H^{N-2}_x) \,, \\
    & \kappa \Delta_x ( \tfrac{3}{5} \theta_\eps - \tfrac{2}{5} \rho_\eps ) \rightarrow \kappa \Delta_x ( \tfrac{3}{5} \theta - \tfrac{2}{5} \rho ) \quad \textrm{strongly in } C(\R^+; H^{N-3}_x)
  \end{aligned}
\end{equation}
as $\eps \rightarrow 0$. It remains to prove
\begin{equation}\label{Converge-theta-3}
  \begin{aligned}
    R_{\eps, \theta} \rightarrow 0
  \end{aligned}
\end{equation}
in the sense of distribution as $\eps \rightarrow 0$, where $R_{\eps, \theta}$ is defined in \eqref{R-eps-theta}. Indeed, form the convergences \eqref{Convgc-g-phi}, \eqref{Convgn-rho-u-theta} and \eqref{Convgnc-Pu-perp}, one easily derives that for almost all $t \geq 0$
\begin{equation}\label{Limit-theta-1}
  \begin{aligned}
    & \tfrac{2}{5} \gamma \mathcal{P}^\perp u_\eps (t) \cdot \nabla_x \phi_\eps (t) \rightarrow 0 \quad \textrm{strongly in } H^{N-1}_x \,, \\
    & \nabla_x \cdot [ \mathcal{P}^\perp u_\eps (t) \theta_\eps (t) ] \rightarrow 0 \quad \textrm{strongly in } H^{N-2}_x
  \end{aligned}
\end{equation}
as $\eps \rightarrow 0$. Moreover, the convergence \eqref{Div-R-eps-Xi-to-0} tells us
\begin{equation}\label{Limit-theta-2}
  \begin{aligned}
    \tfrac{2}{5} \nabla_x \cdot R_{\eps, B} \rightarrow 0
  \end{aligned}
\end{equation}
in the sense of distribution as $\eps \rightarrow 0$. Furthermore, it follows from the convergences \eqref{Convgc-g-phi}, \eqref{Convgn-rho-u-theta}, \eqref{Convgnc-Pu} and the Boussinesq relation \eqref{Boussinesq-Reltn} that
\begin{equation}\label{Limit-theta-3}
  \begin{aligned}
    \tfrac{2}{5} \mathcal{P} u_\eps \cdot \nabla_x ( \rho_\eps + \theta_\eps - \gamma \phi_\eps ) \rightarrow \tfrac{2}{5} u \cdot \nabla_x ( \rho + \theta - \gamma \phi ) = 0
  \end{aligned}
\end{equation}
weakly-$\star$ for $t \geq 0$, weakly in $H^{N-1}_x$ and strongly in $H^{N-2}_x$ locally as $\eps \rightarrow 0$. Consequently, the limits \eqref{Limit-theta-1}, \eqref{Limit-theta-2} and \eqref{Limit-theta-3} imply the convergence \eqref{Converge-theta-3}. It is thereby yielded by collecting the limits \eqref{Converge-theta-1}, \eqref{Converge-theta-2} and \eqref{Converge-theta-3} that
\begin{equation}
  \begin{aligned}
    \partial_t ( \tfrac{3}{5} \theta - \tfrac{2}{5} \rho ) + \nabla_x \cdot \big[ u ( \tfrac{3}{5} \theta - \tfrac{2}{5} \rho ) \big] = \kappa \Delta_x \theta
  \end{aligned}
\end{equation}
with the initial data
\begin{equation}
  \begin{aligned}
    ( \tfrac{3}{5} \theta - \tfrac{2}{5} \rho ) (0,x) = ( \tfrac{3}{5} \theta_0 - \tfrac{2}{5} \rho_0 ) (x) \,.
  \end{aligned}
\end{equation}

Finally, it follows from the convergence \eqref{Convgc-g-phi}that
\begin{equation}\label{Limit-Delta-phi}
  \begin{aligned}
    \Delta_x \phi_\eps \rightarrow \Delta_x \phi
  \end{aligned}
\end{equation}
weakly-$\star$ for $t \geq 0$ and weakly in $H^{N-1}_x$ as $\eps \rightarrow 0$. Then, from the convergence \eqref{Convgn-rho-u-theta}, \eqref{Limit-Delta-phi} and Poisson equation $\Delta_x \phi_\eps = \gamma \rho_\eps$ in \eqref{Consevtn-Law-rho-u-theta-phi}, we deduce that
\begin{equation}\no
  \begin{aligned}
    \Delta_x \phi = \gamma \rho \,.
  \end{aligned}
\end{equation}
Recalling the Boussinesq relation \eqref{Boussinesq-Reltn}, we have $\Delta_x (\rho + \theta) = \gamma \Delta_x \phi$, which immediately implies that
\begin{equation}
  \begin{aligned}
    \Delta_x (\rho + \theta) = \gamma^2 \rho \,.
  \end{aligned}
\end{equation}
Collecting the all above convergence results, we have shown that $( \rho, u , \theta ) \in L^\infty (\R^+ ; H^N_x)$ with $ ( u, \tfrac{3}{5} \theta - \tfrac{2}{5} \rho ) \in C(\R^+; H^{N-1}_x) $ obey the following incompressible NSFP equations
\begin{equation*}
  \left\{
    \begin{array}{l}
       \partial_t u + u \cdot \nabla_x u + \nabla_x p = \mu \Delta_x u + \rho \nabla_x \theta \,, \\
       \nabla_x \cdot u = 0 \,, \\
       \partial_t ( \tfrac{3}{5} \theta - \tfrac{2}{5} \rho ) + u \cdot \nabla_x ( \tfrac{3}{5} \theta - \tfrac{2}{5} \rho ) = \kappa \Delta_x \theta \,, \\
       \Delta_x (\rho + \theta) = \gamma^2 \rho \,,
    \end{array}
  \right.
\end{equation*}
with initial data
\begin{equation}\no
  \begin{aligned}
    u (0,x) = u_0 (x) \,, \quad ( \tfrac{3}{5} \theta - \tfrac{2}{5} \rho ) (0,x) = \tfrac{3}{5} \theta_0 (x) - \tfrac{2}{5} \rho_0 (x) \,.
  \end{aligned}
\end{equation}
Consequently, the proof of Theorem \ref{Thm-Limit} is finished. \qquad \qquad \qquad \qquad \qquad \qquad\qquad \qquad \qquad$\square$

%%%%%%%%%%%%%%%%%%%%%%%%%%%%%%%%%%%%%%%%%%%%%%%%%%%%%%%%%%%%%%%%%%%%%%%%%%%

\section*{Acknowledgment}

This work was supported by the grants from the National Natural Science Foundation of
China under contract No. 11471181 and No. 11731008.

%%%%%%%%%%%%%%%%%%%%%%%%%%%%%%%%%%%%%%%%%%%%%%%%%%%%%%%%%%%%%%%%%%%%%%%%%%%
\bigskip
% \phantomsection
% \addcontentsline{toc}{section}{\refname}

%\bibliographystyle{unsrtnat}
%\nocite{*}
\bibliography{reference}

\begin{thebibliography}{99}

\bibitem{Adams-Fournier-2003-Sobolev} R. A. Adams and J. F. Fournier. {\em Sobolev spaces. Second edition.} Pure and Applied Mathematics (Amsterdam), 140. Elsevier/Academic Press, Amsterdam, 2003. xiv+305 pp. ISBN: 0-12-044143-8.

\bibitem{amuxy3} R. Alexandre, Y. Morimoto, S.  Ukai, C.-J.   Xu
and T. Yang, {\it
	Global existence and full regularity of the Boltzmann equation without angular cutoff}.
Comm. Math. Phys. {\bf 304} (2011) 513-581.

\bibitem{amuxy4-3} R. Alexandre, Y. Morimoto, S. Ukai,,  C.-J. Xu
and T. Yang, {\it The Boltzmann equation without angular cutoff in the whole
	space: II.  Global existence for hard potentia}l.
Anal. Appl.(Singap.) {\bf 9} (2011), 113-134.

\bibitem{amuxy4-2} R. Alexandre, Y. Morimoto, S. Ukai,,  C.-J. Xu
and T. Yang, {\it The Boltzmann equation without angular cutoff in the whole
	space: I. Global existence for soft potential}. J. Funct. Anal. {\bf 262}
(2012), 915-1010.


\bibitem{amuxy7}
R. Alexandre, Y. Morimoto, S. Ukai,,  C.-J. Xu and T. Yang,
{ \it Local existence with mild regularity for the Boltzmann equation}.
Kinet. Relat. Models  {\bf 6} (2013), no. 4, 1011-1041.

\bibitem{al-3}
R. Alexandre and C. Villani,  {\it  On the Boltzmann equation for
	long-range interaction}.  Commun. Pure and Appl. Math. {\bf 55} (2002),  30-70.

\bibitem{Arsenio} D. Ars\'enio,
{\it From Boltzmann's Equation to the incompressible Navier-Stokes-Fourier system with long-range interactions}. { Arch. Ration. Mech. Anal.} {\bf 206} (2012), no. 3, 367-488

\bibitem{Arsenio-SaintRaymond-2016} D. Ars\'enio and L. Saint-Raymond, From the Vlasov-Maxwell-Boltzmann system to incompressible viscous electro-magneto-hydrodynamics, arXiv:1604.01547[math.AP], 2016.

\bibitem{BGL1} C. Bardos, F. Golse and C. D. Levermore, Fluid dynamic limits of kinetic equations I: formal derivation. {\em J. Stat. Phys.}, {\bf 63} (1991), 323-344.

\bibitem {BGL2} C. Bardos, F. Golse, and C. D. Levermore,
{\em Fluid dynamic limits of kinetic equations II: convergence proof
	for the Boltzmann equation.} Commun. Pure and Appl. Math. {\bf 46}
(1993), 667-753

\bibitem {BGL3} C. Bardos, F. Golse, and C. D. Levermore,
{\em The acoustic limit for the Boltzmann equation.} { Arch. Ration.
	Mech. Anal.} {\bf 153} (2000), no. 3, 177-204

\bibitem{b-u}
C. Bardos and S. Ukai, {\it The classical incompressible Navier-Stokes limit of the Boltzmann equation.} Math. Models Methods Appl. Sci. {\bf 1} (1991), no.2, 235-257.

\bibitem{Boyer-Fabrie-2013-BOOK} F. Boyer and P. Fabrie. {\em Mathematical tools for the study of the incompressible Navier-Stokes equations and related models}. Applied Mathematical Sciences, {\bf 183}, Springer, New York, 2013.

\bibitem{Briant} M. Briant,
From the Boltzmann equation to the incompressible Navier-Stokes equations on the torus: a quantitative error estimate. {\em J. Differential Equations} {\bf 259} (2015), no. 11, 6072-6141.

\bibitem{Briant2017} M. Briant,
Perturbative theory for the Boltzmann equation in bounded domains with different boundary conditions. {\em Kinet. Relat. Models} {\bf 10} (2017), no. 2, 329-371.

\bibitem{Briant-Guo} M. Briant and Y. Guo,
Asymptotic stability of the Boltzmann equation with Maxwell boundary conditions. {\em J. Differential Equations} {\bf 261} (2016), no. 12, 7000-7079.

\bibitem{BMMouhot} M. Briant, S. Merino-Aceituno, C. Mouhot,
From Boltzmann to incompressible Navier-Stokes in Sobolev spaces with polynomial weight. arXiv:1412.4653


\bibitem{Caflisch}
R. Caflisch,
{\it The fluid dynamic limit of the nonlinear Boltzmann equation.} Comm. Pure Appl. Math. {\bf 33} (1980), no. 5, 651-666.

\bibitem{Cer} C. Cercignani, {\em The Boltzmann equation and its applications.}
Springer, New York, 1988.

\bibitem{CIP} C. Cercignani, R. Illner and M. Pulvirenti. {\em the mathematical theory of dilute
	gases.} Springer, New York, 1994.

\bibitem{DEL-89}
A. De Masi, R. Esposito, and J. L. Lebowitz,
{\it Incompressible Navier-Stokes and Euler limits of the Boltzmann equation.}
Comm. Pure Appl. Math. {\bf 42} (1989), no. 8, 1189-1214.

\bibitem{D-L} R. J. DiPerna, P.L.  Lions, {\it On the Cauchy problem for Boltzmann
	equations: global existence and weak stability}.  Ann. Math.
{\bf 130}(1989), 321-366.

\bibitem{D-Y-Z(hard)-2011} Duan R.-J., Yang, T., and Zhao H.-J., 
The Vlasov-Poisson-Boltzmann system in the whole space: The hard potential case. \textit{J. Differential Equations} \textbf{252} (2012), no. 12, 6356--6386.


\bibitem{D-Y-Z(soft)-2011} Duan R.-J., Yang T., and Zhao H.-J., 
The Vlasov-Poisson-Boltzmann system for soft potentials. \textit{Mathematical Models and Methods in Applied Sciences} \textbf{23} (2013)  no. 06, 979--1028.

\bibitem{Glassey} R. Glassey,
{\em The Cauchy problem in kinetic theory.} Society for Industrial and Applied Mathematics (SIAM), Philadelphia, PA, 1996.

\bibitem{GL} F. Golse and C. D. Levermore,
{\em The Stokes-Fourier and acoustic limits for the Boltzmann
	equation.} Comm. on Pure and Appl. Math. {\bf 55} (2002), 336-393.

\bibitem{GL-05} F. Golse and C. D. Levermore,
Hydrodynamic limits of kinetic models. {\em Topics in kinetic theory}, 1-75, Fields Inst. Commun., 46, Amer. Math. Soc., Providence, RI, 2005.

\bibitem{Go-Sai04} F. Golse and L. Saint-Raymond,
{\em The Navier-Stokes limit of the Boltzmann equation for bounded
	collision kernels.} Invent. Math. {\bf 155} (2004), no. 1, 81--161.

\bibitem{G-S05} F. Golse and L. Saint-Raymond,
Hydrodynamic limits for the Boltzmann equation. {\em Riv. Mat. Univ. Parma} (7) {\bf 4}** (2005), 1-144.

\bibitem{Go-Sai09} F. Golse and L. Saint-Raymond,
{\em The incompressible Navier-Stokes limit of the Boltzmann
	equation for hard cutoff potentials.}  J. Math. Pures Appl. (9) {\bf
	91} (2009), no. 5, 508--552.

\bibitem{G-S} P. Gressman and R. Strain,
{\it Global classical solutions of the Boltzmann equation without angular cut-off.}
{ J. Amer. Math. Soc.} {\bf 24} (2011), no. 3, 771-847

\bibitem{GMM} M. P. Gualdani, S. Mischler and C. Mouhot,
Factorization for non-symmetric operators and exponential H-theorem. arXiv:1006.5523v3 To appear on Memoire de la Societe Mathematique de France.

\bibitem{Guo-2002-VPB} Y. Guo, The Vlasov-Poisson-Boltzmann system near Maxwellians, {\em Comm. Pure Appl. Math.}, {\bf 55} (2002), pp. 1104-1135.

\bibitem{guo-T} Y. Guo,
{\it Classical solution to the Boltzmann Equation for molecules with an angular cutoff.}  Arch. Rational Mech. Anal. {\bf 169} (2003) 305-353.

\bibitem{guo-1} Y. Guo,
{\it The Boltzmann equation in the whole space.} Indiana Univ. Math. J. {\bf 53-4} (2004) 1081-1094.

\bibitem{GY-06} Y. Guo. Boltzmann diffusive limit beyond the Navier-Stokes approximation. {\em Comm. Pure Appl. Math.}, {\bf 59} (2006), no. 5, 626-687.

\bibitem{GY-2010} Y. Guo,
Decay and continuity of the Boltzmann equation in bounded domains. {\em Arch. Ration. Mech. Anal.} {\bf 197} (2010), no. 3, 713-809.

\bibitem{GJJ-KRM2009} Y. Guo, J. Jang and N. Jiang,
Local Hilbert expansion for the Boltzmann equation. {\em Kinet. Relat. Models} {\bf 2} (2009), no. 1, 205-214.

\bibitem{GJJ-CPAM2010} Y. Guo, J. Jang and N. Jiang,
Acoustic limit for the Boltzmann equation in optimal scaling. {\em Comm. Pure Appl. Math.} {\bf 63} (2010), no. 3, 337-361.

\bibitem{GKTT-2016} Y. Guo, C. Kim, D. Tonon, and A. Trescases,
BV-regularity of the Boltzmann equation in non-convex domains. {\em Arch. Ration. Mech. Anal.} {\bf 220} (2016), no. 3, 1045-1093.

\bibitem{GKTT-2017} Y. Guo, C. Kim, D. Tonon, and A. Trescases,
Regularity of the Boltzmann equation in convex domains. {\em Invent. Math.} {\bf 207} (2017), no. 1, 115-290.

\bibitem{JJ-DCDS2009} J. Jang and N. Jiang,
Acoustic limit of the Boltzmann equation: classical solutions. {\em Discrete Contin. Dyn. Syst.} {\bf 25} (2009), no. 3, 869-882.

\bibitem{JLM-CPDE2010} N. Jiang, C. D. Levermore and N. Masmoudi,
Remarks on the acoustic limit for the Boltzmann equation. {\em Comm. Partial Differential Equations} {\bf 35} (2010), no. 9, 1590-1609.

\bibitem{Jiang-Luo-2019-VMB} N. Jiang and Y.-L. Luo. From Vlasov-Maxwell-Boltzmann system to two-fluid incompressible Navier-Stokes-Fourier-Maxwell system with Ohm's law: convergence for classical solutions. {arXiv:1905.04739[math.AP]}, 2019. %\href{https://arxiv.org/pdf/1905.04739.pdf}

\bibitem{JM-CPAM2017} N. Jiang and N. Masmoudi,
Boundary layers and incompressible Navier-Stokes-Fourier limit of the Boltzmann equation in bounded domain I. {\em Comm. Pure Appl. Math.} {\bf 70} (2017), no. 1, 90-171.

\bibitem{JX-SIMA2015}N. Jiang and L. J. Xiong,
Diffusive limit of the Boltzmann equation with fluid initial layer in the periodic domain. {\em SIAM J. Math. Anal.} {\bf 47} (2015), no. 3, 1747-1777.

\bibitem{Jiang-Xu-Zhao-2018-Indiana} N. Jiang, C-J, Xu and H. Zhao. Incompressible Navier-Stokes-Fourier limit from the Boltzmann equation: classical solutions. {\em Indiana University Mathematical Journal}. {\bf 67} (2018), no. 5, 1817-1855.

\bibitem{KMN} S. Kawashima, A. Matsumura and T. Nishida,
{\em On the fluid dynamical approximation to the Boltzmann equation at the level of the Navier-Stokes equation,} Comm. Math. Phys., {\bf 70} (1979), 97-124.

\bibitem{LM} C. D. Levermore and N. Masmoudi,
{\em From the Boltzmann equation to an incompressible Navier-Stokes-Fourier system.}
Arch. Ration. Mech. Anal. {\bf 196} (2010), no. 3, 753-809.

\bibitem{Levermore-Sun-2010-KRM} C. D. Levermore and W. Sun, Compactness of the gain parts of the linearized Boltzmann operator with weakly cutoff kernels. {\em Kinet. Relat. Models}, {\bf 3} (2010), no. 2, 335-351.

\bibitem{Lions1994-1} P.-L. Lions,  
Compactness in Boltzmann's equation via Fourier integral operators and applications. I, II. {\em J. Math. Kyoto Univ.} {\bf 34} (1994), no. 2, 391-427, 429-461.

\bibitem{LM3} P.-L. Lions and N. Masmoudi,
{\em From Boltzmann equation to Navier-Stokes and Euler equations
	I.} Arch. Ration. Mech. Anal. {\bf 158} (2001), 173-193.

\bibitem{LM4} P.-L. Lions and N. Masmoudi,
{\em From Boltzmann equation to Navier-Stokes and Euler equations
	II.} Arch. Ration. Mech. Anal. {\bf 158} (2001), 195-211.

\bibitem{liu-2}
T.-P. Liu, T. Yang, T., and S.-H. Yu, {\it
	Energy method for Boltzmann equation},
{ Phys. D } {\bf 188}  (2004), 178-192.

\bibitem{MSRM-CPAM2003} N. Masmoudi and L. Saint-Raymond,
From the Boltzmann equation to the Stokes-Fourier system in a bounded domain. {\em Comm. Pure Appl. Math.} {\bf 56} (2003), no. 9, 1263-1293.

\bibitem{mischler2010asens}	S.~Mischler,
\newblock Kinetic equations with {M}axwell boundary conditions.
\newblock \emph{Ann. Sci. \'Ec. Norm. Sup\'er. (4)} \textbf{43} (2010), 719--760.

\bibitem{Nirenberg-1959-ASNSP} L. Nirenberg, On elliptic partial differential equations. {\em Ann. Scuola Norm. Sup. Pisa (3)}, {\bf 13} (1959), 115-162.

\bibitem{Nishida} T. Nishida,
{\em Fluid dynamical limit of the nonlinear Boltzmann equation to the level of the compressible Euler equation,} Comm. Math. Phys., {\bf 61} (1978), 119-148.

\bibitem{Saint-Raymond-2009-Boltzmann} L. Saint-Raymond. {\em Hydrodynamic limits of the Boltzmann equations}, volume 1971 of {\em Lecture Notes in Mathematics.} Springer-Verlag, Berlin, 2009.

\bibitem{SRM2010} L. Saint-Raymond,
Some recent results about the sixth problem of Hilbert: hydrodynamic limits of the Boltzmann equation.{\em  European Congress of Mathematics}, 419-439, Eur. Math. Soc., Zurich, 2010.

\bibitem{Simon-1987-AMPA} J. Simon. Compact sets in the space $L^p (0, T; B)$. {\em Ann. Mat. Pura Appl.}, {\bf 146} (1987), no. 4, 65-96.

\bibitem{Ukai-1974} S. Ukai,
On the existence of global solutions of mixed problem for non-linear Boltzmann equation. {\em Proc. Japan Acad.} {\bf 50} (1974), 179-184.

\bibitem{Ukai-1986} S. Ukai,
Solutions of the Boltzmann equation. {\em Patterns and waves}, 37-96,
Stud. Math. Appl., 18, North-Holland, Amsterdam, 1986.

\bibitem{Ukai-2006} S. Ukai,
Asymptotic analysis of fluid equations. {\em Mathematical foundation of turbulent viscous flows}, 189-250,
Lecture Notes in Math., {\bf 1871}, Springer, Berlin, 2006.

\bibitem{XXZhard}  Xiao Q.-H., Xiong L.-J., and Zhao H.-J., 
The Vlasov-Posson-Boltzmann system without angular cutoff for hard potential. {\it Science China Mathematics}, \textbf{57} (2014), no. 3, 515-540.

 \bibitem{XXZsoft}  Xiao Q.-H., Xiong L.-J., and Zhao H.-J., 
 The Vlasov-Poisson-Boltzmann system for the whole range of cutoff soft potentials.{\it J. Funct. Anal.} \textbf{272} (2017), no. 1, 166-226.

\end{thebibliography}

\end{document}